\documentclass[a4paper,10pt]{amsart}

\usepackage[latin1]{inputenc}
\usepackage{amsmath, amsthm, amssymb, comment}
\usepackage[dvips]{graphicx}
\usepackage{overpic}
\usepackage{enumitem}
\usepackage[colorlinks=true, citecolor=red]{hyperref}
\usepackage{pinlabel}

\usepackage{thm-restate}

\usepackage[margin=3.5cm]{geometry}

\usepackage{color}

\newtheorem{theorem}{Theorem}[section]
\newtheorem{corollary}[theorem]{Corollary}
\newtheorem{proposition}[theorem]{Proposition}
\newtheorem{definition}[theorem]{Definition}
\newtheorem{lemma}[theorem]{Lemma}
\newtheorem{claim}[theorem]{Claim}
\newtheorem{construction}[theorem]{Construction}

\newtheorem*{theorem*}{Theorem}
\newtheorem*{proposition*}{Proposition}
\newtheorem*{definition*}{Definition}
\newtheorem*{lemma*}{Lemma}
\newtheorem*{claim*}{Claim}
\newtheorem*{corollary*}{Corollary}
\newtheorem*{convention*}{Convention}
\newtheorem{observation}[theorem]{Observation}

\theoremstyle{definition}

\newtheorem{example}[theorem]{Example}
\newtheorem{question}{Question}
\newtheorem*{notation}{Notation}

\theoremstyle{remark}
\newtheorem{rem}[theorem]{Remark}
\newtheorem*{rem*}{Remark}
\newtheorem*{acknowledgement}{Acknowledgement}

\newcommand{\wt}[1]{\widetilde{#1}}

\newcommand\bR{\mathbb R}

\newcommand\bZ{\mathbb Z}

\newcommand{\R}{\mathbb R}

\DeclareMathOperator{\id}{id}

\newcommand\orb{ \mathcal O }

\newcommand{\cC}{\mathcal{C}}

\newcommand{\cF}{\mathcal{F}}
\newcommand{\cL}{\mathcal{L}}

\newcommand{\cR}{\mathcal{R}}

\newcommand{\cW}{\mathcal{W}}

\newcommand{\Pbound}{\partial P}

\newcommand{\Fixbar}{\overline{\mathrm{Fix}}}
\newcommand{\Fix}{\mathrm{Fix}}

\newcommand{\lG}{\leq_G}

\newcounter{notes}%

\title{Nontransitive pseudo-Anosov flows}

\author[Thomas Barthelm\'e]{Thomas Barthelm\'e}
\address{Queen's University, Kingston, Ontario}
\email{thomas.barthelme@queensu.ca}
\urladdr{sites.google.com/site/thomasbarthelme}

\author[Christian Bonatti]{Christian Bonatti}
 \address{Universit\'e de Bourgogne, Dijon, France}
 \email{bonatti@u-bourgogne.fr}

\author[Kathryn Mann]{Kathryn Mann}
 \address{Cornell University, Ithaca, NY}
 \email{k.mann@cornell.edu}
\urladdr{https://e.math.cornell.edu/people/mann}

\begin{document}

 \begin{abstract}
This paper describes a structure theory for (topological) pseudo-Anosov flows from the perspective of the associated group actions on their orbit spaces and boundary at infinity.   We extend the definition of Anosov-like action from \cite{BFM} from the transitive to the general nontransitive context and show that one can recover many dynamical features such as the basic sets of a flow, and the Smale order on basic sets, from such general group actions.  

Using these tools, we prove that a pseudo-Anosov flow on a compact $3$ manifold is determined by the associated action of the fundamental group on the boundary at infinity of its orbit space, answering an open question.  We also give a number of other applications, including a proof that any topological pseudo-Anosov flow on an atoroidal 3-manifold is necessarily transitive, and that density of periodic orbits implies transitivity.  These latter results fill some gaps in the literature, where proofs have only been written in the smooth Anosov case.  
 \end{abstract}

 \maketitle

 \section{Introduction}
This paper is motivated by the structure theory and classification of  pseudo-Anosov flows in dimension 3.   In \cite{Mosher_homologynormI}, Mosher introduced {\em pseudo-Anosov flows} as flows ``locally modeled on the suspension of a pseudo-Anosov homeomorphism of a surface."
 Such flows are characterized by the presence of a pair of transverse 2-dimensional invariant foliations, generalizing the weak-stable and unstable foliations of an Anosov flow, which Mosher suggested should be considered their defining feature.  Since then, pseudo-Anosov flows have become increasingly recognized as important objects in hyperbolic dynamics and low-dimensional topology, and Mosher's definition has been formalized in a variety of ways, resulting in definitions of ``smooth pseudo-Anosov" or ``topological (pseudo) Anosov" flow.   However, there has not been widespread agreement on the definition, and the foundations remain incomplete.\footnote{See \cite[Remark 1.1.11.]{BM_book} for a discussion, and Example 1.2.15 of \cite{BM_book} for a cautionary example illustrating a critical gap in some existing definitions.  Remark \ref{rem_smooth_vs_topological} below also gives further references.}
  
In this paper we develop a formal, axiomatic approach to the topological study of pseudo-Anosov flows, using the framework of {\em Anosov-like actions} of discrete groups on the plane.  This perspective was introduced in \cite{BFM} to give an invariant towards the classification of transitive pseudo-Anosov flows.  The present paper develops further groundwork, applicable also to the nontransitive case.  Using this we prove a number of structural results.  For instance, we show that any (topological) pseudo-Anosov flow or topological Anosov flow on an atoroidal 3-manifold is necessarily transitive.  This fills an apparent gap in the literature, since the existing proofs for smooth Anosov or pseudo-Anosov flows (due to Brunella \cite{Brunella} and Mosher \cite{Mosher_homologynormI}) use Smale's spectral decomposition theorem, which is only fully proved in the smooth setting.  We also characterize transitive versus nontransitive flows (see Theorem \ref{thm_dense_orbits_is_transitive}), and give many examples illustrating the nuances that appear in the presence of prongs in the pseudo-Anosov setting.  

Finally, the main result of this article shows pseudo-Anosov flows are completely determined up to orbit equivalence (time-change and conjugacy) by an action of the fundamental group of the supporting manifold on $S^1$.  This improves a theorem of Barbot, who showed such flows are determined by an action on $\bR^2$, and fully settles what was an important question about the natural action on $S^1$ induced by pseudo-Anosov flow, which had been circulating since Fenley's construction of the ideal $S^1$ boundary in 2005. 

\subsection*{The orbit-space perspective} 
A pseudo-Anosov flow $\phi$ on a 3-manifold $M$ gives rise to an {\em orbit space} $\mathcal{O}_\phi$, a topological plane with two transverse (possibly singular) foliations which admits a natural action of $\pi_1(M)$ by homeomorphisms.  Following work of Barbot \cite{Bar_caracterisation}, this action of a group of the plane (up to conjugacy) uniquely determines the flow up to orbit equivalence.  Barbot's theorem indicates that all features of a pseudo-Anosov flow and its orbit equivalence class should, in theory, correspond to features of the orbit space with the action of $\pi_1(M)$.  Building such a ``dictionary" between flows and their orbit spaces is now regarded as an essential tool in the subject.  Early work of Barbot and Fenley used this idea towards a structure theory of pseudo-Anosov flows on graph manifolds \cite{Barbot96,BarbFen_pA_toroidal,BF_totally_per} among other results; more recently in \cite{BFM, BFeM} it was used to obtain a complete algebraic invariant of transitive pseudo-Anosov flows (on any given manifold $M$) via the dynamics of $\pi_1(M)$ on $\mathcal{O}_\phi$.

To develop a formal framework for pseudo-Anosov flows, we abstract the essential dynamical features of actions on orbit spaces, called an {\em Anosov-like action}.   This allows one to use 2-dimensional topological dynamics to prove results about discrete groups acting on the plane, then translate this back to the 3-manifold setting.   

\subsection*{Statement of results}  A {\em bifoliated plane} is a topological plane $P$ with two transverse topological foliations $\cF^\pm$, possibly with isolated prong singularities.  The following definition captures the dynamics of the orbit-space actions of flows. 
\begin{definition}[Anosov-like action] \label{def_action}
An action of a group $G$ on a bifoliated plane, preserving each foliation, is called \emph{Anosov-like} if it satisfies the following: 
	\begin{enumerate}[label = (A\arabic*)]
	\item\label{Axiom_A1} 
	If a nontrivial element of $G$ fixes a leaf $l \in \cF^\pm$, then it has a fixed point $x \in l$ and is topologically expanding on one leaf through $x$ and topologically contracting on the other. 
		\item\label{Axiom_dense} The union of leaves of $\cF^{+}$ that are fixed by some element of $G$ is dense in $P$, as is the union of leaves of $\cF^{-}$ that are fixed by some element of $G$.  
		\item \label{Axiom_prongs_are_fixed} Each singular point is fixed by some nontrivial element of $G$.
		\item \label{Axiom_non-separated} If $l$ is a leaf of $\cF^+$ or $\cF^-$ that is non-separated with some leaf $l'$ in the corresponding leaf space, then some nontrivial element $g\in G$ fixes $l$.
\item \label{Axiom_totallyideal} There are no \emph{totally ideal quadrilaterals} in $P$ (see Definition \ref{def:totally_ideal}).
	\end{enumerate}
\end{definition}   
We note that, for the results of this article, Axiom \ref{Axiom_totallyideal} could be weakened slightly, requiring that we only forbid such quadrilaterals that do not contain any fixed points of elements of $G$ (see Remark \ref{rem_wandering_totally_ideal}).  This weak version prevents examples of groups which split as free products (see Theorem \ref{thm:wandering_quadrilateral}), a behavior very different from the fundamental groups of manifolds supporting flows which are always aspherical. 
Proposition \ref{prop:Orbit_space_anosov_like} below shows that the orbit space action from a pseudo-Anosov flow satisfies these axioms; however, these are not the only examples.   

These axioms have surprisingly strong consequences, both for the topological structure of the bifoliated plane, and the global and local dynamics of the action of $G$.   For example, we show (following Barbot) that a bifoliated plane admitting an Anosov-like action either has prongs or nonseparated leaves, is trivial, or has a special form called {\em skew} (Proposition \ref{prop_trichotomy}).  We also show a number of relationships between dynamics and group structure:

\begin{restatable}[Point stabilizers]{theorem}{poinstabilizers}\label{thm:discrete_stabilizer} 
If $P$ is a bifoliated plane with either prongs or nonseparated leaves and $G$ has an Anosov-like action, then 
 point stabilizers for the action are trivial or virtually (index at most 4) isomorphic to $\bZ$.
 \end{restatable}
 
\begin{theorem}[Atoroidal implies transitive]\label{thm:transitive_or_torus}
For any cocompact Anosov-like action, $G$ either contains a subgroup isomorphic to $\bZ^2$ or the action is topologically transitive. 
\end{theorem}
(See Theorem \ref{thm_cocompact_Z2} for a stronger statement.)  As a consequence of this, we have: 

\begin{corollary}\label{cor_atoroidal_transitive}
Any topological pseudo-Anosov flow on an algebraically atoroidal 3-manifold is transitive.
\end{corollary} 
The proof does not require any 3-manifold topology (beyond showing that the orbit space of a flow is Hausdorff), nor require any smoothness of the flow, so applies also to topological (pseudo)-Anosov flows in the sense of \cite[Def. 1.1.10]{BM_book} or \cite[Def. 5.9]{AgolTsang}.
We also show
\begin{theorem} \label{thm_dense_orbits_is_transitive}
If a pseudo-Anosov flow $\phi$ has a dense set of periodic orbits then it is transitive. 
\end{theorem}
Again, this is well known for smooth Anosov flows, but to our knowledge has not been proved in the topological and pseudo-Anosov setting.  

\begin{rem}\label{rem_single_fixed_point}
As a consequence of the axioms, if a leaf $l$ has nontrivial stabilizer, then all nontrivial elements of the stabilizer have a common, unique fixed point in $l$.  For singular leaves, this is a consequence of Axioms \ref{Axiom_A1} and \ref{Axiom_prongs_are_fixed}. 
In the nonsingular case, this follows from Axiom \ref{Axiom_A1} together with Solodov's Theorem (see \cite[Lemma 2.6]{BFM}).  Thus, Theorem \ref{thm:discrete_stabilizer} has an equivalent statement in terms of stabilizers of {\em leaves}.  
\end{rem}

An important tool in Theorem \ref{thm:transitive_or_torus} and \ref{thm_dense_orbits_is_transitive} as well as the general structure theory is the definition and description of {\em Smale classes}.  

\smallskip
\noindent  \textbf{Smale classes and chains.}
Brunella \cite{Brunella} showed that the {\em basic sets} of an Anosov flow are separated by a collection of tori transverse to the flow.   In the nontransitive case any basic set contains ``boundary" orbits, which are periodic orbits such that (at least) one of their rays are contained in the wandering set.  A related theory is developed for Smale diffeomorphisms of surfaces in \cite{BJL}. 

Here we generalize this theory to pseudo-Anosov flows, where the situation is significantly more subtle.
One issue is that, in this setting, non-wandering components can be strictly finer than chain-recurrent components (see Proposition \ref{prop:weird_loops} for an example).  This gives rise to unusual dynamical behavior not seen in the Anosov setting.  To address this, we introduce a structural theory of {\em Smale classes}, which generalize the notion of a {\em basic set} for an Anosov flow.  These inherit a natural partial order, which we denote by $\leq_G$ (see Definition \ref{def:smale_class}).  We show:
\begin{restatable}{theorem}{transitiveonclasses}\label{thm:transitive} 
If $G$ has an Anosov-like action on a bifoliated plane, then $G$ acts topologically transitively on each Smale class.  Moreover, an action has a unique Smale class if and only if the set of points fixed by group elements is dense in $P$. 
\end{restatable} 
This answers a problem from \cite[Remark 2.4]{BFM}, see Corollary \ref{cor:transitive}.

By \cite{Brunella}, when $M$ is orientable and compact, any two basic sets of an Anosov flow are separated by embedded tori in $M$ transverse to the flow.  
We describe the counterpart structure in the orbit space.  These are called {\em Smale chains}.  We show: 
\begin{theorem}\label{thm:smale_chain_vague}
Any Anosov-like action with at least two distinct Smale classes admits Smale chains; more precisely, if $x \leq_G y$ are in distinct Smale classes, then there is a Smale chain ``separating'' these classes. 
\end{theorem}
See Theorem \ref{thm:Smale_chains_separate} for a more precise statement, in which we clarify what is meant by separating.  
Together with Theorem \ref{thm_cocompact_Z2}, this shows that, provided an Anosov-like action is cocompact, distinct Smale classes can be separated by $\bZ^2$-invariant wandering chains. 
 
 Translating this to the language of pseudo-Anosov flows, we show that wandering chains separate distinct non-wandering components; whereas the (often coarser) chain-recurrent components are those separated by Lyapunov functions.

\subsection*{``Reduction of dimension":  classification by actions on $S^1$.}
A bifoliated plane comes equipped with a natural compactification by a boundary at infinity (see \cite{Bonatti_boundary,Fen_ideal_boundaries,Mather}) to which any foliation-preserving action by homeomorphisms extends.  This compactification has become an important tool for the study of group actions on such planes and, as a special case, the study of pseudo-Anosov flows.  It was shown in \cite{BFM} that for transitive Anosov-like actions, the induced action at infinity uniquely determines the action on the interior.  Here, using tools from \cite{prelaminations}, we show this is true generally, with an independent proof. 
\begin{restatable}{theorem}{boundarydeterminesaction}\label{thm:boundary_determines_action}
Let $G$ be a Smale-bounded Anosov-like action on a bifoliated plane $P$.  The induced action on $\partial P$ uniquely determines the action of $G$ up to conjugacy.  
\end{restatable}
 The definition of Smale-bounded is given in Definition \ref{def_smale_bounded}. For the purpose of this introduction, it is enough to know that every Anosov-like action with finitely many Smale classes is Smale-bounded (Proposition \ref{prop_finite_or_cocmpact_implies_Smale_bounded}). 

Consequently, pseudo-Anosov flows on compact 3-manifolds are determined up to orbit equivalence by the action of $\pi_1(M)$ on $\partial P$: 

\begin{theorem}[Circle actions classify flows] \label{thm_action_infinity}
The action of $\pi_1(M)$ on $S^1$, obtained from the compactification of the orbit space by a circle at infinity, (up to conjugacy) completely determines a pseudo-Anosov flow up to orbit equivalence.  
\end{theorem} 

The proof of this theorem has two ingredients: One is the work of the same authors done in \cite{prelaminations}, which allows to uniquely recover the bifoliated plane (and, by uniqueness, also the group action) from relatively ``sparse'' data on the boundary (Corollary E of \cite{prelaminations}, restated as Theorem \ref{thm_prelamination_to_foliation} below). The second ingredient, obtained in this article, consists in proving that when the foliations are not $\bR$-covered, the \emph{non-corner fixed leaves} are dense in $P$. The endpoints of such leaves are sparse -- it is a countable set -- but sufficiently rich so that we can apply the result of \cite{prelaminations}. A non-corner fixed leaf is a leaf $l$ of $\cF^\pm$, fixed by some nontrivial element $g\in G$, and such that the fixed point $x$ of $g$ on $l$ is not the corner of any lozenge (see Definition \ref{def_lozenge}). Such non-corner fixed points and fixed leaves can be ``seen"  by the action of $G$ on the circle at infinity: an element $g$ fixes a non-corner leaf if and only if the subgroup generated by $g$ has exactly $4$ fixed points on $\partial P$ such that two are attracting and two are repelling (see Propositions \ref{prop:boundary_action_general} and \ref{prop:four_fixed_points}). Thus, having two actions that are conjugated on the circle at infinity forces the non-corner fixed leaves for both actions to be ``the same'', and the result of \cite{prelaminations} applies. 

Thus, the technical work toward the proof of Theorem \ref{thm:boundary_determines_action} is to show the density of non-corner fixed leaves.  This is obtained in Corollary \ref{cor_simple_leaves_are_dense}.  In the case of topologically transitive Anosov-like actions, this is comparatively easy, and was obtained in \cite{BFM}. In the nontransitive case however, it relies on the description of Smale classes and Smale chains that we obtain here. 

There is also a shortcut one can apply in the case of smooth Anosov flows by using the classical Smale Decomposition Theorem and the results we prove in Section \ref{sec:density}.  More concretely, the classical Smale Decomposition Theorem can be used to show that the set of periodic orbits of a flow that are ``dense enough'' in an attractor will lift to non-corner fixed points, which is sufficient to proceed with the proof (see Remark \ref{rem_density_Anosov_case}).

\smallskip
\noindent \textbf{The axioms for transitive Anosov-like actions.}
Theorem \ref{thm:transitive} answers the question posed in \cite[Remark 2.4]{BFM}, and in fact shows the axioms given there for transitive Anosov-like actions are redundant and can be simplified.  To explain this, we recall that \cite{BFM} gave
a list of strictly stronger axioms than those in Definition \ref{def_action} for groups acting on bifoliated planes, capturing the behavior of \emph{transitive} pseudo-Anosov flows.  
 In the transitive case, Axiom \ref{Axiom_dense} does not appear, and in its place we have the following two axioms.
{\it 
\begin{enumerate}[label=(A2\alph*)] 
\item \label{Axiom_fixed_points_dense} The set of points that are fixed by some nontrivial element of $G$ is dense in $P$.
\item \label{Axiom_topologically_transitive} The action is topologically transitive, i.e., has a dense orbit.
\end{enumerate}}
An immediate consequence of Theorem \ref{thm:transitive} is the following.  

\begin{corollary} \label{cor:transitive}
 For transitive Anosov-like actions, Axiom \ref{Axiom_fixed_points_dense} implies Axiom \ref{Axiom_topologically_transitive}. Thus the only difference between a transitive Anosov-like action and a general Anosov-like action is that the former satisfies Axiom \ref{Axiom_fixed_points_dense}. 
\end{corollary}

We point out that we also weakened slightly Axiom \ref{Axiom_non-separated} compared to \cite{BFM}. We show in Proposition \ref{prop_axiom4weak_implies_4strong} that the weaker axiom of this article imply the stronger one of \cite{BFM}.

\smallskip
Throughout the course of this paper, we illustrate many of the nuances of possible dynamical and topological behavior for Anosov-like actions with examples.   Since many of these examples are built using a similar toolkit of previously introduced techniques and constructions, we defer their discussion to a separate section (Section \ref{sec:examples}) referring the main text there when needed.   

\subsection*{Outline} 
Section \ref{sec:preliminaries} contains background on bifoliated planes and Anosov-like actions.  We extend some results from \cite{BFM} to the nontransitive case, as well as introduce definitions and a toolkit used later in this work. 

Section \ref{sec:smale_class} introduces Smale classes and the Smale order, shows that Smale classes have a product structure, and contains the proof of Theorem \ref{thm:transitive} and consequently Theorem \ref{thm_dense_orbits_is_transitive}.  In Section \ref{sec:smale_chains} we show these always exist and behave like separating tori, proving Theorem \ref{thm:smale_chain_vague}. 

Section \ref{sec:Z2invariant} gives the proof of Theorems \ref{thm:discrete_stabilizer} and \ref{thm:transitive_or_torus} and improves Theorem \ref{thm:smale_chain_vague} in the case of cocompact actions.  

Section \ref{sec:density} shows that extremal Smale classes behave like attractors and repellers, proving the property which in the flow setting translates to having a dense set of weak stable or unstable leaves containing periodic orbits unique in their free homotopy class, which is a key ingredient in the next section.  

Section \ref{sec:boundary} discusses the induced action on the boundary of the bifoliated plane, shows this is minimal for Smale-bounded actions (Theorem \ref{thm_minimal_action_circle}), and proves Theorem \ref{thm:boundary_determines_action}. 

Finally, in Section \ref{sec:examples} we give several examples to illustrate nuances in the definitions and unusual behaviors which can occur, especially in the presences of prong singularities.

\begin{acknowledgement}
TB was partially supported by the NSERC (Funding reference number RGPIN-2017-04592).  CB thanks the Banach center Warsaw and Bedlevo, where some of the ideas for this article were
developed, and the support of the Simons Foundation
Award No.~63281 granted to the Institute of Mathematics of the Polish Academy
of Sciences for the years 2021-2023. KM was partially supported by NSF CAREER grant DMS-1933598, a Sloan fellowship, and a Simons sabbatical fellowship, and thanks the Institut Henri Poincar\'e (UAR 839 CNRS-Sorbonne Universit\'e), and LabEx CARMIN (ANR-10-LABX-59-01).
\end{acknowledgement}

 \section{Preliminaries: Anosov-like actions and bifoliated planes} \label{sec:preliminaries}
 
As mentioned in the introduction, our definition of Anosov-like actions on bifoliated planes is motivated by the dynamics of the action of the fundamental group of a 3-manifold on the orbit space of a pseudo-Anosov flow.  If $\phi$ is a flow on a manifold $M$, the {\em orbit space} $\orb_\phi$ of $\phi$ is the quotient of $\wt{M}$ by orbits of the lift of $\phi$ to $\wt{M}$. Fenley \cite{Fen_Anosov_flow_3_manifolds} and Barbot \cite{Bar_caracterisation} showed that, if $\phi$ is a (pseudo)-Anosov flow\footnote{The case of pseudo-Anosov flow was treated by Fenley--Mosher in \cite{FenMosher}.} on a 3-manifold, then $\orb_\phi$ is a topological plane.  The weak stable and unstable foliations descend to transverse 1-dimensional singular foliations $\cF^\pm$ on $\orb_\phi$ preserved by the action of $\pi_1(M)$.
 
\begin{rem} \label{rem_smooth_vs_topological}
There are two common notions of Anosov and pseudo-Anosov flows in the literature: \emph{smooth} and \emph{topological} (pseudo)-Anosov flows (see, e.g., \cite{AgolTsang}, although even these definitions have some variation within them in different texts). These notions are only known to coincide, up to orbit equivalence, for \emph{transitive} pseudo-Anosov flows (\cite{Shannon,AgolTsang}). The following proposition, and thus all the results of this article, applies even for the more general topological version of the definition (equivalent to the notion of an {\em expansive flow} -- see \cite[Section 1.1]{BM_book}.)
\end{rem} 

\begin{proposition} \label{prop:Orbit_space_anosov_like}
Let $\phi$ be a pseudo-Anosov flow on a compact 3-manifold $M$.  Then the action of $\pi_1(M)$ on the orbit space $\mathcal{O}_\phi$ satisfies 
\ref{Axiom_A1}--\ref{Axiom_totallyideal}.  
\end{proposition} 

\begin{proof} 
Axiom \ref{Axiom_A1} follows directly from the hyperbolicity of Anosov flows, and \ref{Axiom_prongs_are_fixed} holds by definition of pseudo-Anosov flow.  

Axiom \ref{Axiom_dense} is classical for Anosov flows, we briefly recall the argument to indicate that it holds also in the pseudo-Anosov case.   To establish this property, one needs to show that leaves containing lifts of periodic orbits are dense in both the weak stable and weak unstable foliations on $\wt{M}$. Let $\tilde x$ be any point in $\wt M$ and let $x$ be its projection to $M$. Since $M$ is compact, the forward orbit $\{\phi^t(x), t>0\}$ has accumulation points. Let $y$ be such an accumulation point. If $y$ is on a singular orbit $\gamma$, then the local unstable leaf of some $\phi^{t_n}(x)$, $t_n\to \infty$, intersects the local stable leaf of $y$ arbitrarily close to $y$. Hence, one can find lifts of $\gamma$ whose stable leaves intersects any neighborhood of the unstable leaf of $\wt x$. If $y$ is not on a singular orbit, then we can find $t_2>>t_1>>1$ such that $\phi_{t_1}(x)$ and $\phi_{t_2}(x)$ are as close as we want and inside a flow box not intersecting any singular orbits. Then, as in the proof of the Anosov closing lemma, one obtains that there exists a periodic orbit $\gamma$ through the flow box containing $\phi_{t_1}(x)$ and $\phi_{t_2}(x)$. Since $t_1,t_2$ can be chosen arbitrarily large, we deduce as above that the stable leaf of a lift of $\gamma$ will be as close as one wants to $\wt x$. This shows that stable leaves of periodic orbits are dense in $\wt M$. Flipping the direction of the flow yields the result for unstable leaves.
   
Finally, Axiom \ref{Axiom_non-separated} is Theorem D of \cite{Fenley_structure_branching} (that result is stated there for Anosov flows, but the same proof holds in the pseudo-Anosov case), and Axiom \ref{Axiom_totallyideal} is Proposition 4.4 of \cite{Fenley16}.  
 \end{proof}

\subsection{Structure of the bifoliated planes}
Following work of Barbot and Fenley, there is a well developed structure theory for the orbit spaces of Anosov flows.  
Several of these foundational results were established for bifoliated planes admitting (transitive) Anosov-like actions in \cite{BFM}.  Below, we recall some of this structure theory and establish necessary extensions to our setting.  

Going forward, $(P, \cF^\pm)$ always denotes a bifoliated plane equipped with an Anosov-like action of a group $G$.  

A bifoliated plane is called {\em trivial} if each leaf of $\cF^+$ meets each leaf of $\cF^-$; equivalently, if it is homeomorphic to $\R^2$ with the standard product foliation.  If $(P, \cF^\pm)$ is not trivial, then either the foliations are singular, or there exist pairs of leaves that do not intersect but rather ``just miss" each other, called a {\em perfect fit}.  Precisely: 
 
 \begin{definition}\label{def_perfect_fit}
 Leaves $l^\pm$ in $\cF^\pm$ make a \emph{perfect fit} if there are arcs $\tau^\pm$ starting at a point of $l^\pm$ and transverse to $\cF^\mp$ such that every leaf $k^+$ of $\cF^+$ that intersects the interior of ${\tau}^+$ intersects $l^-$, and every leaf $k^-$ of $\cF^-$ that intersects the interior of ${\tau}^-$ intersects $l^+$.
\end{definition}

We use the following terminology
\begin{definition} 
A {\em ray} of $\cF^\pm(x)$ is a properly embedded copy of $[0, \infty)$ in $\cF^\pm(x)$ based at $x$.\footnote{Elsewhere these are also called {\em half-leaves}, we prefer the term ray because it is more accurate in the case of singular points.} 

For a singular leaf $l$, we call a properly embedded copy of $\R$ bounding a connected component of $P \setminus l$ a {\em face} of $l$.  
\end{definition} 

As for leaves, we say that two rays make a \emph{perfect fit} if they satisfy Definition \ref{def_perfect_fit} given above, with ``leaves'' replaced by ``rays''.
Two perfect fits can be arranged to bound a region called a {\em lozenge} 

\begin{definition}[Lozenge]\label{def_lozenge}
Suppose $x,y \in P$ are distinct points such that there exist rays $r_x^{+}$ and $r_y^{-}$ of $\cF^+(x)$ and $\cF^-(y)$, respectively that make a perfect fit, as does a second pair of rays $r_x^{-}$ and $r_y^{+}$ (in $\cF^-(x)$ and $\cF^+(y)$).  Then
\begin{equation*}
 L := \lbrace p \in P \mid \cF^u(p) \cap r_x^s \neq \emptyset \text{ and } \cF^s(p) \cap r_x^u \neq \emptyset \rbrace
\end{equation*}
is called a {\em lozenge} with corners $x$ and $y$, and
the rays $r_x^{\pm}$ and $r_y^{\pm}$ are called its \emph{sides}. 
A \emph{closed lozenge} is the union of a lozenge with its sides and corners.
\end{definition} 
In other words, a closed lozenge is homeomorphic to a rectangle with two opposite corners removed, with the segments of $\cF^+$-leaves (resp.~$\cF^-$) sent to horizontal (resp.~vertical) segments.  A basic and elementary fact is as follows: 

\begin{lemma}[Lemma 2.23 \cite{BFM}] \label{lem_pf_gives_lozenge}
If $x$ is fixed by some nontrivial $g \in G$ and some ray through $x$ makes a perfect fit with a leaf $l$, then this perfect fit is part of a lozenge with corner $x$, and opposite corner on $l$. 
\end{lemma} 

A point that is not the corner of any lozenge is called a {\em non-corner point}.   If this point is fixed by some element $g \in G$, we call this a {\em non-corner fixed point}.  These will play an important role later on. For now, we note for future reference (without proof) the following easy result: 

\begin{lemma}[Lemma 2.29 \cite{BFM}] \label{noncorner_criterion}
Suppose $r^+$ and $r^-$ are two rays bounding a quadrant $Q$ of $x$.  If each of $r^+$ and $r^-$ 
intersect one of a pair of leaves making a perfect fit, or if they each intersect leaves of a singular point, then there are no lozenges in $Q$ with $x$ as a corner.
\end{lemma} 
The statement in \cite{BFM} is for transitive Anosov-like actions, but the only axioms used are in our definition here and so the proof applies verbatim.  

Lozenges may share corners, leading to what is called a \emph{chain of lozenges}.
\begin{definition}[Chain of lozenges]
A \emph{chain of lozenges} is a union of closed lozenges that satisfies the following connectedness property: for any two lozenges $L,L'$ in the chain, there exist lozenges $L_0, \dots,L_n$ in the chain such that $L=L_0$, $L'=L_n$, and, for all $i$, $L_i$ and $L_{i+1}$ share a corner (and may or may not share a side).

A chain of lozenges is called \emph{maximal} if it is not properly contained in any other chain of lozenges.
\end{definition} 

At this point we can also define {\em totally ideal quadrilateral}, a pathological structure that does not occur in the orbit spaces of pseudo-Anosov flows:

\begin{definition} \label{def:totally_ideal}
A {\em totally ideal quadrilateral} is an open set $Q \subset P$ bounded by four leaves, or faces, $l_0, l_2 \in \cF^+$ and $l_1, l_3 \in \cF^-$ such that each $l_i$ makes a perfect fit with $l_{i+1}$ (with indices taken modulo 4), and any leaf of $\cF^\pm$ intersects $l_i$ if and only if it intersects $l_{i+2}$. 
\end{definition} 
The condition on leaves intersecting the bounding leaves means that the interior of $Q$ is trivially foliated.

The following global structure result is proved for transitive Anosov-like actions in \cite[Theorem 2.16]{BFM}, following the outline of \cite[Th\'eor\`eme 4.1]{Bar_caracterisation}.  It also holds for general Anosov-like actions, with minor modifications to the proof, which we include below.  

 \begin{proposition}[Dynamical trichotomy for Anosov-like actions] \label{prop_trichotomy}
 Let $(P,\cF^+, \cF^-)$ be a bifoliated plane with a Anosov-like action of a group $G$. Then exactly one of the following holds:
  \begin{enumerate}[label=(\roman*)]
   \item $(P,\cF^+, \cF^-)$ is trivial.
   \item $(P,\cF^+, \cF^-)$ is homeomorphic to the 
diagonal strip bounded by $y=x-1$ and $y=x+1$ in $\bR^2$ with the horizontal and vertical foliations.  Such planes are called {\em skewed}.
   \item There is either a singular point in $P$, or the leaf spaces of $\cF^+$ and $\cF^-$ are both non-Hausdorff.
  \end{enumerate}
 \end{proposition}
 
 We describe the outline of the proof, for completeness, and describe the required modification from the proof given in \cite[Theorem 2.16]{BFM}.
 
\begin{proof}
Let $\cL^-$ and $\cL^+$ denote the leaf spaces of $\cF^-$, $\cF^+$ respectively.  
Suppose that $P$ has no singular points and $\cL^+$ is Hausdorff.  Since it is a simply connected 1-manifold it is thus homeomorphic to $\mathbb{R}$.   We need to show in this case that $\cL^- \cong \R$ and the plane is either skewed or trivial.    
Following the argument from \cite[Theorem 2.16]{BFM} verbatim, one shows that, either both leaf spaces are $\R$ and the plane is trivial, or each leaf $l \in \cL^-$ meets a {\em bounded} subset of leaves of $\cL^+$.  Fix an identification of $\cL^+$ with $\R$ and for $l \in L^-$, let $\alpha(l)$ and $\omega(l)$ denote the infimum, and supremum, respectively, of the sets of leaves of $\cF^+$ intersected by $l$.

We claim any leaf $l$ of $\cF^-$ separates $\alpha(l)$ and $\omega(l)$ in $P$.  This claim implies that $\cL^-$ is Hausdorff, hence homeomorphic to $\bR$, and thus $P$ can be seen as a subset of the plane obtained as the product of the leaf spaces $\cL^+\times\cL^- \simeq  \bR^2$ bounded by the graphs of the function $\alpha$ and $\omega$, as in  \cite[Theorem 3.4]{Fen_Anosov_flow_3_manifolds}, this is easily seen to be homeomorphic to the skewed plane.  

Thus, it remains only to prove the claim.  
Let $C\subset \cL^-$ be the set of leaves $l$ such that $l$ does \emph{not} separate $\alpha(l)$ and $\omega(l)$.  We assume $C$ is nonempty for a contradiction. 
 For any $l \in C$, denote by $D(l)$ the connected component of $\cL^- \smallsetminus l$ that does not contain $\alpha(l)$ and $\omega(l)$. Since $\cL^+\simeq \bR$, one shows easily as in \cite[Th\'eor\`eme 4.1]{Bar_caracterisation} that $D(l) \subset C$.  Hence $C$ has non-empty interior. Thus, by Axiom \ref{Axiom_dense}, it contains some leaf $l'$ fixed by some element $g$ (in fact a dense set of such leaves). Let $x \in l'$ be the point fixed by $g$ on $l'$. Up to replacing $g$ with $g^2$, we have that $\alpha(l')$ and $\omega(l')$ are both fixed by $g$.  Thus, they contain (unique) fixed points, say $x_\alpha$, and $x_\omega$, respectively, for $g$, and by considering the ($g$-invariant) sets of leaves simultaneously intersecting $\cF^+(x)$ and $\alpha(l')$ and $\omega(l')$, one can show that $x_\alpha$ and $x_\omega$ are corners of two adjacent lozenges, sharing a side in $\cF^+(x)$ and whose other (shared) corner is $x$.

Here is where a little extra care is needed compared to the case of transitive Anosov-like actions. 
The assertion above is true for any $l''$ in $D(l)$ fixed by some non trivial element $h \in G$. So, thanks to Axiom \ref{Axiom_dense}, we can find such a leaf $l''$ that intersects $\cF^+(x)$. Calling $x''$ the fixed point of $h$ on $l''$, we  will first show that $\cF^+(x'')$ intersects $l'$\footnote{This is the only argument that needs to be added compared to \cite{Bar_caracterisation} or \cite{BFM}, since in the transitive case, one can choose $x''$ close to $x$ so that property is automatically satisfied.}. Consider the set $S$ of $\cF^+$ leaves intersecting both $L''$ and $l'$. If $\cF^+(x'')$ is not in this set, then it means that there is a leaf $l_1^+$ in $\partial S$ that separates $\cF^+(x'')$ from $S$. Since $\alpha(l')$ and $\omega(l')$ are not in $S$ either, there are also leaves $l_\alpha^+,l_\omega^+\subset \partial S$ that separates $S$ from $\alpha(l')$ and $\omega(l')$ respectively (we may a priori have $\alpha(l')\subset \partial S$ in which case $\alpha(l')=l_\alpha^+$ and same for $\omega(l')$). Then $l_1^+$ is not separated with either $l_\alpha^+$ or $l_\omega^+$, contradicting the fact that $\cL^+$ is homeomorphic to $\bR$.
Hence, $\cF^+(x'')$ intersects $l'$ and one finishes the proof as in \cite{Bar_caracterisation} or \cite{BFM}. 
 \end{proof} 

Anosov-like actions on trivial bifoliated planes are well-understood:

 \begin{proposition}[\cite{BFM}, Prop. 2.7]
 Any Anosov-like action on a trivial bifoliated plane is conjugate to an action by affine transformations. 
 \end{proposition}
 This was proved for flows by Barbot \cite[Th\'eor\`eme 2.7]{Bar_caracterisation}, and for transitive Anosov-like actions in \cite[Proposition 2.7]{BFM}, the proof there does not use the full strength of transitivity and carries over in our setting without modification.   In the case of flows, this is equivalent to the fact that any (pseudo)-Anosov flow with trivial orbit space is the suspension flow of a hyperbolic linear map of the torus.

 \subsection{Axiom \ref{Axiom_non-separated} and non-separated leaves} 
 Axiom \ref{Axiom_non-separated} states that any non-separated leaf is fixed by some nontrivial element of $G$. Our next lemma improves this to show that for any maximal set of pairwise non-separated leaves, there exists a single (nontrivial) element $g\in G$ fixing all the leaves in that set. In particular this shows that Axiom \ref{Axiom_non-separated} implies the (a priori stronger) corresponding axiom in \cite{BFM}.
 
\begin{proposition}\label{prop_axiom4weak_implies_4strong}
Let $S$ be a set of pairwise non-separated leaves in $\cF^+$ or $\cF^-$. Given any leaf $l\in S$ and any $g\in G$ such that $gl = l$, then there exists $n\geq 1$ such that $g^n$ fixes every leaf in $S$.

In particular, Axiom \ref{Axiom_non-separated} implies that there always exists a nontrivial element fixing every leaf of a set of pairwise non-separated leaves.
\end{proposition}
 Note that this result and its proof are essentially the same as the proof of Theorem 4.3 in \cite{Fenley_structure_branching}, with some simplifications added and some details omitted.
 
 For the proof we use an easy observation 
 \begin{observation} \label{obs:nonseparated_implies_pf}
Suppose $l$ and $l'$ are nonseparated in $\cF^+$.  Consider the set $D$ of leaves of $\cF^-$ that separate $l$ from $l'$.  
If $l_n$ is a sequence of leaves converging to a union of leaves with both $l$ and $l'$ in the limit, then each leaf of $D$ will intersect $l_n$ for all sufficiently large $l_n$, hence $D$ inherits a linear order.  The supremum and infimum of $D$ are leaves of $D$ that make a perfect fit with $l$ and $l'$ (respectively, up to reversing the order) and each separates $l$ from $l'$. 
 \end{observation}
 
 \begin{proof}[Proof of Proposition \ref{prop_axiom4weak_implies_4strong}]
 To fix notation, we do the proof for $S\subset \cF^+$. Notice that $S$ comes with a natural order coming from a choice of coherent orientation on leaves of $\cF^+$ that accumulates to $S$.
 
 Take $l_0\in S$ and $g\in G$ fixing $l_0$.  If $g$ is trivial there is nothing to show; so we assume  $g$ nontrivial.  By Axiom \ref{Axiom_A1}, there exists a unique $x_0\in l_0$ fixed by $g$. Up to taking a power, we further assume that that $g$ fixes all the rays of $\cF^\pm(x_0)\smallsetminus \{x_0\}$.
 
 From now on, if $S$ contains singular leaves, we instead consider only the {\em faces} of these singular leaves that are pairwise nonseparated with other non-singular leaves or faces.  
We abuse notation and still denote by $S$ this set of leaves and faces, and call an element of $S$ a {\em leaf}, as this simplifies terminology and does not affect the proof.  
 
 Pick $l_0'\in S$ another leaf.  By Observation \ref{obs:nonseparated_implies_pf}, there exists $f_0\in \cF^-$ that makes a perfect fit with $l_0$ and separates $l_0$ from $l'_0$.
 As $g$ fixes every ray of $x_0$, it also fixes $f_0$ and Axiom \ref{Axiom_A1} again gives the existence of a fixed point $y_0\in f_0$. 
 Let $r_0$ be the ray of $\cF^+(y_0)\smallsetminus \{y_0\}$ that is contained in the side of $f_0$ that contains $l'_0$. Call $S_0$ the subset of $S$ on that same side of $f_0$. So in particular, $l'_0\in S_0$.
 
Similarly, reversing the roles of $l_0$ and $l'_0$, there exists $f'_0 \in \cF^-$ making a perfect fit with $l_0'$, separating it from $l_0$, fixed by some $g' \in G$ which also fixes $l_0'$.  Let $y'$ be the fixed point of $g'$ on $f_0'$ and let $r'_0$ be the ray of $\cF^+(y')\smallsetminus \{y'\}$ that is contained in the side of $f'_0$ that contains $l_0$, and $S_0'$ the subset of $S$ on the same side of $f_0'$ as $l_0$.  
 
 Now consider $\cF^-(r_0)$ the saturation of $r_0$ by leaves of $\cF^-$. Then, either: 
 \begin{enumerate} 
 \item $S_0 \subset \cF^-(r_0)$, or 
 \item there exists some point $z$ on a leaf of $S_0$ that is in a different connected component of $P\smallsetminus \cF^-(r_0)$ from $l_0$. 
 \end{enumerate}
 We first make a short argument to reduce to the latter case, up to reversing the roles of $l_0$ and $l_0'$: 
 
 \begin{claim}  \label{claim:switch_role}
 Either there exists $z$ on a leaf of $S_0$ that is in a different connected component of $P\smallsetminus \cF^-(r_0)$ from $l_0$, or there exists $z'$ on a leaf of $S_0'$ that is in a different connected component of $P\smallsetminus \cF^-(r_0')$ from $l'_0$.  
  \end{claim}

 \begin{proof}  
Suppose that $S_0 \subset \cF^-(r_0)$ holds.  Recall that $y'$ denotes the fixed point of $g'$ on $f'_0$. Let $x'$ be the fixed point of $g'$ on $l'_0$.  
Since $g'$ has a unique fixed point on $\cF^+(y)$, by Axiom \ref{Axiom_A1}, the leaves $\cF^-(x')$ and  $\cF^+(y')$ cannot intersect.  Thus, $\cF^+(y_0)$ separates $\cF^+(y')$ from $l'_0$ and $l_0$.   
Similarly, $\cF^-(x_0)$ cannot intersect $\cF^+(y_0)$, and therefore won't intersect $\cF^+(y')$.  It follows that the $\cF^-$ saturation of $r'_0$ does not contain $l_0$, and so after reversing the roles of $l_0$ and $l'_0$ in the dichotomy above, we have arrived in case (2), which is the statement in the claim.  
See Figure \ref{fig:4weak4strong}.
\end{proof} 

\begin{figure}[h]
   \labellist 
  \small\hair 2pt
     \pinlabel $y'$ at 60 60 
     \pinlabel $y_0$ at 90 94 
   \pinlabel $l_0$ at 152 145
    \pinlabel $f_0$ at 152 90
    \pinlabel $f'_0$ at 152 70
  \pinlabel $l'_0$ at 152 15  
 \endlabellist
     \centerline{ \mbox{
\includegraphics[width=5cm]{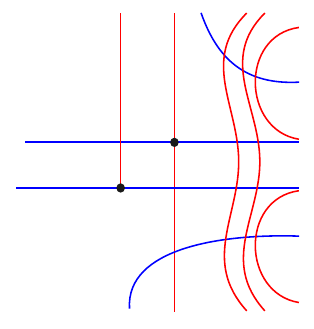} }}
\caption{}
\label{fig:4weak4strong}
\end{figure}
We now assume we are in case (2) and use the following claim.  
 
\begin{claim} \label{claim:find_next_leaf}
Suppose there exists some point $z$ on a leaf of $S_0$ that is in a different connected component of $P\smallsetminus \cF^-(r_0)$ from $l_0$ (i.e., case 2 holds). 
Then there exists $l_1\in S_0$, fixed by $g$, and such that no leaves of $S$ are in between $l_0$ and $l_1$ for the natural linear order of $S$.
\end{claim}

\begin{proof}
Let $z$ be as above.  
There exists a unique leaf $f_1 \in \cF^-$ such that it, or one of its face in the singular case, is in $\partial ( P\smallsetminus \cF^-(r_0))$ and separates $z$ from $x_0$.
 We claim that $f_1 \cap r_0 = \emptyset$. This is automatic if $f_1$ is non-singular, but assume not and call $t_0 = f_1\cap r_0$. Then, by construction, for all $t>t_0$ in $r_0$ (where $>$ designate the positive direction of the ray), we have $\cF^-(t)\cap S_0 = \emptyset$. However, since $S$ consists of pairwise non-separated leaves, there exists $l^+$ that intersects both $f_0$ and $\cF^-(z)$. But (up to replacing $g$ by $g^{-1}$) the sequence $g^n l^+$ contains $\cF^+(y_0)$ in its limit as $n\to +\infty$. In particular, we would get $t>t_0$ in $r_0$ such that $\cF^-(t)\cap S_0 \neq \emptyset$, a contradiction.

So $f_1  \cap r_0 = \emptyset$. Moreover, $f_1$ must be fixed by $g$ (since $S$ is invariant under $g$, as well as $\cF^-(r_0)$). In fact, one can notice that $f_1$ makes a perfect fit with $r_0$, but we do not need that.
Thus, there exists $x_1\in f_1$ fixed by $g$. It is then easy to see that $l_1 = \cF^+(x_1)$ is in $S$ and no other leaf of $S$ can be between $l_0$ and $l_1$.
\end{proof}
  
 Iteratively applying Claim \ref{claim:switch_role} and Claim \ref{claim:find_next_leaf}, we can produce two sequences of leaves $l_0, l_1, \ldots l_k$ and $l_0', l_1', \ldots l_j'$ (possibly with $k=1$ and/or $j=1$) such that each $l_i$ is fixed by $g$, each $l_i'$ is fixed by $g'$ and such that no leaves of $S$ are in between $l_i$ and $l_{i+1}$ for the natural linear order of $S$ nor between $l_i'$ and $l_{i+1}'$. 

If at some stage $l_k$ and $l_j'$ share a common leaf that makes a perfect fit with each of them (i.e. $f_k = f_j'$ in the notations used in Claim \ref{claim:switch_role}, with indices shifted appropriately), then the process terminates, so there are only finitely many orbits in $S$ between $l_0$ and $l_0'$ and they are all fixed by $g$ (and $g'$).

Thus, it remains only to argue that the process does indeed terminate for any choice of $l_0'$.  If not, we would obtain at least one infinite sequence of the form $l_0, l_1, \ldots$ or $l_0', l_1', \ldots$.  For concreteness, suppose the sequence $\{ l_i \}$ is infinite.  These leaves are all distinct, fixed by $g$, and with distinct leaves $f_i$ of $\cF^-$ making perfect with them, which are also fixed by $g$.  
The leaves $f_i$ lie in a bounded region of the leaf space (the closure of the interval separating all $l_i$ from all $l_i'$), so must accumulate somewhere, i.e., limits onto a leaf or union of leaves. A unique leaf in that limit, call it $f_\infty$, separates the $l_i$ from the $l'_j$,  so $f_\infty$ is also fixed by $g$.  
This gives a contradiction with Axiom \ref{Axiom_A1} since the leaf of $\cF^+$ through the fixed point of $g$ on $f_\infty$ would intersect (infinitely many) of the fixed leaves $f_i$.  This contradiction concludes the proof.  
 \end{proof}
 
\begin{rem}
We note that one could be slightly more precise in the proof of Proposition \ref{prop_axiom4weak_implies_4strong} to deduce directly that the leaves in the set $S$ are arranged as a \emph{line of lozenges} (see Definition \ref{def_line_lozenges}), but this will follow from Lemma \ref{lem:two_fix_points} below.
\end{rem} 
 
 \subsection{Product and scalloped regions} 
 
 \begin{definition}
 An {\em infinite product region} is a subset of a bifoliated plane properly homeomorphic to $I \times [0, \infty) \subset \bR^2$ with its product foliations; here $I$ may be any closed interval, finite or infinite.    We say the interval $I$ {\em bounds} the region. 
 \end{definition}
 
 \begin{proposition}\label{prop:no_product}
 $P$ has an infinite product region if and only if $P$ is trivial 
 \end{proposition}
 
 \begin{proof} 
 Suppose $U$ is an infinite product region in $P$, bounded by an interval $I$. 
 Up to switching the labels $+$ and $-$, we assume $I$ lies in a leaf of $\cF^+$.  
 By axiom \ref{Axiom_dense} there is a leaf $l^-$ of $\cF^-$ intersecting $U$ and fixed by some $g \in G$.    Let $x$ be the unique fixed point of $g$ on $l^-$.  Then $\bigcup_{k \in \mathbb{Z}} g^k(U)$ contains an infinite product region bounded by $\cF^+(x)$.  Let $y_i$ be points along  $\cF^+(x)$ such that the sequence $y_i$ eventually leaves every compact set, and for each $i$ let $r_i$ be a  ray of $\cF^-(y_i)$ contained in $U$. 
  
 We claim that the sequence $\{r_i\}$ does not accumulate onto any leaf: if a subsequence had a limit, then it would either limit onto a single leaf,
 or limit onto a union of nonseparated leaves.  
 In the first case, this leaf is fixed by $g$, so has some point $z$ fixed by $g$. Hence, $\cF^+(z)$ cannot intersect $\cF^-(x)$, contradicting the product region structure. In the second case by Proposition \ref{prop_axiom4weak_implies_4strong} there is some $h$ fixing all the limit leaves. Then we deduce that $h$ also fixes $\cF^+(x)$ (because $\cF^+(x)$ is in the boundary of the $\cF^+$-saturation of the union of limit leaves), and thus admit a fixed point $x'$ on $\cF^+(x)$.   But then $\cF^-(x')$ intersects a leaf of $\cF^+$ fixed by $h$, giving two fixed points for $h$ on $\cF^-(x')$, contradicting \ref{Axiom_A1}.  

As a consequence, if $U'$ is another infinite product region bounded by some leaf of $\cF^+$ fixed by some $g'$ and that intersects $U$, then either $U \subset U'$ or $U' \subset U$.  
Consider the union of all such regions $U'$.  Denote this set by $V$.  Then, by the above, $V$ is $\cF^+$ saturated, and we claim that it is all of $P$.  To show this, consider any leaf $l_0$ of $\cF^-$ fixed by an element $h \in G$ and such that $l_0$ intersects $V$.   Such leaves form a dense subset of $V$ by Axiom \ref{Axiom_dense}.  
Let $x_0$ be the point fixed by $h$ in $l_0$.  By the argument above, we have $x_0 \in V$.  
But we also have $h^N(V) \subset V$, and so we conclude that $l_0 \subset V$.   This implies that $V$ is $\cF^-$ saturated as well, and so $V = P$.    Since $V=P$ is an increasing union of infinite product regions, it has a global product structure, so is trivial.  
 \end{proof} 

The following result is an important statement about the structure of fixed points of group elements.  
\begin{lemma}[Lemma 2.23 and Prop. 2.24 in \cite{BFM}, see also \cite{Fenley_QGAF}] \label{lem:two_fix_points}
If $g \in G$ is a nontrivial element fixing more than one point, then the fixed points of $g$ are corners of a $g$-invariant chain of lozenges.  
In particular: 
\begin{enumerate} 
\item if $g$ preserves two leaves making a perfect fit, then it preserves a lozenge whose corners lie in these leaves. 
\item if $g$ fixes a non-corner point, then this point is its unique fixed point in $P$.
\end{enumerate} 
\end{lemma} 

\begin{proof}
The proof of Lemma 2.23 and Prop. 2.24 in \cite{BFM} applies directly in this setting, given Proposition \ref{prop:no_product} and Proposition \ref{prop_axiom4weak_implies_4strong}.
\end{proof}

We introduce another definition (from \cite{BFM}).
\begin{definition}\label{def_line_lozenges}
A {\em line of lozenges} is a chain $L_i$, indexed by some subset of consecutive integers $i \in \bZ$, such that $L_i$ shares opposite sides with $L_{i-1}$ and $L_{i+1}$.  Thus, there is a common leaf of either $\cF^+$ or $\cF^-$ intersecting all $L_i$.
\end{definition}

As a consequence of Proposition \ref{prop_axiom4weak_implies_4strong} and Lemma \ref{lem:two_fix_points}, we have the following 
\begin{corollary}[See, e.g., \cite{BFM} Lemma 2.26] \label{cor:nonseparated_leaves}
Any collection of pairwise nonseparated leaves forms the sides of a line of lozenges.  
\end{corollary} 

It turns out that \emph{infinite} lines of lozenges are special: They are part of what is called a \emph{scalloped region}. 
A scalloped region is a trivially foliated open set $U \subset P$ whose boundary $\partial U$ consists in four infinite families of pairwise non-separated leaves. In particular, a scalloped  region can be realized as two distinct bi-infinite lines of lozenges: one sharing sides along leaves of $\cF^+$ and the other sharing $\cF^-$-sides. See Figure \ref{fig_scalloped_region}.  A formal definition can be found in e.g., \cite[Definition 2.31]{BFM}.  

\begin{lemma}\label{lem_infinite_line_lozenges}
Let $\{L_i\}_i$ be an infinite line of lozenges, then there exists a (unique) \emph{bi-infinite} line of lozenges $\cL$ containing $\{L_i\}_i$ and forming a scalloped region.  
Furthermore, the stabilizer of a bi-infinite line of lozenge is a subgroup $H$ that is virtually isomorphic to $\bZ^2$, with one $\bZ$-factor fixing every corners of $\cL$ and every other disjoint $\bZ$-factor acting freely.
\end{lemma}

\begin{figure}[h]
     \centerline{ \mbox{
\includegraphics[width=8cm]{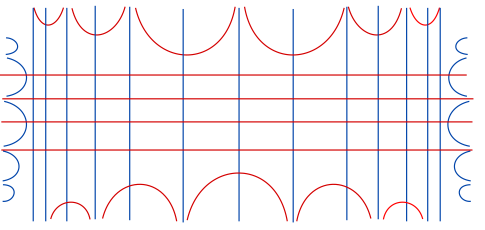} }}
\caption{A scalloped region, with its two infinite lines of lozenges}
\label{fig_scalloped_region}
\end{figure}

\begin{proof}
The fact that infinite lines are contained in bi-infinite lines is given by Lemma 2.32 of \cite{BFM}, the second assertion is given by Lemma 2.33 of \cite{BFM}. Note that neither of these proofs use the transitivity axioms, hence apply directly in our more general setting.
\end{proof}

\subsection{Finite branching} 
The next lemma is a generalization of Fenley's finite branching result (see \cite[Theorem F]{Fenley_structure_branching}). Fenley's result states that for a pseudo-Anosov flow on a compact $3$-manifold $M$, only finitely many leaves in $M$ lift to \emph{branching leaves} in $\tilde M$, i.e., leaves that project to non-Hausdorff points in the leaf space. 

In our general setting, Axiom \ref{Axiom_non-separated} implies that any branching leaf $l \in \cF^{\pm}$ is fixed by some nontrivial element $g\in G$. In particular, it contains (by Axiom \ref{Axiom_A1}) a unique fixed point $x\in l$ by Remark \ref{rem_single_fixed_point}. 

Moreover, the fixed point $x$ in $l$ is, by Lemma \ref{lem:two_fix_points}, a corner of a lozenge $L$. Since $F$ is branching, there exists a lozenge $L'$ which shares a side with $L$ and such that the common corner $p$ of $L$ and $L'$ is opposite to $x$ in $L$.
Following the terminology introduced in \cite{Fenley_structure_branching}, we call such shared corners of lozenges sharing a side a \emph{pivots}.
Our generalization of Fenley's finiteness result is the following. 

\begin{lemma}\label{lem:adjacent_corners_discrete}
The sets 
\begin{align*}
&\mathrm{Sing}:=\{p \in P \mid p \text{ is a singular point}\}, \text{ and} \\
&\mathrm{Pivot}:= \{p\in P \mid \text{a half leaf of } p \text{ is the side of two lozenges}\}
 \end{align*}
are closed and discrete subsets of $P$.
\end{lemma}

For the proof we use the following basic terminology
\begin{definition}
For $x\in P$, a \emph{quadrant} of $x$ is a connected component of $P \setminus (\cF^+(x) \cup \cF^-(x))$.  
\end{definition}

\begin{proof}[Proof of Lemma \ref{lem:adjacent_corners_discrete}]
By definition the set of singular points is discrete and it is also closed since any nonsingular point has a product foliated neighborhood.  Thus, we need only treat the case of pivots. 

Suppose for contradiction that $x_n$ is a sequence of pivot points converging to some point $x$.   After passing to a subsequence, we can assume that $x_n$ lie in a single quadrant $Q$ of $x$, and passing to the tail end we assume that all are in a trivially foliated subset of this quadrant, in particular they are nonsingular. Up to taking a further subsequence (and potentially relabelling the foliations) we may assume that the leaves $\cF^+(x_n)$ are always the shared sides of two lozenges. 

For concreteness, fix a local orientation on $Q$ in a product-foliated neighborhood of $x$ so that $Q$ is the upper-left quadrant of $x$ and the local foliations are represented with $\cF^+$ as vertical and $\cF^-$ horizontal.  After passing to a further subsequence, we can assume that the pivots all have the same orientation, that is to say that the 
perfect fits of all the $\cF^+(x_n)$ leaves are either all above $\cF^-(x)$ or all below $\cF^-(x)$.
We will show that each case contradicts Lemma \ref{noncorner_criterion} 

Suppose as a first case that the perfect fits are above.   Note that the fact that all $x_i$ are in a trivially foliated neighborhood of $x$ in $Q$ implies that each of $\cF^\pm(x_i)$ intersects $\cF^\mp(x)$ (respectively). 
Take $n$ large enough so that $x_n$ lies in the rectangle bounded by $\cF^\pm(x_1)$ and $\cF^\pm(x)$.  Then the perfect fit formed by $\cF^+(x_1)$ and the side of the rightmost lozenge with corner $x_1$ is in the top left quadrant of $x_n$, so, by \cite[Lemma 2.29]{BFM} $x_n$ cannot have a lozenge there, a contradiction.  Similarly, if the perfect fits are all below, $x_n$ cannot have a lozenge in its lower left quadrant, again a contradiction. \qedhere

\end{proof}

 \section{Smale classes and topological transitivity} \label{sec:smale_class}

In this section, we generalize the Smale order and basic sets for Anosov flows to the context of an Anosov-like action of a group $G$ on bifoliated plane $(P, \cF^+,\cF^-)$.   In particular, we show how the classical Smale order and basic sets (see \cite{Smale}) can be explicitly detected in the orbit space, that Smale classes have a product structure, and we prove Theorems \ref{thm_dense_orbits_is_transitive} and \ref{thm:transitive}.  
Later, in Section \ref{sec:smale_chains}, we will also recover the analogue of tori separating basic sets from the orbit space. 

As in the previous section $(P, \cF^+,\cF^-)$ denotes a bifoliated plane and $G$ a group acting on $P$ with an Anosov-like action.  
 
\begin{notation}
Define $\Fix_G := \{x \in P  \mid \exists g\neq \id \in G \text{ with } gx = x\}$. Let $\Fixbar_G \subset P$ denote the closure of $\Fix_G$ in $P$.  
\end{notation}

\begin{definition} 
The {\em regular set}  $\cR_G$ is the subset of $\Fixbar_G$ consisting of points $x\in \Fixbar_G$ such that no leaf of $\cF^{\pm}(x)$ is a singular leaf. Equivalently, $\cR_G$ is obtained from $\Fixbar_G$ by removing all prongs and their leaves.
\end{definition}

\begin{rem}
When $P$ is the orbit space of an Anosov flow on a 3-manifold $M$, and  $G =\pi_1(M)$ with the induced action, then 
 $\Fixbar_G = \cR_G$ and it is equal to the non-wandering set of the action of $G$. It is also equal to the (projection to the orbit space of the) non-wandering set of the flow, which is the chain-recurrent set.

However, even for actions coming from orbit spaces of pseudo-Anosov flows on compact manifolds, there may be differences between chain-recurrent and non-wandering sets.  In this case $\Fixbar_G$ corresponds to the projection of the non-wandering set.  The non-wandering set can be strictly smaller than the chain recurrent set, as shown by the examples given in Propositions \ref{prop:weird_loops} and \ref{prop_chain-recurrent_non_transitive}.
\end{rem}

For Anosov-like actions, the interesting dynamics happens in the set $\Fixbar_G$.  Working in the regular set $\cR_G$ is necessary to deal with the unusual behavior that arise in the presence of prongs, but it may happen that the closure of $\cR_G$ misses parts (or even all, see Proposition \ref{ex_only_singular_Smale}) of $\Fixbar_G$, and therefore misses parts of the dynamics.  Thus, we introduce some terminology for points of $\Fixbar_G\smallsetminus \bar{\cR}_G$ and we will keep track of them separately.  

\begin{definition}
A point $p\in P$ is called an \emph{isolated prong singularity} if $p$ is a singularity of the foliations and an isolated point of $\Fixbar_G$. Equivalently, isolated prong singularities are the points in $\Fixbar_G\smallsetminus \overline{\cR}_G$.
\end{definition}

\begin{definition}[Smale order, Smale class] \label{def:smale_class}
For $x,y\in \cR_G$, we write $x\lG y$ if there exists $g\in G$ such that $\cF^+(x)\cap \cF^-(gy)\neq\emptyset$.
If $x\lG y$ and $y\lG x$, we say that $x,y\in \cR_G$ are in the same {\em Smale class}\footnote{Our next lemma will show that $\sim_G$ is an equivalence relation, justifying this terminology.}, and write $x\sim_G y$. 

We extend this relation to isolated prong singularities, by saying $x \sim_G y$ if $x = gy$ for some $g \in G$.  
\end{definition}

If $P$ is trivial, then there is a unique Smale class, so the results of this section are only interesting in the nontrivial case.  
Note also that, by definition, if $x \lG y$, then $x \lG gy$ for all $g \in G$.  In particular, Smale classes are $G$-invariant sets.  

In the case of the orbit space of an Anosov flow, the relation $\lG$ is defined and a transitive relation on all of $\Fixbar_G$, giving the classical Smale order, and the Smale classes are basic sets (\cite{Smale}, or see e.g., \cite[Section 5.3]{FH_book}).   Transitivity of the relation $\lG$ on $\Fixbar_G$ also holds whenever $\cF^{\pm}$ are true foliations, without singularities, by Lemma \ref{lem_lG_reflexive_transitive} below.   However, in the presence of prongs, the property of having nonempty intersection is not open in the space of pairs of leaves in $\cF^+$ and $\cF^-$ and transitivity may fail for points in $\Fixbar_G$ -- see Proposition \ref{ex_prongs_in_distinct_classes}. The regular set $\cR_G$ is the set where $\lG$ always induces a partial order on Smale classes, as we show in the next lemma.  

\begin{lemma}\label{lem_lG_reflexive_transitive}
The relation $\lG$ is reflexive and transitive on $\cR_G$.  Consequently, the relation $\sim_G$ is an equivalence relation, and $\lG$ induces a partial order on the Smale classes that do not consist of the orbit of an isolated prong singularity.
\end{lemma}
\begin{proof}
Reflexivity is trivial, taking $g = id$.   For transitivity, suppose $x,y,z\in \cR_G$ satisfy $\cF^+(x)\cap \cF^-(gy)\neq\emptyset$ and $\cF^+(y)\cap \cF^-(hz)\neq\emptyset$, so also $\cF^+(gy)\cap \cF^-(ghz)\neq\emptyset$.  Let $y' \in \cR_G$ be a point fixed by some nontrivial element $k \in G$ and sufficiently close to $gy$ so that $\cF^+(x)\cap \cF^-(y') \neq \emptyset$ and $\cF^+(y')\cap \cF^-(ghz)\neq\emptyset$.  Then (up to replacing $k$ with $k^{-1}$) we will have
that $\cF^-(k^n gh z) \cap \cF^+(x) \neq \emptyset$ as desired. 
\end{proof}

 In order to more succinctly differentiate between Smale classes whose elements are in $\cR_G$ from the ones containing only isolated prong singularities, we introduce the following terminology:
\begin{definition}
A Smale class $\Lambda$ is called \emph{regular} if $\Lambda\subset \cR_G$. It is called \emph{singular} if $\Lambda$ is the orbit of an isolated prong singularity.
\end{definition}

Since we have an order on regular Smale classes, we introduce the following terminology, which will be important in Section \ref{sec:density}.
\begin{definition}[Extremal classes]\label{def_extremal_class}
We call a Smale class \emph{extremal} if it is  regular and either minimal or maximal for the Smale order. 
\end{definition}

Our next goal is to give an equivalent characterization of Smale classes that is easier to work with.  For this, we use the following definition from \cite{BFM}
\begin{definition}
Two points $a,b\in P$ are \emph{totally linked}, abbreviated {\em (TL)} if
\[
\cF^+(a)\cap \cF^-(b) \neq \emptyset \text{ and }  \cF^-(a)\cap \cF^+(b) \neq \emptyset.
\] 
\end{definition}

As in the classical case of hyperbolic basic sets, we will now show that the Smale classes of an Anosov-like action have a {\em product structure}, in the same sense as hyperbolic sets -- see e.g., \cite[Definition 6.2.6]{FH_book}. 

\begin{proposition}[Product structure on Smale classes] \label{prop_basic_product_structure}
 Let $a,b\in \cR_G$ be in the same Smale class $\Lambda$. Assume that $\cF^+(a)\cap \cF^-(b) \neq \emptyset$. 
Then, for any neighborhood $U$ of $\cF^+(a)\cap \cF^-(b)$, we have $U\cap \Lambda \neq \emptyset$. 
In particular, the intersection point $x= \cF^+(a)\cap \cF^-(b)$ is in $\cR_G$ and in the same Smale class as $a$ and $b$.
\end{proposition}

To show this, we will in fact prove something stronger (Proposition \ref{prop:product_structure_general} below) which applies also to singular points.  
This generalized version will be useful later in this work.  

\begin{proposition} \label{prop:product_structure_general}
Let $a, b \in \Fixbar_G$ satisfy $x = \cF^+(a)\cap \cF^-(b) \neq \emptyset$. Assume that $x$ is not singular.
Let $r_a$ and $r_b$ be the rays of $\cF^+(a)$ and $\cF^-(b)$ (respectively) containing $x$, and let $f_a, f_b$ be the faces of $\cF^-(a)$ and $\cF^+(b)$ bounding quadrants with common boundary $r_a$, $r_b$ (respectively). 

If there exists $g\in G$ such that $g (f_b) \cap f_a \neq \emptyset$, then for any neighborhood $U$ of $x$, there exists a point $y \in \mathcal{R}_G \cap U$ totally linked with $a$ and $b$.  
In particular, $x\in \Fixbar_G$.\footnote{If $x$ was singular, then we automatically would have that $x\in \Fixbar_G$ and is totally linked with both $a$ and $b$, but may not have any regular fixed points totally linked with both $a$ and $b$.}
\end{proposition}
Note that in the case where $a, b$ are on nonsingular leaves, the ``faces" are simply the leaves through $a$ and $b$, and so the assumptions are satisfied whenever $a$ and $b$ are in the same Smale class, giving the statement of Proposition \ref{prop_basic_product_structure}.  

\begin{proof}[Proof of Proposition \ref{prop:product_structure_general}]
We adopt the notation and assumptions of 
the proposition.  
Let $U$ be a small product foliated neighborhood of $\cF^+(a)\cap \cF^-(b)$.  
We will find a point of $\Fix_G$ in $U$; for $U$ small enough, this point will automatically be totally linked with $a$ and $b$.  Moreover, since $U$ is product foliated, this point cannot be a prong singularity, and since it is a fixed point, it cannot be a nonsingular point on a prong leaf.  Thus, this point will be in $\mathcal{R}_G$, which is what we wanted to show.  

If $a$ (and/or $b$) is singular, it is (or they are) in $\Fix_G$.  If either is not in $\Fix_G$ (and hence nonsingular), we can replace them by arbitrarily nearby (nonsingular) points which are in $\Fix_G$, while preserving the fact that $\cF^+(a)\cap \cF^-(b)\neq \emptyset$ and $\cF^-(a)\cap g\cF^+(b)\neq\emptyset$. Thus, going forward we make the assumption that $a, b \in \Fix_G$.  
 
 Let $\alpha$ and $\beta$ be elements fixing $a$ and $b$, respectively.  Up to replacing with powers and inverses, we assume they fix all rays of $\cF^\pm(a)$ and $\cF^\pm(b)$ (respectively) and are expanding on $\cF^+(a)$ and $\cF^+(b)$. 
 Let $I$ be a small interval of leaves of $\cL^+$ containing, in its interior, the face $f_a$ of $\cF^+(a)$ such that $g(f_b) \cap f_a \neq \emptyset$ as in the hypotheses of the Proposition.  
 Choose this interval small enough so that 
 $I$ lies inside the saturation of $U$ by $\cF^+$.  
 Then, for any large enough $n,m>0$, we will have that $g \beta^n(I)$ consists of the leaves passing through a small interval of $\cF^-(a)$, or of $f_a$ in the case where $a$ is singular.  Thus $\alpha^m g \beta^n(I) \subset I$.  We conclude that  $h= \alpha^m g \beta^n$ fixes a leaf of $\cF^-$ in $I$. 
 
Similarly, if we consider an interval $J$ in $\cL^-$ that contains $\cF^-(b')$ in its interior and is inside the $\cF^-$-saturation of $U$, we see that (up to increasing $n,m$ if necessary) $h^{-1}(J) \subset J$. Therefore $h$ has a fixed point in $U$, which is what we needed to show. 
\end{proof}

\begin{figure}
   \labellist 
  \small\hair 2pt
 \pinlabel $a$ at 80 150
 \pinlabel $gb$ at 90 35
 \pinlabel $b$ at 175 105
 \pinlabel $I$ at 245 138
 \pinlabel $\beta^n(I)$ at 140 110
 \pinlabel $g\beta^n(I)$ at 135 20
 \endlabellist
     \centerline{ \mbox{
\includegraphics[width=8cm]{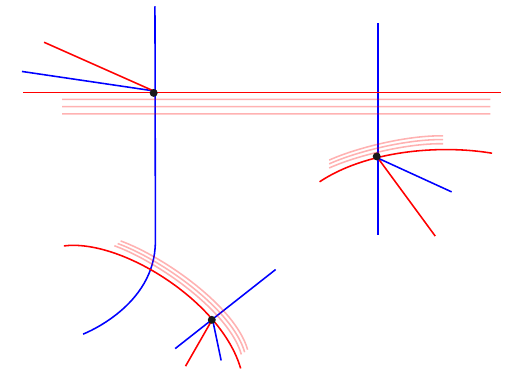} }}
\caption{Proof of local product structure}
\label{fig:product_structure}
\end{figure}

We note a special case of the argument above that will be useful later on.  

\begin{corollary} \label{cor:TL_product}
Suppose $a,b\in \Fix_G$ are totally linked. Let $z_1$ and $z_2$ be the intersections of $\cF^{\pm}(a) \cap \cF^{\mp}(b)$, and let $U_i$ be a neighborhood of $z_i$ in the quadrilateral bounded by $\cF^{\pm}(a)$ and $\cF^{\mp}(b)$.
 Then there exists $g \in G$ such that $g(U_1) \cap U_2 \cap  \Fix_G \neq \emptyset$, and preserving the local orientation of the foliations in the rectangle bounded by $\cF^{\pm}(a)$ and $\cF^{\pm}(b)$.
\end{corollary}

\begin{proof} 
As above, consider $\alpha$, $\beta$ fixing all half leaves through $a$ and $b$ that are expanding on $\cF^+(a)$ and $\cF^+(b)$, respectively; and hence preserving local orientation in the rectangle bounded by $\cF^{\pm}(a)$ and $\cF^{\pm}(b)$.   Proposition \ref{prop:product_structure_general}, taking $g$ to be the identity, shows that for $n, m$ sufficiently large $\alpha^n \beta^m$ has a fixed point $z$ in any given neighborhood $U$ of $\cF^+(a) \cap \cF^-(b)$, and the proof (considering the images of $I$ and $J$) easily shows that this necessarily lies in the quadrilateral bounded by $\cF^{\pm}(a)$ and $\cF^{\mp}(b)$.  Similarly, given a neighborhood $U'$ of $\cF^+(b) \cap \cF^-(a)$ in the quadrilateral bounded by the leaves of $a, b$, for sufficiently large $m, n$ the map $\beta^m \alpha^n$ has a fixed point $z'$ in $U'$.   Since $\beta^m \alpha^n = \beta^m( \alpha^n \beta^m) \beta^{-m}$, we have that $z' = \beta^n(z)$, as desired. 
\end{proof}

As an immediate consequence, we also obtain a local product structure on $\Fixbar_G$. 
\begin{corollary}\label{cor_product_means_product}
  $\Fixbar_G$ has a local product structure: in a trivially foliated neighborhood $I \times J$ of some nonsingular point $x \in \Fix_G$, the set $\Fixbar_G$ has the form $C_1 \times  C_2$, where $C_1, C_2$ are closed subsets of the intervals $I, J$ respectively.  
\end{corollary} 

\begin{rem}
One easily shows that, up to shrinking the product neighborhood in the above corollary, the closed sets $C_i$ are either a closed interval, a Cantor set or an isolated point. The work in the next few sections will show that, as in the case of Anosov flows, only four of the six (up to symmetry) possible topological types of product can happen: One may have a product $\mathrm{Interval} \times \mathrm{Interval}$ (which happens only if the action is transitive), a product $\mathrm{Interval}\times \mathrm{Cantor}$ (which happens in extremal Smale classes, see Observation \ref{obs:maximal_are_F+_saturated}), a product $\mathrm{Cantor}\times \mathrm{Cantor}$ (which happens for any non-isolated point in a non-extremal Smale class), and finally isolated fixed points, i.e., of the form $\mathrm{Isolated}\times \mathrm{Isolated}$.
See Proposition \ref{prop:isolated_fixed_points} and Corollary \ref{cor_boundary_leaves_and_boundary_points} for a proof. 
\end{rem}

The product structure on Smale classes also gives us the following characterization:
\begin{corollary}\label{cor:char_smale_class}
For any $x,y\in \cR_G$, the following are equivalent
\begin{enumerate}
\item $x\sim_G y $ 
\item there exists $g\in G$ and $z\in \cR_G$ such that $x$ is TL with $z$, and $z$ is TL with $gy$.
\item there exists $g\in G$ and $z_1,\dots z_k\in \cR_G$ such that $x$ is TL with $z_1$, $z_i$ is TL with $z_{i+1}$, and  $z_k$ is TL with $gy$.  
\end{enumerate} 
\end{corollary}

\begin{proof}
The fact that (1) implies (2) is a direct consequence of Proposition \ref{prop_basic_product_structure}: If $x\sim_G y $, then given $g\in G$ such that $\cF^+(x)\cap \cF^-(gy)\neq \emptyset$, we can choose $z\in \cR_G$ close enough to that intersection, and thus $x$ is totally linked with $z$ and $z$ is totally linked with $gy$.

Condition (2) trivially implies (3), and transitivity of the relation $\sim_G$ (and its invariance under the group) show that (3) implies (1).
\end{proof}

Using this, we will be able to prove Theorem \ref{thm:transitive}.  Recall the statement: 

\transitiveonclasses*

While the fact that $G$ acts topologically transitively on each singular Smale class is tautologically true, in order to deduce the ``moreover..." part of the above result, one needs to deal separately with the case of Anosov-like actions such that all Smale classes are singular. We use the following result, which will also play a role in Proposition \ref{prop_finite_or_cocmpact_implies_Smale_bounded}. 

\begin{lemma}  \label{lem_only_prongs}
If $\Fixbar_G$ consists only of isolated prong singularities, then these must lie in infinitely many distinct $G$-orbits.  Consequently, there are infinitely many Smale classes. 
\end{lemma} 

We remark that there do exist examples of such Anosov-like actions with only isolated prong singularities as fixed points, see Proposition \ref{ex_only_singular_Smale}. To streamline the proof of this lemma, we use one more fact whose proof we postpone to Section \ref{sec:smale_chains}.  There is no circularity here since Lemma \ref{lem_only_prongs} and its consequence are not used in Section \ref{sec:smale_chains}.
\begin{lemma}
\label{lem_isolated_prong_implies_lozenges}
If $p$ is an isolated prong singularity, then each of its quadrants admits a lozenge with corner $p$.
\end{lemma}
This is proved in Proposition \ref{prop:isolated_fixed_points}.  

\begin{proof} [Proof of Lemma \ref{lem_only_prongs}]
Suppose that $\Fixbar_G$ consists only of isolated prong singularities.

By Lemma \ref{lem_isolated_prong_implies_lozenges}, each fixed point is a corner of a lozenge in each of its quadrants, thus each lozenge is included in a scalloped region in $P$.  Fix such a lozenge $L_1$ with corner $a_1$, in scalloped region $S$ and enumerate the quadrants of $a_1$ by $Q^1_1, Q^1_2, \ldots Q^1_k$, so that $L_1 \subset Q^1_1$.   For concreteness, fix an orientation on this scalloped region $S$ so that $a_1$ is the upper right corner of $L$.  Since $S$ can be realized in two ways as a line of lozenges, there exists $L_2 \neq L_1$ in $S$ with $L_1 \cap L_2 \neq \emptyset$, and also with a corner $a_2$ in the upper right.  (In fact, there are infinitely many such choices of $L_2$, we fix one).   See Figure \ref{fig_singular_point_orbit}.

\begin{figure}[h]
   \labellist 
  \small\hair 2pt
 \pinlabel $L_1$ at 85 30
  \pinlabel $a_1$ at 100 80
  \pinlabel $a_2$ at 158 58
  \pinlabel $L_2$ at 120 58
    \pinlabel $a_3$ at 195 52
 \pinlabel $L_3$ at 180 51
 \endlabellist
      \centerline{ \mbox{
\includegraphics[width=10cm]{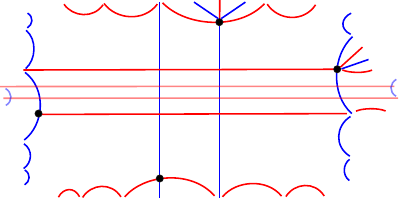}}}
\caption{If $g_1(L_1) = L_2$ then $g$ has a fixed point in $L_1 \cap L_2$}
\label{fig_singular_point_orbit}
\end{figure}

If $a_2 = g_1(a_1)$ for some $g_1 \in G$, and $g_1(L_1) =L_2$ then we already have a contradiction: $g_1(L_1) \cap L_1$ would contain a (necessarily nonsingular) fixed point for $g_1$.
Thus, either $a_2$ is in a distinct $G$-orbit from $a_1$, or we have $a_2 = g_1(a_1)$ but $L_2$ is contained in $g_1(Q^1_j)$ for some $j \neq 1$. 
We now iterate this process: $L_2$ is contained in a scalloped region $S_2$ and we may find a lozenge $L_3$ in $S_2$, with $L_3 \cap L_2 \cap L_1 \neq \emptyset$, and with corner $a_3$ on the upper right.   If we have $a_3 = g_2(a_1)$ (or $a_3 = g_2(a_2)$) for some $g_2 \in G$, and $g_2(L_1) =L_3$ (or $g_2(L_2) =L_3$), then we arrive at the same contradiction.
 
 Iterating this process, we get a sequence $a_k$ of singular points that are corners of lozenges $L_n$ such that $L_1\cap \dots \cap L_n \neq\emptyset$ for any $n$.  
 If there are only finitely many  distinct $G$-orbits of singular points, there is a subsequence $a_{k_i}$ that are all in the same $G$-orbit .  Since $a_{k_i}$ has only finitely many quadrants, 
 we will eventually find some $g = g_{k_j}g_{k_i}^{-1}$ which maps some $L_{k_i}$ to $L_{k_j}$, and so has a nonsingular fixed point, a contradiction.  
\end{proof} 

\begin{proof}[Proof of Theorem \ref{thm:transitive}]
Consider $\Lambda$ a regular Smale class.
We will show that for any open sets $U, V$ such that $U\cap \Lambda \neq \emptyset$ and $V\cap \Lambda \neq \emptyset$, there exists $g\in G$ such that $gU \cap V \cap \Lambda \neq \emptyset$. This implies the existence of a dense orbit (the classical proof for continuous maps applies here, see e.g.~\cite[Lemma 1.4.2]{KH_book}, since regular Smale classes are closed in the complement of the singular leaves in $P$).

Let $a\in U \cap \Lambda$ and $b\in V\cap\Lambda$. Since $a\sim_G b$, there exists $h \in G$ such that $\cF^{+}(a) \cap \cF^-(hb) \neq \emptyset$. If $a = hb$, then we are done. Otherwise, Proposition \ref{prop_basic_product_structure} gives us the existence of a nonsingular point $z$, fixed by some element of $G$ in any arbitrarily small neighborhood of $\cF^{+}(a) \cap \cF^-(hb)$.  In particular, we can assume that $\cF^+(z) \cap U \neq \emptyset$ and that $z$ is chosen close enough to $\cF^{-}(hb)$ so that $\cF^-(z) \cap hV \neq\emptyset $.   If $g$ is the element fixing $z$, then for some $n$ sufficiently large (either positive or negative), we will have $g^n U \cap hV \neq\emptyset$.

Without loss of generality, we can assume that both $U$ and $V$ are chosen small enough so that they are both product foliated. Therefore $g^n U \cap hV$ contains the intersection of $\cF^-(g^n a)$ with $\cF^+(hb)$.  Proposition \ref{prop_basic_product_structure} now implies that $g^n U \cap hV$ must contain points of $\Lambda$.  This shows $G$ acts topologically transitively.

In the special case where there is a unique Smale class $\Lambda$, then $\Lambda$ is a regular Smale class by Lemma \ref{lem_only_prongs}, so Axiom \ref{Axiom_dense} together with Proposition \ref{prop_basic_product_structure} implies that $\bar\Lambda = P$. Hence, $G$ is topologically transitive on $P$ and $P$ contains a dense subset of points fixed by nontrivial elements, i.e., Axioms \ref{Axiom_fixed_points_dense} and \ref{Axiom_topologically_transitive} are both satisfied. 
Conversely, if $\Fix_G$ is dense, then $\cR_G$ is dense (since $\cR_G = \Fix_G\smallsetminus \mathrm{Sing}$, and $\mathrm{Sing}$ is a closed discrete subset of $P$ by Lemma \ref{lem:adjacent_corners_discrete}), moreover any product neighborhood intersects at most one Smale class, so there is a unique Smale class by connectedness of $P\smallsetminus \mathrm{Sing}$. 
\end{proof}

Theorem \ref{thm_dense_orbits_is_transitive} is now an immediate consequence, since lifts of periodic orbits to $\wt M$ project to fixed points for the action of $\pi_1(M)$ on $\orb_\phi$.

\section{Smale chains and the structure of wandering sets} \label{sec:smale_chains}

As described at the start of Section \ref{sec:smale_class}, for Anosov flows on compact manifolds, Smale classes correspond to basic sets.  By Brunella \cite{Brunella} these sets are separated by a collection of embedded tori transverse to the flow.  Moreover, in the nontransitive case any basic set contains ``boundary'' orbits, which are periodic orbits such that (at least) one of their rays are contained in the wandering set (these are the \emph{free separatrices} appearing in \cite[Lemme 1.6]{BB_flots_Smale}).   See also  \cite{BJL} for a closely related theory for Smale diffeomorphisms of surfaces.  

In this and the following sections, we recover this picture for flows from the perspective of the orbit space and generalize it to the setting of Anosov-like actions.
We show that ``boundary leaves'' of Smale classes are fixed by some nontrivial elements of $G$ (Theorem \ref{thm:complement_of_fixbar}), hence recovering the boundary orbits in the Anosov flow setting. We also define {\em Smale chains}, which are the natural generalizations of the projections to the orbit space of the separating transverse tori.  In fact, in the flow setting, any \emph{$\bZ^2$-invariant} Smale chain corresponds to such a torus -- see Corollary \ref{cor:Z2_Smale_chains_in_pA_flows} and Remark \ref{rem:weakly_embedded_to_transverse}.

We will also show that one can always find such  chains ``between'' distinct Smale classes (Theorem \ref{thm:Smale_chains_separate}), i.e., meeting $\cF^+$ leaves of one and $\cF^-$ leaves of another; as these foliations play the roles of the stable/unstable leaves in our setting.  This requires some preliminary work, where we describe in more detail the product structure on Smale classes.  
In Section \ref{sec:Z2invariant}, we prove the existence of Smale chains invariant by a $\bZ^2$-subgroup. 

Note that, by comparison, the first step of the standard proof of the existence of separating transverse tori is the existence of a Lyapunov function for Anosov flows.   Furthermore, as mentioned in the introduction, in the case of pseudo-Anosov flows, here we detect \emph{more} of the tori in the non-wandering set than what the proof using Lyapunov functions would, since Lyapunov functions only allow to separate chain transitive components, and not the components of the non-wandering set. See Propositions  \ref{prop:weird_loops} and \ref{prop_chain-recurrent_non_transitive}.

\subsection{Structure of the regular and ``wandering" sets} 

\begin{definition}
A lozenge $L$ is called \emph{regular} if its interior contains a point of the regular set, i.e. $\mathring{L} \cap \cR_G \neq \emptyset$. 
\end{definition}

\begin{definition}
Recall a \emph{quadrant} of $x \in P$ is a connected component of $P \setminus (\cF^+(x) \cup \cF^-(x))$.  A \emph{TL subquadrant} of $x$ is the subset of a quadrant consisting of the points that are totally linked with $x$. 
 A quadrant, or TL subquadrant, is called {\em regular} if it contains points of $\cR_G$.  
 \end{definition} 
 One description of a lozenge is that it is a TL subquadrant of each of its corners.  Thus, TL subquadrants can be viewed as a generalization of lozenges.
Quadrants, subquadrants, or lozenges are called {\em adjacent} if they share a common leaf segment in their boundary.  

Motivated by the fact that, for Anosov flows, lozenges that are not regular lie in the wandering set, we introduce further terminology: 
 \begin{definition}
We call a lozenge or subquadrant {\em wandering} if it is not regular.
\end{definition} 

\begin{rem}
In this setting, ``wandering" is in fact not an abuse of terminology. Combining
Corollary \ref{cor_boundary_leaves_and_boundary_points}  and Theorem \ref{thm:discrete_stabilizer} 
shows that each point in a wandering lozenge $L$ is in fact wandering in the usual dynamical sense that it has a neighborhood $U$ such that $gU \cap U = \emptyset$ for all nontrivial $g \in G$.  Theorem \ref{thm:discrete_stabilizer} reduces this to considering $g$ that do not stabilize the lozenge $L$; the structure of lozenges implies that if $g(L) \cap L \neq \emptyset$, then either $g(L) = L$, $g$ sends a corner of $L$ into $L$ (which is impossible for wandering lozenges by Corollary \ref{cor_boundary_leaves_and_boundary_points}), or $g$ and $g^{-1}$ contract the intervals of $\cF^\pm$ leaves (respectively) through $L$, giving a fixed point inside $L$.  
\end{rem} 

\begin{lemma}\label{lem:wand_subquadrant_share_sides}
Let $x \in \Fix_G$ be a nonsingular point.  If two non-adjacent TL subquadrants of $x$ are regular, then all TL subquadrants of $x$ are regular.   Equivalently, any wandering TL subquadrant of $x$ has an adjacent TL subquadrant that is also wandering.  
\end{lemma}

\begin{proof}
Suppose $x$ is fixed by $g \neq \id$ and $x$ is nonsingular i.e., it has exactly four quadrants.  Denote the TL subquadrants of $x$ by $Q_1, \ldots, Q_4$ in cyclic order, and suppose without loss of generality $Q_1$ and $Q_3$ are regular.  Let $x_i \in Q_i \cap \Fix(G)$ for $i \in \{1,3\}$ be totally linked with $x$.  (So, in particular, $x_1 \sim_G x_3$). 
Up to replacing $g$ with $g^{-1}$, we will have $\cF^-(g^n(x_1)) \to \cF^-(x)$ as $n \to \infty$ and $\cF^+(g^{n}(x_1)) \to \cF^+(x)$.  Thus, for $N$ large we have 
$\cF^-(g^N(x_1)) \cap \cF^+(x_3) \neq \emptyset$ and $\cF^+(g^{-N}(x_1)) \cap \cF^+(x_3) \neq \emptyset$, and these points will lie in the TL-subquadrants $Q_2$ and $Q_4$ respectively.  Thus, applying Corollary \ref{cor:TL_product}, we conclude $\Fix_G \cap Q_2 \neq \emptyset$ and $\Fix_G \cap Q_4 \neq \emptyset$.  
\end{proof}

Lemma \ref{lem:wand_subquadrant_share_sides} does not hold as stated for singular points, instead, the correct generalization is the following:
\begin{lemma} \label{lem:singular_wandering_quadrants}
Let $c$ be a singular point. If alternating $TL$ subquadrants of $c$ are regular, then \emph{all} its $TL$ subquadrants are regular.

Equivalently, if $Q_0,\dots, Q_{2p-1}$ are the $TL$ subquadrants, cyclically ordered, and $Q_0$ is wandering, then there exists $i$, $0\leq i\leq 2p-1$, such that both $Q_i$ and $Q_{i+1}$ are wandering. 
\end{lemma} 

\begin{proof}
Up to renaming the $TL$ subquadrants, we assume that  $Q_0,\dots, Q_{2p-1}$ are ordered so that $Q_0$ shares a $\cF^+$-side with $Q_1$ and a $\cF^-$-side with $Q_{2p-1}$. If the $Q_{2i-1}$,$i=1,\dots, p$ are all regular, then there exist points $y_i \in \Fix(G) \cap Q_{2i-1}$.   As in the proof of Lemma \ref{lem:wand_subquadrant_share_sides}, after applying a high positive or negative power of the element of $G$ fixing $c$ (which exists by axiom \ref{Axiom_prongs_are_fixed}), we can replace $y_i$ with points $y_i'$ with $\cF^\pm{y_i'} \cap \cF^\mp(c)$ as close as we like to $c$.  By then applying  Corollary \ref{cor:TL_product} using $y_i'$ and $c$, we can obtain points $x_i \in \Fix(G) \cap Q_{2i-1}$ as close as we like to $c$.  In particular, we can take these sufficiently close to $c$ so that $\cF^+(x_i)\cap \cF^-(x_{i+1})\neq \emptyset$ and $\cF^+(x_p)\cap \cF^-(x_1)\neq \emptyset$.   
Let $h_i \neq id$ be an element of $G$ that fixes $x_i$.  Then for appropriate powers $n_i$, we have that $h_{p-1}^{n_{p-1}}\dots h_{2}^{n_{2}} \cF^+(x_1) \cap \cF^-(x_{p})\neq \emptyset$. That is, we are in the setting of Proposition \ref{prop_basic_product_structure}, and therefore there is a point fixed by some nontrivial element of $G$ arbitrarily close to $\cF^+(x_p)\cap \cF^-(x_1) \in Q_0$, contradicting the fact that $Q_0$ is wandering.
\end{proof}

For the next arguments, it will be useful to define certain subsets of TL-subquadrants, that we call TL-substrips:

\begin{definition}\label{def_TLsubstrip}
Let $r^+$ be a ray starting at a point $x$. A \emph{TL-substrip based on $r^+$} is the subset of a TL-subquadrant of $x$ given by its intersection with an infinite strip bounded by $r^+$, a segment of $\cF^-(x)$ and another ray $r'$ of a leaf of $\cF^+$, as in Figure \ref{fig:TLsubstrip}. 
A TL-substrip is called \emph{wandering}, if its interior does not contain any point of $\Fixbar$.
\end{definition}

\begin{figure}[h]
     \centerline{ \mbox{
\includegraphics[width=11cm]{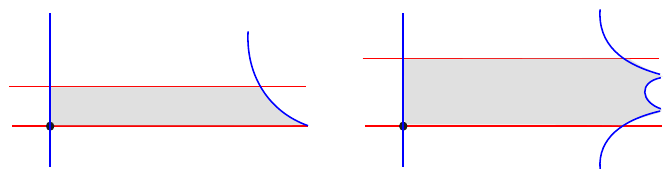} }}
\caption{Two TL substrips}
\label{fig:TLsubstrip}
\end{figure}

Notice that if $Q$ is a wandering TL-subquadrant, then any intersection of $Q$ with an infinite strip is a wandering TL-substrip. However, one a priori may have a wandering TL-substrip based at a ray starting at a point $x$ and such that the corresponding TL-subquadrant of $x$ is \emph{not} wandering.

The next lemma, which is an easy consequence of the fact that there are no infinite product regions, says that up to shrinking the width of a TL-substrip, one can assume that its boundary consists of two rays of leaves $\cF^\pm$, a compact interval in a leaf of $\cF^\mp$ and a unique ray of $\cF^\mp$ making a perfect fit with one of the two rays of $\cF^\pm$:
\begin{lemma} \label{lem:bound_strip}
Let $r_0, r_1$ be rays of $\cF^\pm$ based at a leaf $l$ of $\cF^\mp$ and on the same side of $l$.  
There exists a ray $r_1'$ (possibly equal to $r_1$), between $r_1$ and $r_0$, such that $r_1'$ is on a leaf fixed by some nontrivial element of $G$, and the set of leaves of $\cF^\mp$ intersecting both $r_0$ and $r_1'$ is bounded on one side by $l$, and on the other by a single leaf, which either has a prong singularity between $r_0$ and $r_1'$ 
 or makes a perfect fit with one of $r_0$ and $r_1'$ and intersects the other.
\end{lemma}

\begin{proof}
To fix notation, we suppose  $r_0$ and $r_1$ are rays of $\cF^+$.
If $r_0$ makes a perfect fit with a leaf $l^-$ in the quadrant bounded by $l$ and $r_0$ containing $r_1$, then any ray $r_1'$ sufficiently close to $r_0$ in this quadrant will intersect $l$ and $l^-$, and by Axiom \ref{Axiom_dense}, we can choose this to be a ray of a fixed leaf.  

So we suppose now that $r_0$ does not make a perfect fit in this quadrant.  By Axiom \ref{Axiom_dense}, we may replace $r_1$ with a ray $r_1'$ strictly closer to $r_0$ and so that it lies on a fixed leaf.  

Let $S$ be the set of $\cF^-$ leaves that intersect both $r_0$ and $r_1'$.  Then $S$ is bounded on one side by $l$, and since there are no infinite product regions (Proposition \ref{prop:no_product}), it is bounded on the other side by some leaf or union of leaves.   
We assumed above that there is no perfect fit at $r_0$, and as there cannot be infinitely many leaves between $r_0$ and $r_1'$ (because of Lemma \ref{lem_infinite_line_lozenges}), one of these leaves intersects $r_0$ -- call it $l_0$.  If $l_0$ makes a perfect fit with $r_1'$, we are done.  If $l_0$ intersects $r_1'$, then it is a singular leaf, and some ray of this prong (possibly $r_1'$) intersects $l$; we can replace $r_1'$ with this ray if needed, and again are done.   Finally, the third case is that there is a nontrivial union of leaves in the boundary.  This gives a family of nonseparated leaves, so by Corollary \ref{cor:nonseparated_leaves} they are sides of a line of lozenges, and in particular fixed by some $g \in G$.  Thus, $l_0$ makes a perfect fit with a ray $r$, which is the side of two adjacent lozenges. So $r$ is fixed by $g$ and in between $r_0$ and $r_1'$, as desired.  
\end{proof}

\begin{lemma} \label{lem:substrip_perfect_fit}
Suppose $x \in \Fixbar_G$ has a wandering TL-substrip $S$ based on a ray $r_0$ of a leaf of $\cF^+$.
Then $r_0$ makes a perfect fit with a leaf of $\cF^-$ on the side of the wandering substrip.  
\end{lemma}

\begin{proof}
By Lemma \ref{lem:bound_strip}, we may assume that the boundary of $S$ consists of $r_0$, a ray $r'_1$ of a leaf fixed by some element $g\in G$, a compact interval of a leaf of $\cF^-$ between the two initial points of the rays, and a part of a leaf $l^-\in \cF^-$.  Moreover, Lemma \ref{lem:bound_strip} gives three possibilities for $l^-$. One of the three options is that $l^-$ makes a perfect fit with $r^+$, which is what we want to show. Hence, we have to eliminate the other two.

First, suppose that $l^-$ is a singular leaf with a prong singularity $p$ between $r_0$ and $r'_1$. Then $p$ is TL with $x$, so is in the interior of the TL-substrip $S$, but $p\in \Fix$, contradicting the fact that $S$ is wandering.

Second, suppose that $l^-$ makes a perfect fit with $r_1'$ and intersects $r_0$. Since $r_1'$ is fixed by $g$, this implies that $r_1'$ and a ray of $l^-$ are part of a lozenge $L_1$. Then either one of the corners of $L_1$ is in the TL-substrip $S$, contradicting again the fact that it is wandering, or $x$ is contained in $L_1$. In that latter case, we have $gx\in S$ or $g^{-1}x\in S$, that once again contradicts the fact that $S$ is wandering, since $x\in \Fixbar$.

Thus we eliminated the other cases, and must have that $l^-$ makes a perfect fit with $r_0$, as desired. 
\end{proof} 

Lemma \ref{lem:substrip_perfect_fit} has an important consequence:
\begin{corollary}\label{cor:wandering_then_lozenge}
If $Q$ is a wandering TL subquadrant of a point $x \in \Fix_G$, then $Q$ is a wandering lozenge. 
\end{corollary}
\begin{proof}
A wandering TL subquadrant contains a wandering TL-substrip.  By Lemma \ref{lem:substrip_perfect_fit}, one ray bounding the quadrant makes a perfect fit.   Since $x \in \Fix_G$, Lemma \ref{lem:two_fix_points} item (1) says that these leaves are sides of a lozenge, and thus $x$ is the corner, and $Q$ a lozenge.  
\end{proof} 

\subsection{Boundary leaves} 
Recall that $\Fixbar_G$ has a local product structure (Corollary \ref{cor_product_means_product}), so it makes sense to introduce the following terminology.
\begin{definition}
A \emph{boundary leaf} is a leaf $l\in \cF^\pm$ such that there exists $a\in l\cap \Fixbar_G$ and $a$ is a boundary point of $\Fixbar_G\cap \cF^\mp(a)$. Analogously, one can also define {\em boundary faces} for singular leaves. 
\end{definition}

As noted at the beginning of this section, in the Anosov flow setting, boundary leaves are the stable or unstable leaves of ``boundary orbits'' which are special periodic orbits in each basic set.
Our next goal will be to recover this property in our general setting and understand the structure of $\Fixbar_G$ near such boundary leaves.
We start with the easiest case: that of ``isolated'' boundary leaves, as given by the following proposition. 

\begin{proposition}\label{prop:isolated_fixed_points}
Let $x\in \cR_G$ be such that $\cF^\pm(x)$ is an isolated boundary leaf, i.e., for some small enough product neighborhood $U$ of $x$, we have $\Fixbar_G\cap U \subset \cF^\pm(x)$. Then $x\in \Fix_G$ and it is the corner of four wandering lozenges. In particular, the Smale class of $x$ consists only of $x$ and its images under $G$.

 Similarly, if $p$ is an isolated prong singularity, then all quadrants of $p$ are wandering lozenges.
\end{proposition} 

\begin{proof}
Let $U$ be a product neighborhood of $x$ such that $\Fixbar_G\cap U \subset \cF^+(x)$. (The proof for $\cF^-(x)$ being isolated is obviously symmetric). 
Since $x\in \cR_G$, it is accumulated by elements of $\Fix_G$, but any such element in $U$ must be in $\cF^+(x)$ by assumption. Since any leaf contains at most one element of $\Fix_G$ (see Remark \ref{rem_single_fixed_point}), we deduce that $x\in \Fix_G$. Hence $\Fixbar_G\cap U = \{x\}$, so (by the product structure of $\Fixbar_G$) every TL subquadrant of $x$ are wandering and the conclusion follows from Corollary \ref{cor:wandering_then_lozenge}.

Similarly, if $p$ is an isolated prong singularity, Corollary \ref{cor:TL_product} implies that each of its TL subquadrant must be wandering and the conclusion follows again from  Corollary \ref{cor:wandering_then_lozenge}. 
\end{proof}

The remainder of this section is devoted to the proof of the following theorem, which describes the structure of $\Fixbar_G$ near boundary leaves in the general setting. 

\begin{theorem}[Global structure of $\Fixbar_G$] \label{thm:complement_of_fixbar}
Let $l$ be a leaf of $\cF^\pm$ containing some point of $\Fixbar_G$.  Then, each connected component of $l \setminus \Fixbar_G$ is either bounded and the intersection of $l$ with a maximal line (necessarily finite) of wandering lozenges, or is the (infinite) shared side of two wandering lozenges.
\end{theorem}

 \begin{figure}[h]
   \labellist 
  \small\hair 2pt
 \pinlabel $l$ at 5 75
  \pinlabel $l'$ at 25 100
  \pinlabel $a$ at 40 60
  \pinlabel $l$ at 195 68
    \pinlabel $l'$ at 210 100
    \pinlabel $a$ at 230 50
      \pinlabel $b$ at 320 50
 \endlabellist
     \centerline{ \mbox{
\includegraphics[width=8cm]{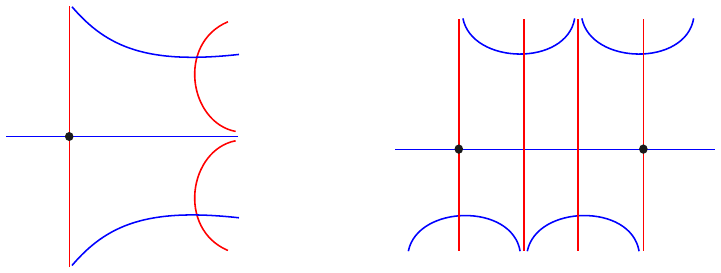} }}
\caption{Two possibilities for boundary leaves $l'$}
\label{fig:newpicture}
\end{figure}
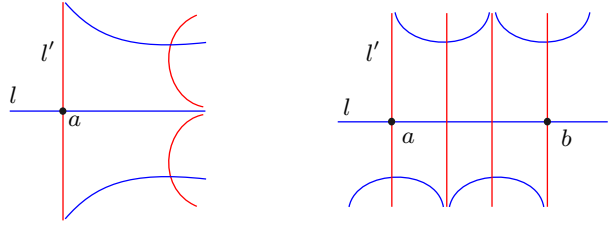

Thanks to Theorem \ref{thm:complement_of_fixbar}, we obtain, as claimed, the generalization of the existence of ``boundary orbits'' to Anosov-like actions.  For convenience, we state this immediate consequence as the following corollary: 

\begin{corollary}\label{cor_boundary_leaves_and_boundary_points}
Any boundary leaf $l'$ is fixed by some nontrivial element $g\in G$ and the fixed point $x$ of $g$ on $l'$ is the corner of (at least) two wandering lozenges. 
\end{corollary}

\begin{rem} 
The situation of Corollary \ref{cor_boundary_leaves_and_boundary_points}, and more generally, much of the work in this section, strongly parallels the detailed structure theory for Smale diffeomorphisms of surfaces developed by Bonatti, Jeandenans and Langevin \cite[Ch.~2]{BJL}.  Our boundary leaves play the role of their boundary points of basic pieces (\emph{points bords d'une pi\`ece basique}), common sides of two wandering lozenges correspond to \emph{free separatrices} (as in \cite{BB_flots_Smale}) 
and extremal sides of maximal lines of wandering lozenges give {\em separatrices coupl\'ees}.   See \cite[Sec.~2.5]{BJL} for further details. 
\end{rem} 

To prove Theorem \ref{thm:complement_of_fixbar}, we prove a sequence of shorter lemmas on the structure of $\Fixbar_G$.  

\begin{lemma} \label{lem:get_wandering_substrip}
Suppose $x \in \Fixbar_G$ and some segment $(x, y)$ of $\cF^+(x)$ lies in the complement of $\Fixbar_G$.  Then 
each ray of the face of $\cF^-(x)$ that bounds the half-plane containing the segment $(x, y)$ is the base of a wandering TL-substrip bounded by $(x,y)$ and a ray of $\cF^-(y)$.  Furthermore, each ray of $\cF^-(x)$ adjacent to the segment $(x, y)$ makes a perfect fit with a $\cF^+$-leaf.

In the case where $y = \infty$, i.e., a full ray of $\cF^+(x)$ lies in the complement of $\Fixbar_G$, then $x \in \Fix_G$ and this ray is the common side of two adjacent wandering lozenges.  
\end{lemma} 

\begin{proof} 
Let $S$ be the TL-substrip based on a ray $r_x$ of $\cF^-(x)$ and bounded also by ray $r_y$ of $\cF^-(y)$.  
Suppose for contradiction there exists some $a \in \Fix_G \cap S$, so $\cF^-(a)$ intersects the segment $(x,y)$ and is totally linked with $x$. Since $x \in \Fixbar_G$, Proposition \ref{prop:product_structure_general} implies that $\cF^-(a)\cap (x,y) \in \Fixbar_G$, a contradiction. 
This, together with Lemma \ref{lem:substrip_perfect_fit}, finishes the proof of the lemma in the case where $y$ is finite, and when $y = \infty$, it shows that the two TL subquadrants of $x$ along the ray $(x,y)$ are wandering.

Assuming $y = \infty$, by the above, both TL subquadrants of $x$ along the ray $(x,y)$ are wandering. In particular, the ray $(x,y)$ is the base of wandering TL-substrips on each sides.
By Lemma \ref{lem:substrip_perfect_fit} there exist leaves $l_1, l_2$ of $\cF^-$ making perfect fits on each side.  These are nonseparated in the leaf space of $\cF^+$, so are fixed by some nontrivial element $g \in G$ (Prop. \ref{prop_axiom4weak_implies_4strong}), forming the sides of adjacent $g$-invariant lozenges (Corollary \ref{cor:nonseparated_leaves}), with common shared side on $\cF^+(x)$. Moreover, these lozenges are wandering, so the common corner on $\cF^+(x)$ is the element in $\Fixbar_G\cap \cF^+(x)$ closest to the perfect fit, i.e., it is $x$.
\end{proof}

\begin{lemma} \label{lem:a_fixed} 
 Let $a\in \Fixbar_G$ be such that $\cF^-(a)$ is fixed by some nontrivial $g \in G$ and $a$ is one end of a connected component $I$ of $\cF^+(a) \setminus \Fixbar_G$. Then, either
\begin{enumerate}
\item $I$ is unbounded, $a$ is fixed by $g$, and $a$ is the corner of two wandering lozenges with $I$ as their shared sides,
\item or $I=(a,b)$ is bounded, $a$ is not fixed by $g$, and both $\cF^-(a)$ and $\cF^-(b)$ contain sides of lozenges in a line of $g$-invariant wandering lozenges which contains the segment $(a,b)$. 
\end{enumerate} 
See Figure \ref{fig:newpicture} for these two configurations.  
\end{lemma} 

\begin{proof}
By Lemma \ref{lem:get_wandering_substrip}, $\cF^-(a)$ makes perfect fits with two leaves $l_1^+, l_2^+$ on the side containing $I$. Moreover, since $\cF^-(a)$ is invariant by $g$, each of the leaves $l_i^+$ are invariant under a power of $g$, and we thus deduce that $\cF^-(a)$ contains the sides of two adjacent lozenges. Moreover, each lozenge is wandering as they contain a wandering TL-substrip.
From there we easily prove the first alternative: If $I$ is unbounded, the conclusion is just Lemma \ref{lem:get_wandering_substrip}.  Additionally, without any assumptions on $I$, if $a$ is fixed by $g$, then $a$ is the corner of the two adjacent wandering lozenges, and hence $I$ is the shared side, thus infinite.

So we now have to prove that if $a$ is not fixed by $g$, or equivalently if $I=(a,b)$ is bounded, we have that 
both $\cF^-(a)$ and $\cF^-(b)$ contain sides of lozenges in a line of $g$-invariant wandering lozenges containing the segment $(a,b)$.

Let $a_1$ be the fixed point of $g$ on $\cF^-(a)$ (so $a_1\neq a$ by assumption).  Then $a_1$ is the corner of two adjacent lozenges (sharing a side along $\cF^+(a_1)$, and each having a ray of $\cF^-(a_1)$ as another side).  Let $L_1$ denote the one of these lozenges which has a side containing $a$.  
Thus, $L_1$ intersects $I$. Call $a_2$ the other corner of $L_1$.  This configuration is illustrated in Figure \ref{fig:a_fixed}.

 \begin{figure}[h]
   \labellist 
  \small\hair 2pt
 \pinlabel $a_1$ at 34 85
 \pinlabel $a$ at 35 58
 \pinlabel $a_2$ at 78 22
 \pinlabel $b_2$ at 95 57
 \pinlabel $b$ at 172 58
 \endlabellist
     \centerline{ \mbox{
\includegraphics[width=7cm]{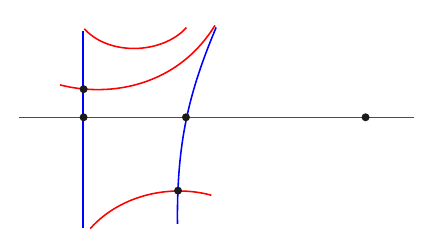} }}
\caption{The lozenge $L_1$ intersecting $I =(a,b)$ with corners $a_1, a_2$.}
\label{fig:a_fixed}
\end{figure}
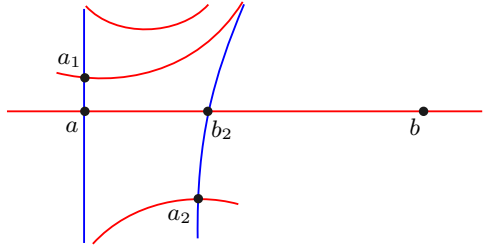

Since $L_1$ is wandering, $\cF^-(a_2)$ intersects $(a,b]$ between $a$ and $b$ (inclusive of $b$).  If it intersects at $b$, we are done. So we assume that the intersection is at a point $b_2\in (a,b)$. Then we deduce that the TL-subquadrant of $a_2$ intersecting $(b_2,b)$ is wandering: Suppose not, so there exists a point $z\in \Fix_G$ that is totally linked with $a_2$. Up to replacing $z$ by $g^nz$, we may then assume that $\cF^-(z)$ intersects $(b_2,b) \subset \cF^+(a)$. But then the product structure of $\Fixbar_G$ (Proposition \ref{prop:product_structure_general}) implies that this intersection point is in $\Fixbar_G$, contradicting the fact that $(b_2,b)$ is disjoint from $\Fixbar_G$. So $a_2$ is the corner of a wandering lozenge $L_2$ adjacent to $L_1$ along the side $\cF^-(a_2)$.  Denote by $a_3$ the opposite corner of $L_2$; we may repeat this process again and find another adjacent lozenge.  It remains only to show this terminates in finitely many steps.

Suppose for contradiction that the process does not terminate.  Then one obtains an infinite line of wandering lozenges between $a$ and $b$.  By Lemma \ref{lem_infinite_line_lozenges}, this infinite line is contained in a bi-infinite line that defines a scalloped region.  Using the fact that a scalloped region can be realized as lines of lozenges in two different ``transverse" ways,  there exists a lozenge $L$, invariant by some element $h\in G$, intersecting all the lozenges $L_i$ and containing in its interior $a$ as well as all the intersections $\cF^+(a)\cap \cF^-(a_i)$. Since $a\in \Fixbar_G$, $L$ is not wandering, and since it is $h$-invariant, we deduce that the segment $(a,b)\cap L$ contains elements of $\Fixbar_G$, a contradiction.
\end{proof} 

Our next lemma reduces the main setting of Theorem  \ref{thm:complement_of_fixbar} to the case treated in Lemma \ref{lem:a_fixed} above.  

\begin{lemma}\label{lem:boundary_fixed}
Suppose $a \in \Fixbar_G$ is a boundary point of $\Fixbar_G \cap \cF^+(a)$ (i.e., on at least one side it is not accumulated along $\cF^+(a)$ by points of $\Fixbar_G$).  Then $\cF^-(a)$ is fixed by some nontrivial element of $G$.   
\end{lemma}

\begin{proof}
Let $I = (a,b)$ be a connected component of the complement of $\Fixbar_G$ in $\cF^+(a)$.  If $b = \infty$, the conclusion follows from the second statement of Lemma \ref{lem:get_wandering_substrip}.  So, we assume $b$ finite.  

By Lemma \ref{lem:get_wandering_substrip} there are leaves $l_1, l_2$ of $\cF^+$ making perfect fits with $\cF^-(a)$ on each side, in the direction of $b$.   
We argue first that (up to switching labels, i.e., changing $a$ to $b$ and the $l_i$ to the corresponding ones for $b$) we have $l_1 \cap \cF^-(b) = \emptyset$.  To see this, 
 suppose first that $l_1$ intersects $\cF^-(b)$. Then we can apply the same argument with the roles of $a$ and $b$ reversed (since $b \in \Fixbar_G$ also), and find leaves $l_3, l_4$ making perfect fits in the direction of $a$, and the fact that $l_1 \cap \cF^-(b) \neq \emptyset$ prevents at least one of $l_3, l_4$ from intersecting $\cF^-(a)$.  Thus, up to switching the roles of $a$ and $b$ we can assume that $l_1$ does not intersect $\cF^-(b)$. Moreover, since there are not totally ideal quadrilaterals (Axiom \ref{Axiom_totallyideal}), we can further assume that $l_1$ does not make a perfect fit with $\cF^-(b)$ either.  See Figure \ref{fig:ideal_quad}.
 
 \begin{figure}[h]
   \labellist 
  \small\hair 2pt
 \pinlabel $\cF^-(a)$ at -10 40
 \pinlabel $\cF^-(b)$ at 40 40
 \pinlabel $l_1$ at 30 70
 \pinlabel $l_2$ at 30 10
 \pinlabel $l_1$ at 130 80
 \pinlabel $l_2$ at 137 10
 \endlabellist
     \centerline{ \mbox{
\includegraphics[width=5.5cm]{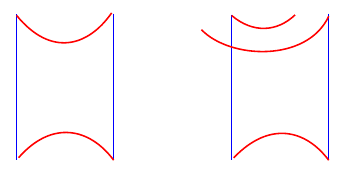} }}
\caption{If $l_1$ makes a perfect fit with $\cF^-(b)$, we obtain an ideal quadrilateral (left).  The allowed configuration is shown to the right.}
\label{fig:ideal_quad}
\end{figure}
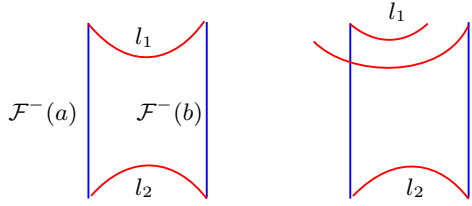

Consider the set $S$ of leaves intersecting both $l_1$ and $l_2$.  
As in the proof Lemma \ref{lem:bound_strip}, $S$ is bounded on one side by $\cF^-(a)$ and, since there are no infinite product regions, it is bounded on the other side by a leaf or union of leaves.   
Suppose first that this boundary is a nontrivial union of leaves or a pronged leaf, then they are fixed by some nontrivial element $h\in G$. Consider the fixed points of $g$ in these leaves. Either there is one $x$ fixed by $h$ that is in between $l_1$ and $l_2$ (this is automatic in the single singular leaf case), we get some point $x$ fixed by some nontrivial $h \in G$ such that $\cF^+(x) \cap \cF^-(a) \neq \emptyset$, and $\cF^-(x) \cap (a,b) \neq \emptyset$. If not, consider $y$ the fixed point of $h$ on the leaf intersecting, say $l_1$. It is the corner of a lozenge whose other corner $x$ is in between $l_1$ and $l_2$. Again we deduce that $\cF^+(x) \cap \cF^-(a) \neq \emptyset$, and $\cF^-(x) \cap (a,b) \neq \emptyset$.    Then the product structure of $\Fixbar_G$ (Proposition \ref{prop:product_structure_general}) implies that $\cF^-(x) \cap (a,b) \in \Fixbar_G$, contradicting the hypotheses.   

Thus, it remains to treat the case where $S$ is bounded by a single leaf $l^-$, making a perfect fit with either $l_1$ or $l_2$. (This latter case can happen only when $l_2$ does not intersect nor make a perfect fit with $\cF^-(b)$). Then, up to relabelling, we may assume that $l^-$ makes a perfect fit with $l_1$.  Then $l^- = \cF^-(c)$ for some $c \in (a, b)$\footnote{Axiom \ref{Axiom_totallyideal} is used only here in this article, to ensure that $c\neq b$.}  

This configuration will allow us either to conclude the proof directly, or to find a leaf $l^+$ making a perfect fit with $\cF^-(c)$ and nonseparated with $l_1$ (which will then give a contradiction).  To this end, let $r_0$ denote the ray of $\cF^-(c)$ making a perfect fit with $l_1$.  By Lemma \ref{lem:bound_strip}, there exists $d \in (c, b)$ on a leaf fixed by some nontrivial $g \in G$ such that the set of leaves intersecting both $\cF^-(d)$  and $r_0$ is bounded by a leaf making a perfect fit with $r_0$ (as we desired), or a leaf with a prong singularity between $\cF^-(c)$ and $\cF^-(d)$, or a leaf making a perfect fit with $\cF^-(d)$ and intersecting $r_0$.  
The prong case cannot occur because the product structure of $\Fixbar_G$ (Proposition \ref{prop:product_structure_general}) and the fact that $a \in \Fixbar_G$ would imply that some point of $(c,d)$ lies in $\Fixbar_G$, contradicting our hypothesis.  

It remains to treat the case of a leaf $l$ making a perfect fit with $\cF^-(d)$, as illustrated in Figure \ref{fig:414end}.  Since $\cF^-(d)$ is fixed by $g$, the existence of such a perfect fit implies that $\cF^-(d)$ is the side of a lozenge.  We have $l \cap \cF^-(a) \neq \emptyset$, so 
either there is a corner of the lozenge on $l$ between $\cF^-(a)$ and $\cF^-(d)$ (which would be totally linked with $a$ and hence by the product structure give a point of $\Fixbar_G$ between $a$ and $d$, impossible), or $a$ lies on the side of the lozenge, or $a$ lies inside the lozenge.  

 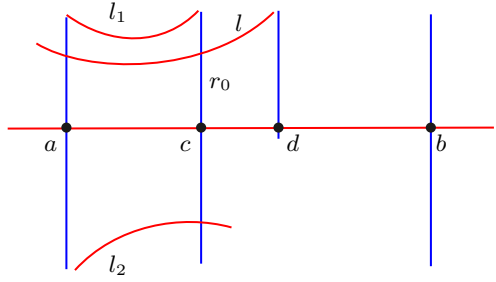
\begin{figure}[h]
   \labellist 
  \small\hair 2pt
 \pinlabel $a$ at 25 55
 \pinlabel $c$ at 75 55
 \pinlabel $d$ at 115 56
 \pinlabel $r_0$ at 88 78
 \pinlabel $l$ at 95 101
 \pinlabel $b$ at 170 56
\pinlabel $l_1$ at 50 105
\pinlabel $l_2$ at 50 10
 \endlabellist
     \centerline{ \mbox{
\includegraphics[width=7cm]{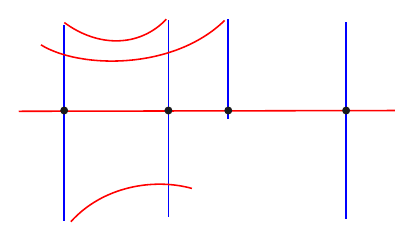} }}
\caption{Configuration of leaves in the proof of Lemma \ref{lem:boundary_fixed}}
\label{fig:414end}
\end{figure}

 We can easily eliminate these later two cases as follows: 

If $a$ lies on the side of the lozenge, we have finished the proof, since this implies that $\cF^-(a)$ is fixed by (a power of) $g$.  If instead $a$ lies inside the lozenge, since $a \in \Fixbar_G$, the product structure implies that $d \in \Fixbar_G$, which is again a contradiction.  Thus, we have either finished the proof,  or there is a leaf $l^+$ making a perfect fit with $\cF^-(c)$ and nonseparated with $l_1$.  But this latter possibility leads to a contradiction:  if $l_1$ is nonseparated with $l^+$, then $l_1$ is fixed by some nontrivial $h \in G$ so contains a fixed point for $h$; and since $a \in \Fixbar_G$, we can again use the product structure to find a point of $(a,c)$ in $\Fixbar_G$, a contradiction. 

That final contradiction shows that $\cF^-(a)$ is fixed by $g$. However, recall that at the beginning of the argument, we may have renamed $a$ and $b$. So, to finish the proof of the lemma as stated, we need to justify that $b$ is also such that $\cF^-(b)$ is fixed by $g$. This is now given by Lemma \ref{lem:a_fixed}.
\end{proof} 

\begin{rem}\label{rem_wandering_totally_ideal}
Notice that for the above argument to hold, we do not need the full strength of Axiom \ref{Axiom_totallyideal}, but instead, we only need the fact that there are no \emph{wandering} totally ideal quadrilaterals, i.e., a totally ideal quadrilateral that does not contain any element of $\Fixbar_G$. We will see in Theorem \ref{thm:wandering_quadrilateral} that one can indeed build examples of actions satisfying Axioms \ref{Axiom_A1}--\ref{Axiom_non-separated}, but that admit wandering totally ideal quadrilaterals. 
\end{rem}

Finally, we can simply assemble the work done in these technical lemmas to prove Theorem \ref{thm:complement_of_fixbar}.
\begin{proof}[Proof of Theorem \ref{thm:complement_of_fixbar}]
Let $l$ be a leaf of $\cF^\pm$ containing some point $x \in \Fixbar_G$, and 
let $I$ be a connected component of $l \setminus \Fixbar_G$. 
If $I$ is unbounded, Lemma \ref{lem:get_wandering_substrip} shows that $x \in \Fix_G$ and $I$ is the side of a wandering lozenge.  

If $I$ is bounded, say equal to $(a,b)$ (with one of $a, b$ equal to $x$) Lemma \ref{lem:boundary_fixed} shows that one endpoint of $I$ lies on a leaf fixed by some nontrivial element of $G$.  Applying Lemma \ref{lem:a_fixed} (reversing the roles of $\cF^+$ and $\cF^-$ if needed) then shows that 
both $\cF^\mp(a)$ and $\cF^\mp(b)$ are fixed by a nontrivial element of $g \in G$ and form sides of a (finite) line of $g$-invariant wandering lozenges.  
This line of wandering lozenges is necessarily maximal, because $a$ and $b$ lie in $\Fixbar_G$ so cannot be contained in the side of an adjacent wandering lozenge.  
\end{proof}

\subsection{Smale Chains} 
We now reinterpret the result obtained in Theorem \ref{thm:complement_of_fixbar} to show the existence of special chains of wandering lozenges that ``separate'' distinct Smale classes. In the Anosov flow setup, the work we are going to do here corresponds to going from the existence of periodic boundary orbits in basics sets to showing that one of their leaves crosses one of the transverse tori that separate basic sets.

\begin{definition}\label{def_Smale_chain} 
A {\em Smale chain} is a bi-infinite chain of lozenges $\cW = \{L_i\}_{i\in \bZ}$, with the following properties:
\begin{enumerate}[label=(\roman*)]
\item\label{item_Smale_wandering} each lozenge is wandering,
\item\label{item_Smale_share_side_or_prong} for all $i$, $L_i$ and $L_{i+1}$ either share a side, or share a \emph{singular} corner,
\item\label{item_Smale_no_triple_shared_corners} no three lozenges in $\cW$ share the same corner.
\end{enumerate} 
\end{definition}
Condition \ref{item_Smale_no_triple_shared_corners} ensures that Smale chains are \emph{minimal} with respect to inclusion. For bi-infinite chains, this is in fact equivalent:
\begin{lemma}\label{lem_minimal_equivalent_no_3_shared_corners}
A bi-infinite chain of lozenges $\cC$ is minimal with respect to inclusion if and only if no three lozenges in $\cC$ share the same corner.
\end{lemma} 

\begin{proof}
For each lozenge in bi-infinite chain of lozenges $\cC$, choose a path from one corner to the other. Then $\cC$ is homotopic to the concatenation, call it $T$, of these paths. Moreover, it is easy to see that $T$ is a tree, as in \cite[Prop. 2.12]{BarbFen_pA_toroidal}. Minimality of $\cC$ is equivalent to the property that removing any lozenge of $\cC$ disconnects $\cC$, which is equivalent to the fact that $T$ must be homeomorphic to $\bR$, which in turn is equivalent to the fact that no three lozenges in $\cC$ share a corner.
\end{proof}

\begin{proposition}\label{prop:Smale_chain}
Any finite chain of lozenges satisfying conditions \ref{item_Smale_wandering}--\ref{item_Smale_no_triple_shared_corners} above can be extended to a Smale chain.
\end{proposition}
We will mostly use the proposition above to finite \emph{lines} of wandering lozenges. The proposition applies, since lines of lozenges trivially satisfy the additional conditions \ref{item_Smale_share_side_or_prong} and \ref{item_Smale_no_triple_shared_corners}.

\begin{proof}
We show first how to extend a single wandering lozenge to an infinite (but not yet bi-infinite) chain of wandering lozenges  also satisfying conditions \ref{item_Smale_share_side_or_prong} and \ref{item_Smale_no_triple_shared_corners}.  
Let $L_0$ be a wandering lozenge with corners $c_0$ and $c_1$.  
Suppose first that $c_1$ is nonsingular.  Lemma \ref{lem:wand_subquadrant_share_sides} says that $L_0$ has an adjacent TL subquadrant that is wandering and Corollary \ref{cor:wandering_then_lozenge} says this is a wandering lozenge.  We then  denote by $L_1$ this wandering lozenge sharing a side with $L_0$ and sharing corner $c_1$.  
If instead $c_1$ is singular then Lemma \ref{lem:singular_wandering_quadrants} with Corollary \ref{cor:wandering_then_lozenge} implies that there is another wandering lozenge, different from $L_0$, sharing the corner $c_1$. In this case we call $L_1$ that wandering lozenge.
  
In either case, let $c_2$ denote the opposite corner of $L_1$.   Repeating this process inductively, we can produce an infinite sequence of corners $c_1, c_2, c_3, \ldots$ such that $c_i$ and $c_{i+1}$ are the two corners of a wandering lozenge $L_i$. 
 Note particularly that, with our construction, the corners $c_i$ are all distinct, hence we will never have three of the lozenges in our chain sharing a common corner.

Now suppose that $L_0$ is a lozenge in a chain of lozenges $L_{-k}, L_{-(k-1)}, \ldots L_0$ (possibly with $k=0$), that satisfy conditions  \ref{item_Smale_wandering}--\ref{item_Smale_no_triple_shared_corners}.
Up to relabelling we assume that $c_1$ is the corner of $L_0$ that is not shared with $L_{-1}$.  
Let $c_i$ denote the corner shared by $L_{i}$ and $L_{i+1}$, and let $c_{-k}$ denote the remaining corner (of $L_{-k}$) at the end of the line.  Apply the argument above to the corner $c_{-k}$, to find  a lozenge $L_{-k-1}$ sharing corner $c_{-k}$, and if $c_{-k}$ is not singular it will share also a side with $L_{-k}$.  Let $c_{-k-1}$ denote the opposite corner of $L_{-k-1}$.  Again, proceeding inductively, we obtain a bi-infinite chain $\cW = \{L_i\}_{i\in \bZ}$ of lozenges with corners $c_i$ and $c_{i+1}$; all distinct. In particular, $\cW$ satisfies all conditions of \ref{def_Smale_chain}, thus $\cW$ is a Smale chain. 
\end{proof}

Combining this with the work done in the previous section, we can prove that Smale chains exist and ``separate" distinct Smale classes, in the following sense:
\begin{theorem}\label{thm:Smale_chains_separate} 
Suppose $x$ and $y$ are in distinct Smale classes and $\cF^+(x) \cap \cF^-(y) \neq \emptyset$.  Then there exists a Smale chain $\cW$ such that $\cF^+(x) \cap \cF^-(y)$ is contained in $\cW$ and the closure of the Smale class of $x$ contains a corner of $\cW$.
\end{theorem}
Note that, if the lozenges of $\cW$ have nonsingular corners, then we may of course find a corner in the Smale class of $x$, the statement about the {\em closure} of the Smale class is simply to include singular points if needed.  

\begin{proof} 
Suppose first that $x$ is an isolated prong singularity. Then by Proposition \ref{prop:isolated_fixed_points}, $x$ has a wandering lozenge in each of its quadrants. 
Thus, $\cF^+(x) \cap \cF^-(y)$ is on a side shared by two wandering lozenges with corner $x$.  Proposition \ref{prop:Smale_chain} says this pair of lozenges may be extended to a Smale chain, and the conclusion follows.  

The same argument applies if $y$ is an isolated prong singularity. So we now assume that both $x$ and $y$ are in $\cR_G$.

Since $x \nsim_G y$, we have that  $p:= \cF^-(y) \cap \cF^+(x) \notin \Fixbar_G$. 
Consider the connected component $I$ of $\cF^+(x) \setminus \Fixbar_G$ containing $p$.  If this is infinite, then by Lemma \ref{lem:get_wandering_substrip}, the endpoint of $I$ is the corner of a pair of wandering lozenges with $p$ on its shared side.  By Proposition \ref{prop:Smale_chain}, this pair can be included into a Smale chain, and we are done.  

If instead, the connected component is finite, then Theorem \ref{thm:complement_of_fixbar} says that $I$ is equal to the intersection of a (finite) line $\cL$ of wandering lozenges with $\cF^+(x)$.  Proposition \ref{prop:Smale_chain} again says that this can be included into a Smale chain.  
Finally, the product structure on Smale classes (Proposition \ref{prop:product_structure_general}) implies that the corner of $\cL$ closest to $x$ along $\cF^+(x)$ lies in the closure of the Smale class of $x$. 
\end{proof}

\section{Virtually cyclic point stabilizers and $\mathbb{Z}^2$ invariant Smale chains} \label{sec:Z2invariant}

In this section we prove discreteness of point stabilizers, then use this to find $\bZ^2$-invariant Smale chains.  First we recall the statement: 

\poinstabilizers*

\begin{rem}
Theorem \ref{thm:discrete_stabilizer} does not hold in the $\R$-covered setting.  Example 1 of \cite{BFM} gives Anosov-like affine actions on trivial planes with indiscrete stabilizers.  On skew planes, one may create such examples by taking subgroups of $\mathrm{PSL}(2,\R)$ with non-discrete point stabilizers (but no parabolic elements), and lifting them to $\wt{\mathrm{PSL}}(2,\R)$.  The action of $\wt{\mathrm{PSL}}(2,\R)$ induces an action on a skew plane, and such a subgroup without parabolics will induce an Anosov-like action on the plane.  
\end{rem}

For the proof we use the following result. 

\begin{lemma} \label{lem:cyclic_or_accumulation}
Let $H_x$ denote the finite index subgroup of the stabilizer of $x$ fixing each ray of $\cF^\pm(x)$.  If $H_x$ is not cyclic, then for any point $p$ in any TL subquadrant of $x$, the orbit $H_x(p)$ has an accumulation point. 

Moreover, one can find such an accumulation point $q$ such that $x$ lies in one quadrant $Q$ of $q$ and $H_x(p)$ accumulates onto $q$ in {\em both} quadrants adjacent to $Q$.  
\end{lemma}

\begin{proof}
Fix $x \in P$ and assume $H_x$ is nontrivial.  Consider two half leaves $r^+$ and $r^-$ bounding a quadrant of $x$.  
Choose some nontrivial $h_0 \in H_x$, acting by expansion on $r^+$ and contraction on $r^-$.  
Fix an identification of each half leaf $r^\pm$ with $\R$, oriented and scaled so that $h_0$ acts by $t \mapsto t+1$ on each.
With this choice, H\"older's theorem (see e.g., \cite[Section 2.2.4]{Navas_book})  shows that there are canonical, well-defined, monomorphisms $\Phi^\pm\colon H_x \to \R$ given by the {\em translation number} (with respect to $h_0$) of elements on each factor; so that $\Phi^+(h_0) = \Phi^-(h_0)=1$.  Furthermore, the actions of $H_x$ on $r^+$ and on $r^-$ are semiconjugate to the translation actions of $\Phi^+(H_x)$ and $\Phi^-(H_x)$, respectively.  

Note that because of our choice of normalization, an element $h\in H_x$ is a topological expansion on $r^+$ if and only if $\Phi^+(h) >0$, and a contraction on $r^-$ if and only if $\Phi^-(h) > 0$.  
Consider the map $\Phi^- \circ (\Phi^+)^{-1}$.  This is a homomorphism between subgroups of $\R$, and the hyperbolicity condition (Axiom \ref{Axiom_A1}) implies that $\Phi^- \circ (\Phi^+)^{-1}(r) > 0$ iff $r>0$. 
Thinking of $\R$ as a vector space over $\mathbb{Q}$, it is an easy exercise to show that this can only hold if $\Phi^- \circ (\Phi^+)^{-1}(r)$ is given by multiplication by some positive constant $c$; since $\Phi^+(h_0) = \Phi^-(h_0) = 1$, we also have $c = 1$.  
Thus, $\Phi^+ = \Phi^-$.  

Consider any point $p$ in the $TL$ subquadrant of $x$ bounded by $r^\pm$.  We can use the identifications of $r^\pm$ with $\R$ above to give coordinates on the TL subquadrant containing $p$; in these coordinates the action of $H_x$ is (factor-wise) semi-conjugate to a diagonal action by translations.  If $H_x$ is not cyclic, this translation subgroup is an indiscrete subgroup of $\R$ and so the closure of the orbit of $p$ under the diagonal action in these coordinates will have an invariant minimal cantor set.  Any two-sided accumulation point of the cantor set will give an accumulation point $q$ as desired.  
\end{proof}

\begin{proof}[Proof of Theorem \ref{thm:discrete_stabilizer}] 
Fix $x \in P$, and let $H_x$ be the subgroup of the stabilizer of $x$ fixing each ray as above.  First, we reduce to the case where $x$ has no lozenge in some quadrant.  

If $x$ has a lozenge in each quadrant, consider the maximal chain of lozenges $\cC$ invariant under $H_x$.  If $y$ is another corner of this chain, then $H_y = H_x$.  Thus, either we can replace $x$ with some other corner $y$ which has no lozenge in some quadrant, or every corner has a lozenge in each quadrant.  In particular, this implies that $\cC$ contains an infinite line of lozenges, so by Lemma \ref{lem_infinite_line_lozenges}) the stabilizer of the corners of $\cC$ is cyclic. 

So now we assume some $TL$ subquadrant $Q$ of $x$ is not a lozenge.  
Let $D=\mathrm{Sing}\cup \mathrm{Pivot}$ denote the set of prongs and pivots.  Define also 
\[ \mathrm{Opp} := \{p \in P :  p \text{ is a corner of a lozenge whose other corner is a prong or a pivot}\} 
\]
By Lemma \ref{lem:adjacent_corners_discrete}, $D$ is closed and discrete. 

We will first show that we can always find a point of $D \cup \mathrm{Opp}$ that is totally linked with $x$. Then, we will use Lemma \ref{lem:cyclic_or_accumulation} to deduce that $H$ is cyclic.  

The following claim reduces this goal to finding a point of $\mathrm{Sing}\cup \mathrm{Pivot}$ that is partially linked with $x$ along one of the rays bounding the non-lozenge quadrant.  

\begin{claim} \label{claim:TLx}
Suppose $x$ is not the corner of a lozenge in the quadrant $Q$ bounded by rays $r^+$ and $r^-$, and suppose some $y \in \mathrm{Sing}\cup \mathrm{Pivot}$ satisfies $\cF^+(y) \cap r^- \neq \emptyset$ or $\cF^-(y) \cap r^+ \neq \emptyset$. Then either $x$ is totally linked with a point of $\mathrm{Sing}\cup \mathrm{Pivot}$, or totally linked with a point of $\mathrm{Opp}\cap Q$.
\end{claim}

\begin{proof}
For concreteness, assume $\cF^+(y) \cap r^- \neq \emptyset$.  As in the proof of Lemma \ref{lem:bound_strip}, we start by considering the set of leaves of $\cF^-$ that intersect both $\cF^+(y)$ and $r^+$.  Since the plane is nontrivial, by Proposition \ref{prop:no_product} it contains no infinite product regions, so $I$ is a proper interval in the leaf space, bounded on one side by $\cF^+(y)$ and on the other side either by leaves of a prong (with the singularity between $\cF^+(y)$ and $r^+$) or by a unique leaf making a perfect fit with $r^+$ or with $\cF^+(y)$, or by a nontrivial union of nonseparated leaves.   We treat these cases separately. 

First, in the prong case, such a prong singularity is necessarily totally linked with $x$, and we are done. Next, no leaf making a perfect fit with $r^+$ can be in the quadrant under consideration, since $x \in \Fix(G)$ is by assumption not a corner of a lozenge in that quadrant. So no boundary leaf of $I$ can make a perfect fit with $r^+$.  If there is a bounding leaf $l$ that makes a perfect fit with $\cF^+(y)$ and intersects $r^+$, then because $y \in \mathrm{Sing}\cup \mathrm{Pivot}\subset \Fix(G)$, this perfect fit is an ideal corner of a lozenge $L$ (Lemma \ref{lem:two_fix_points}).   Then, we have three possible configurations for $L$: either $L$ contains $x$, or it intersects $r^+$ or $r^-$. If $x$ is in $L$, it is TL with both corners, one of which is in $\mathrm{Sing}\cup \mathrm{Pivot}$, so we are done.
Otherwise, $x$ is is TL with the corner of $L$ that sits in the quadrant $Q$, which is a point in $\mathrm{Pivot}\cup\mathrm{Sing}\cup (\mathrm{Opp}\cap Q)$, as claimed.

Thus it remains to consider the case where $I$ is bounded by a nontrivial union $l_0^-, l_1^- \ldots l_k^-$ of nonseparated leaves (none of which makes a perfect fit with $r^+$).   Enumerate these in order so that $l_0^-$ intersects $r^+$ and $l_k^-$ intersects or makes a perfect fit with $\cF^+(y)$.  See Figure \ref{fig:discrete_stabilizer} for an illustration. If $k=0$, then we are in the case already treated above, so we may assume $k>0$.
By Proposition \ref{lem:two_fix_points}, some nontrivial $g \in G$ fixes all leaves $l_i^-$ and so two rays in $l_0^-$ and $l_1^-$ are sides of adjacent lozenges.  Call these $L_0$ and $L_1$, respectively.  

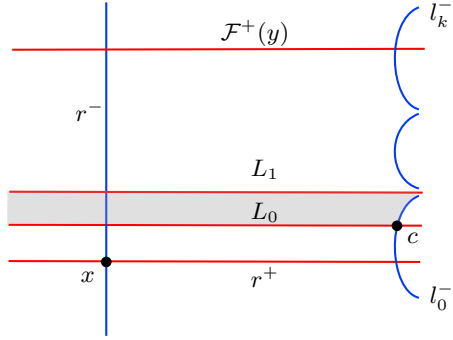
\begin{figure}
   \labellist 
  \small\hair 2pt
 \pinlabel $x$ at 50 45
 \pinlabel $\cF^+(y)$ at 135 170
 \pinlabel $l_k^-$ at 230 180
 \pinlabel $l_0^-$ at 230 35
  \pinlabel $r^+$ at 140 45
 \pinlabel $c$ at 215 65
 \pinlabel $r^-$ at 51 130
 \pinlabel $L_0$ at 140 78
  \pinlabel $L_1$ at 140 100
  \endlabellist
     \centerline{ \mbox{
\includegraphics[width=7cm]{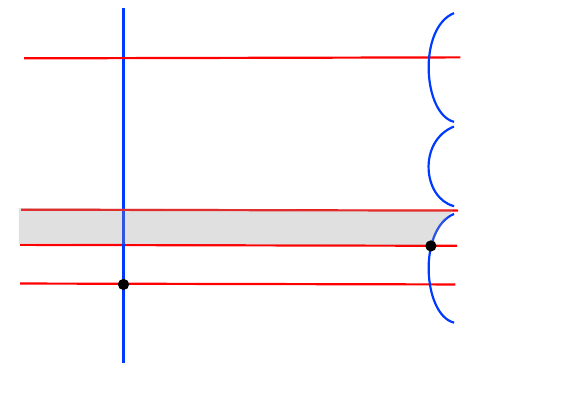} }}
\caption{Set up for the last case of Claim \ref{claim:TLx}. A corner of $L_0$ is always totally linked with $x$, either configured as above, or if $\cF^-(c)$ lies below $x$.}
\label{fig:discrete_stabilizer}
\end{figure}
  
Let $c$ be the corner of $L_0$ on $l_0^-$.   If $x \in L_0$, then $x$ is totally linked with both corners of $L_0$, one of which is in $\mathrm{Pivot}$ since $c$ is on a nonseparated leaf, so we are done.

Otherwise $x\notin L_0$, which implies that $c$ is in $Q$, and since $\cF^+(c) \cap r^-\neq \emptyset$, $x$ is totally linked with $c \in  \mathrm{Pivot}\cup (\mathrm{Opp} \cap Q)$.  
Thus, we have in all cases found that $x$ is TL with a point in $\mathrm{Pivot}\cup \mathrm{Sing}$ or in $\mathrm{Opp} \cap Q$, completing the proof of the claim. 
\end{proof} 

Next, we show that the condition of Claim \ref{claim:TLx} is always satisfied, so that we can always find a point in $D\cup \mathrm{Opp}$ that is totally linked with $x$.  This improves Claim \ref{claim:TLx}  to the following statement.  
\begin{claim}\label{claim_TL_with_pivot_or_Opp}
Suppose $x$ is not the corner of a lozenge in the quadrant $Q$ bounded by rays $r^+$ and $r^-$. Then $x$ is TL with a point in $\mathrm{Pivot}\cup \mathrm{Sing}$ or in $\mathrm{Opp} \cap Q$.
\end{claim}

\begin{proof}
We treat separately the case where $G$ acts topologically transitively or not.

Suppose first that $G$ acts topologically transitively on $P$.  Since  $P$ is not $\bR$-covered, there is a branching leaf or prong.  By \cite[Lemma 2.21]{BFM}, every leaf of $\cF^{\pm}$ has a dense orbit under $G$.  Thus, some branching leaf or prong leaf must intersect $r^+$, and Claim \ref{claim:TLx} applies to give the conclusion.

Now suppose instead that $G$ has a nontransitive Anosov like action.   We first make a short argument that will allow us to treat only the case where $x$ is nonsingular.   Recall that $x$ is fixed by some element, and $Q$, the $TL$-subquadrant bounded by $r^{\pm}$ was assumed not to be a lozenge. Hence it is a regular subquadrant, and therefore we can find $x'$ a nonsingular point in $Q$ totally linked with $x$ and such that $x'$ is fixed by some nontrivial element $g'$ of $G$. Hence if we can find $y\in \mathrm{Pivot}\cup\mathrm{Sing}$ such that $\cF^+(y)\cap \cF^-(x')\neq\emptyset $ (resp.~$\cF^-(y)\cap \cF^+(x')\neq\emptyset $) then for some large enough power $n$, we have $\cF^+((g')^{n}y)\cap r^-\neq\emptyset $ (resp.~$\cF^-((g')^{n}y)\cap r^+\neq\emptyset $). 

Thus we can now assume that $x$ is nonsingular and just find an element of $y\in \mathrm{Pivot}\cup\mathrm{Sing}$ such that at least one of $\cF^{\pm}(y)$ intersects $\cF^{\mp}(x)$.

Since we assumed the action of $G$ was nontransitive, by Theorem \ref{thm:transitive}, there exists an element $y'$ in a distinct Smale class from $x$.  
By Theorem \ref{thm:Smale_chains_separate}, there exists a Smale chain $\cW$ intersecting either $\cF^+(x)$ of $\cF^-(x)$ (or, in general, if $x$ is neither a maximal nor minimal Smale class, then we will find examples of both). In particular, there exists a corner $y$ on $\cW$ with $\cF^\pm(y)\cap \cF^\mp(x)\neq\emptyset$.
Since all corners of a Smale chain are either pivots or singular points, the hypothesis for Claim \ref{claim:TLx} holds, and implies our claim.
\end{proof}

We are now able to finish the proof of Theorem \ref{thm:discrete_stabilizer}.
By Claim \ref{claim_TL_with_pivot_or_Opp},  $x$ is TL with an element of $\mathrm{Pivot}\cup \mathrm{Sing} \cup (\mathrm{Opp} \cap Q)$, where $Q$ is a quadrant of $x$ that does not contain a lozenge with corner $x$. We treat two cases separately.

First assume that $x$ is TL with an element of $\mathrm{Pivot}\cup \mathrm{Sing}$. Since $\mathrm{Pivot}\cup \mathrm{Sing}$ is closed and discrete (by Lemma \ref{lem:adjacent_corners_discrete}), Lemma \ref{lem:cyclic_or_accumulation} directly implies that the stabilizer of $x$ is discrete.

Thus, we only have left to show the theorem when $x$ is TL with an element $y$ of $\mathrm{Opp} \cap Q$. If the stabilizer $H_x$ of $x$ is not cyclic, applying Lemma \ref{lem:cyclic_or_accumulation}, we obtain that there exists a sequence $g_n\in H_x$ such that  $y_n:= g_n y$ converges to a point $q \in Q$ and TL with $x$. Moreover, we can assume that $q$ is accumulated on two of its opposite quadrants by the $y_n$ (the two quadrants of $q$ that are adjacent to the one containing $x$).
Recall that $y_n\in \mathrm{Opp}$, so $y_n$ is the corner of a lozenge $L_n$ with opposite corner $z_n\in \mathrm{Pivot}\cap \mathrm{Sing}$. Notice that if $z_n$ is TL with $x$, then we directly get a contradiction from the discreteness of $\mathrm{Pivot}\cap \mathrm{Sing}$.

If instead $z_n$ is not TL with $x$, there are three possibilities: 
\begin{enumerate}
\item Either $L_n\subset Q$,
\item or $L_n$ intersects the ray $r^-$ of $\cF^-(x)$ bounding $Q$, but not $\cF^-(x)$,
\item or $L_n$ intersects the ray $r^+$ of $\cF^+(x)$ bounding $Q$, but not $\cF^-(x)$,
\end{enumerate}

Since $L_n = g_n L$ and $g_n\in H_x$ (which preserves local orientations), one of the three cases above occurs for one $n$ if and only if it occurs for all $n$.
Now we will find a contradiction in each of these situations.  For concreteness, we fix local orientations so that $Q$ is the top right quadrant of $x$ with $r^+$ the positive horizontal axis and $r^-$ the vertical one.

First consider case 1: $L_n$ is contained in $Q$. For our choice of orientations, it means that $L_n$ is always in the top right quadrant of the corner $y_n$.
Since $y_n$ accumulates onto $q$ from two opposite quadrants, we can assume, up to renaming, that $y_1 = gy_0$ is to the bottom right of $y_0$, and inside a small trivially foliated neighborhood of $q$, and $g= g_1g_0^{-1}\in H_x$. Then we obtain that $L_1 = g L_0$ must be contained in the $\cF^-$-saturation of $L_0$ as shown in Figure \ref{fig:opp_lozenges}. 
It follows that $g$ contracts the interval of $\cF^-$ leaves passing through $L_0$.  
Thus, $g$ must have a fixed point in $L_0$. But $L_0\subset Q$, and since $Q$ is not a lozenge, no point fixing $x$ can fix another point in $Q$ by Lemma \ref{lem:two_fix_points}. So case 1 cannot happen. 

\begin{figure}
   \labellist 
  \small\hair 2pt
 \pinlabel $y_0$ at 25 45
 \pinlabel $y_1$ at 53 25
 \pinlabel $x$ at 10 14
  \endlabellist
     \centerline{ \mbox{
\includegraphics[width=5cm]{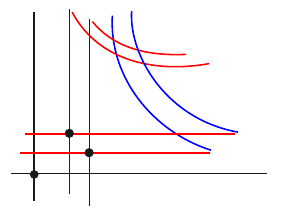} }}
\caption{Configuration of lozenges $L_0, L_1$ in case 1.}
\label{fig:opp_lozenges}
\end{figure}

We now assume that we are in case 2: $L_n$ intersects the ray $r^-$, but not $r^+$. That is, $L_n$ is in the top left quadrant of $y_n$. As above, up to renaming the $y_n$, we may assume that $y_1 = gy_0$ is to the bottom right of $y_0$ and totally linked with it. But this forces $z_1$, which is a pivot or a prong to be inside $L_0$, which is impossible (by the ``non-corner criterion'' \cite[Lemma 2.29]{BFM}, as in the argument of the proof of Lemma \ref{lem:adjacent_corners_discrete}).

Finally case 3 follows as in case 2 by symmetry.\qedhere 
\end{proof}

\subsection{$\mathbb{Z}^2$ invariant Smale chains}
Our motivating example of Smale chains are the (images in the orbit space of) lifts of transverse tori separating basic sets of an Anosov flow on a compact 3-manifold.  Such tori project to chains of lozenges; and coming from tori, these chains are invariant under a  $\bZ^2$ subgroup of $\pi_1(M)$.  Moreover, this $\bZ^2$ subgroup can be taken to be generated by one element fixing every corner, and another element acting freely.

We next show that a similar result holds for arbitrary Smale chains of \emph{cocompact} Anosov-like actions, giving (a more general version of) Theorem \ref{thm:transitive_or_torus}.  

\begin{theorem}[$\bZ^2$-invariant wandering chains] \label{thm_cocompact_Z2}
Suppose $G$ has a cocompact Anosov-like action on $(P,\cF^+,\cF^-)$ and suppose $\cW$ is a Smale chain in $P$.  Then
\begin{enumerate}
\item There exists a lozenge $L$ of $\cW$ contained in a Smale chain $\cW'$ that is invariant under a $\bZ^2$ subgroup, and 
\item For any finite subchain $F$ of $\cW$, there is a (minimal) chain $\cW''$ of wandering lozenges invariant under a $\bZ^2$ subgroup and containing $F$.
\end{enumerate} 
\end{theorem}

This immediately implies Corollary \ref{cor_atoroidal_transitive} stated in the introduction, showing that any pseudo-Anosov flow on an algebraically atoroidal 3-manifold is transitive. 

Note that, in many cases, the chain $\cW''$ given by the second item of the theorem will be a Smale chain, but we do not know if this is true all the time. In the proof we give, the chain $\cW''$ may fail to satisfy condition \ref{item_Smale_share_side_or_prong} of Definition \ref{def_Smale_chain}.

\begin{rem}\label{rem_transverse_tori_and_pivot_only_chains}
Interestingly, in the case of an Anosov flow on a $3$-manifold $M$, Barbot and Fenley \cite{Barbot_MPOT,BarbFen_pA_toroidal} showed that an embedded torus $T$ can be homotoped to be transverse to the flow if and only if there is an associated minimal chain of lozenges $\cC$, invariant under $\pi_1(T)$ such that each corner in $\cC$ is a pivot.\footnote{This particular result is not stated as such in \cite{Barbot_MPOT,BarbFen_pA_toroidal}, but follows from the arguments developed in Section 7 of \cite{Barbot_MPOT}.}
Thus, our theorem gives a sort of converse of that result: For Anosov flows on compact $3$-manifolds, if $\cC$ is a minimal bi-infinite chain of lozenges such that every corner is a pivot, then there exists a lozenge $L$ in $\cC$ contained in a minimal ``pivot-only'' chain of lozenges $\cC'$ (which is not necessarily equal to $\cC$) that is preserved by a $\bZ^2$-subgroup of $\pi_1(M)$ and hence is the projection of a transverse torus. 
\end{rem}

For the proof, we need the following lemma

\begin{lemma} \label{lem:finitely_many_orbits}
If an Anosov-like action is co-compact, and $\cW$ is a Smale chain, then the lozenges of $\cW$ lie in finitely many $G$-orbits.  
\end{lemma}

\begin{proof} 
Let $\cW$ be a Smale chain.  Suppose for contradiction that there are infinitely many orbits of lozenges.  Then we may find a sequence of corners $c_1, c_2, \ldots $ all in distinct orbits.  By cocompactness, for each $i$ there exists $g_i \in G$ such that $g_i(c_i)$ lies in a fixed compact set $K \subset P$.  Thus, there is some accumulation point of the $g_i(c_i)$.  Each $g_i(c_i)$ is a corner of $g_i(\cW)$.   But by Lemma \ref{lem:adjacent_corners_discrete}, the set of corners of Smale chains is closed and discrete.
\end{proof}

Our next lemma shows that Smale chains behave somewhat like train tracks, in the sense that segments can be merged. 

\begin{lemma}\label{lem_merging_Smale_chains}
Let $\cW:=\{L_i, i \in I\}$ and $\cW':=\{L'_i,i\in I'\}$ be two chains of lozenges (finite or infinite, with index sets $I$ and $I'$ some consecutive subsets of integers containing $0$), satisfying conditions \ref{item_Smale_wandering}--\ref{item_Smale_no_triple_shared_corners} given in Definition \ref{def_Smale_chain}.  
Assume that there exists $k$ such that, for $0\leq i\leq k$, $L_i = L'_i$. If $k=0$, suppose moreover that the shared corner of $L_{-1}$ and $L_0$ is the same as the shared corner of $L'_{-1}$ and $L'_0= L_0$.
Then  $\cW'' := \{L_i,i\leq k\}\cup \{L'_i,i\geq k+1\}$
also satisfies conditions \ref{item_Smale_wandering}--\ref{item_Smale_no_triple_shared_corners}. \end{lemma}
\begin{proof}
Obviously, all lozenges of $\cW''$ are wandering, so we only need to verify that conditions \ref{item_Smale_share_side_or_prong} and \ref{item_Smale_no_triple_shared_corners} are satisfied.

We start with condition \ref{item_Smale_share_side_or_prong}, i.e., we need to show that consecutive lozenges in $\cW''$ share a side or a prong corner. By assumption on $\cW$ and $\cW'$, this is true for any pair of consecutive lozenges in $\{L_i,i\leq k\}$ or in $\{L'_i,i\geq k+1\}$. So we only need to check that it also holds for $L_k$ and $L_{k+1}'$. But this directly follows from the fact that $L_k= L'_k$.

Similarly, to prove condition \ref{item_Smale_no_triple_shared_corners}, i.e., that no triplet of lozenges in $\cW''$ share a common corner, it is enough to prove this at the joining of the two chains, that is, for the triplets $L_{k-1}, L_k, L'_{k+1}$ and $L_{k}, L'_{k+1}, L'_{k+2}$.

Since $L_k=L'_k$, the three lozenges $L_{k}, L'_{k+1}, L'_{k+2}$ cannot share a common corner.
For the triplet $L_{k-1}, L_k, L'_{k+1}$, we separate the case $k=0$ from $k>0$: If $k>0$, then $L_{k-1}=L'_{k-1}$ and $L_k=L'_k$, so the three do not share corners because $\cW'$ is a Smale chain. If $k=0$, call $c_i$, resp.~$c'_i$, the shared corner of $L_{i-1}$ and $L_i$, resp.~$L'_{i-1}$ and $L'_i$. The  triplet $L_{-1}, L_0, L'_{1}$ share a common corner if and only the corner shared by $L_{-1}, L_0$, which is $c_0$, is the same as the corner shared by $L_0=L'_0$ and $L'_1$, which is $c_1'$. But by assumption, we have that $c_0=c'_0$ and $c_1=c'_1$, so $c_0\neq c_1'$ and we deduce that $\cW''$ satisfies \ref{item_Smale_no_triple_shared_corners}. This finishes the proof.
\end{proof}

Using this we can now prove Theorem \ref{thm_cocompact_Z2}.  

\begin{proof}[Proof of Theorem \ref{thm_cocompact_Z2}]
Let $\cW$ be a Smale chain.  We begin by proving the first statement.  Enumerate the lozenges in $\cW$ by $\{L_i\}$, in order, and let $c_i^-, c_i^+$ denote the corners of $L_i$, indexed so that $c_i^-$ is the shared corner with $L_{i-1}$.  
By Lemma \ref{lem:finitely_many_orbits}, there are only finitely many orbits of lozenges in $\cW$, and thus we can find some $g \in G$ and some $L_i$ in $\cW$ such that $g L_i = L_{i+k}$ for some $k>0$, {\em and} such that $g(c_i^-) = c_{i+k}^-$. (Since there are only finitely many orbits, some orbit has more than two lozenges in it, and so some element maps a lozenge to another, preserving the signs of corners).  
Without loss of generality, re-index so that $i=0$.  
Thus, by Lemma \ref{lem_merging_Smale_chains}, the chain $L_0, L_{1}, \ldots, L_{k} = g(L_0), g(L_{1}), \ldots, g(L_k)$ satisfies 
conditions \ref{item_Smale_wandering}--\ref{item_Smale_no_triple_shared_corners}.  Now define 
$\cW'$ to be the chain consisting of lozenges $g^n(L_i)$ for $n\in \bZ$ and $n \leq i \leq k$.   Iterated applications of Lemma \ref{lem_merging_Smale_chains} shows that this is a Smale chain, and by construction $g$ preserves $\cW'$ and acts on it freely.  Furthermore, $\cW'$ is contained in the maximal chain containing $\cW$.  Let $h$ be the generator of the stabilizer of the corners of $\cW$ (which is isomorphic to $\bZ$ by Theorem \ref{thm:discrete_stabilizer}).  Then $h$ and $g$ commute, giving the desired $\bZ^2$ subgroup. This proves the first statement of the theorem.

For the second statement, let $F$ be a finite subchain of $\cW$, and let $\cC_w$ denote the maximal chain of wandering lozenges containing $\cW$.  
Any $h \in G$ sending a lozenge of $\cC_w$ to another must globally preserve $\cC_w$.   As in the proof of Lemma \ref{lem_minimal_equivalent_no_3_shared_corners}, if we choose for each lozenge of $\cC_w$, a segment connecting its two corners, the result is a tree (with edges formed by segments and vertices the corners), in which $\cW$ appears as an embedded copy of $\bR$.  Fix an orientation on this tree.  We will find some $h \in G$, preserving orientation on the tree such that $h(F) \cap F = \emptyset$ and $h(F) \cup F$ lie in an oriented, embedded $\bR$.  Having done so, we can define $\cW'$ to be the chain of lozenges forming the {\em axis} of $h$ in the tree, which is a chain of (necessarily wandering) lozenges containing $F$ (and $h(F)$), on which $h$ acts by translations.  As in the proof of the first statement, $h$ together with a generator for the corner-wise stabilizer of $\cC$ generate a $\bZ^2$ subgroup leaving invariant the axis. 

Thus, it remains to find such an $h$.  By Lemma \ref{lem:finitely_many_orbits},
there exists $g_1 \in G$ preserving the orientation and lozenges $L_1, g(L_1) \in \cW$ both on the negative side of $F$ with respect to the chosen orientation of the tree.  Up to replacing $g$ with $g^{-1}$, suppose $L_1$ is closer to $F$ than $g_1(L_1)$.  Similarly, there is $g_2$ preserving orientation of the tree and $L_2$ on the positive side of $F$ with $g_2(L_2)$ further from $F$ than $L_2$.  That is, the axes of $g_1$ and $g_2$ contain lozenges in $\cW$ to the negative (resp.~positive) sides of $F$ in the tree. 
If the axis of $g_1$ or $g_2$ contains $F$, we are already done, otherwise (for sufficiently large $N$) both $g_1^N(F)$ and $g_2^Ng_1^N(F)$ will be disjoint from $F$, and $F \cup g_2^Ng_1^N(F)$ will lie on an embedded, oriented $\bR$, as desired.  
\end{proof}

In the special case where $P$ is the orbit space of a pseudo-Anosov flow on a compact 3-manifold, we can use the following result of 
Barbot and Fenley to show that wandering lozenges lie in weakly embedded Birkhoff surfaces.  
 
\begin{theorem}[See Proposition 6.7 of \cite{BarbFen_pA_toroidal}]
Let $P$ be the orbit space of a pseudo-Anosov flow.  
Suppose $\cC$ is a chain of lozenges in $P$ invariant under a subgroup $G$ of $\pi_1(M)$ isomorphic to $\bZ^2$ or the fundamental group of the Klein bottle.  If $\cC$ is {\em simple} and if any element of $\pi_1(M)$ mapping a lozenge of $\cC$ to itself is in $G$, then $\cC$ is the image of a weakly embedded Birkhoff torus with fundamental group $G$. 
\end{theorem}

``Simple" means that no element of $\pi_1(M)$ maps the corner of a lozenge of $\cC$ into the interior of a lozenge of $\cC$.  This is immediate for chains of wandering lozenges.   
Combined with Theorem \ref{thm_cocompact_Z2}, we have the following immediate corollary.  

\begin{corollary}\label{cor:Z2_Smale_chains_in_pA_flows}
Let $P$ be the orbit space of a pseudo-Anosov flow on a compact 3-manifold. Any $\bZ^2$-invariant Smale chain corresponds to the projection to the orbit space of a weakly embedded Birkhoff torus or Klein bottle. 

Also, if $L$ is a wandering lozenge in $P$, then $L$ lies in the projection of a weakly embedded Birkhoff torus or Klein bottle.
\end{corollary} 
Here, weakly embedded means that the complement of the union of periodic orbits on the Birkhoff surface is an embedded surface, but the union of periodic orbits is in general only immersed.  

\begin{rem}\label{rem:weakly_embedded_to_transverse}
As explained in Remark \ref{rem_transverse_tori_and_pivot_only_chains}, in the case of an Anosov flow, using the work of Barbot and Fenley in \cite{Barbot_MPOT,BarbFen_pA_toroidal}, one can show that, for an orientable manifold, the weakly embedded Birkhoff torus associated to a $\bZ^2$-invariant Smale chain can be isotoped to be made transverse to the flow.  (In the case of pseudo-Anosov flows, such weakly embedded tori can instead be isotoped to a torus that is transverse to the flow except possibly along some singular periodic orbits.)
\end{rem}

\section{Smale bounded actions and density of non-corner leaves} \label{sec:density} 

A {\em non-corner fixed leaf} is a leaf of $\cF^\pm$ that does not contain a corner of a lozenge but is fixed by some nontrivial element of $g$.  Equivalently, it is a leaf of a non-corner fixed point.  
Axiom \ref{Axiom_dense} says that the union of leaves (of either foliation) containing fixed points is dense.  The point of this section is to show density of non-corner fixed leaves under a mild assumption on the actions. This will be proved in two steps, first we will show density on extremal Smale classes, and then use one extra assumption, called Smale-bounded, introduced below, which says that the saturations of extremal Smale classes are dense to deduce the density of non-corner fixed points on the whole plane.

\subsection{Smale-bounded actions}

Recall (see Definition \ref{def_extremal_class}) that a Smale class is called extremal if it is regular and either minimal or maximal for the Smale order.
 \begin{definition}[Smale-bounded actions]\label{def_smale_bounded}
  We say that an Anosov-like action $G$ is \emph{Smale-bounded} if the following two conditions hold:
 \begin{enumerate}[label=(\roman*)]
  \item For any regular Smale class $\Lambda$, there exist extremal Smale classes $\Lambda_m, \Lambda_M$ such that $\Lambda_m \lG \Lambda\lG \Lambda_M$,
  \item For any singular Smale class $\Lambda$, $p\in \Lambda$, and any face $f^+$ of $\cF^+(p)$, resp.~$f^-$ of $\cF^-(p)$, there exists a maximal class $\Lambda_M$, resp.~minimal class $\Lambda_m$, such that $\cF^-(\Lambda_M)\cap f^+ \neq \emptyset$, resp.~$\cF^+(\Lambda_m)\cap f^- \neq \emptyset$.
\end{enumerate}  
 \end{definition}

The point of this definition is just to say that one cannot have an infinite increasing or decreasing sequence of Smale classes, but since the order is not well define on singular Smale class, we have to treat them separately.

We will show that if an action is cocompact, or more generally has only finitely many Smale classes, then it is Smale-bounded.  Before doing this, we give the following equivalent definition, which is the property needed to prove density of non-corner fixed points.

\begin{lemma}\label{lem_charac_Smale_bounded}
An Anosov-like action $G$ is \emph{Smale-bounded} if and only if for any open set $U$, there exists a maximal Smale class $\Lambda_M$ and a minimal Smale class $\Lambda_m$ such that $\cF^-(\Lambda_M)\cap U\neq \emptyset$ and $\cF^+(\Lambda_m)\cap U\neq \emptyset$.
\end{lemma}

\begin{proof}
Suppose that $G$ is Smale bounded and consider $U$ an open set. By Axiom \ref{Axiom_dense}, there exists $x$ fixed by some $g\in G$ with $\cF^+(x)\cap U \neq \emptyset$. 

Up to switching $g$ with $g^{-1}$, we have that $g^n(U)$ will contain a face $f^-$ of $\cF^-(x)$ in its limit as $n\to +\infty$.
By assumption, there exists a minimal Smale class $\Lambda$ such that $\cF^+(\Lambda)\cap f^- \neq\emptyset$. Thus, for some high enough power $n$, $g^n(U)$ will intersects $\cF^+(\Lambda)$ and $G$-invariance of Smale classes yields the result.

The converse direction is obvious: if $x$ is any point in a (regular or singular) Smale class and $f^\pm$ is a face of $\cF^{\pm}(x)$, then we can take a trivially foliated open set $U$ containing a segment of $f^\pm$ in its boundary and the definition of Smale bounded follows immediately.
\end{proof}
 
 Next, we want to show that many natural Anosov-like actions (such as those coming from pseudo-Anosov flows on compact 3-manifolds) are Smale-bounded.
\begin{proposition}\label{prop_finite_or_cocmpact_implies_Smale_bounded}
If an Anosov-like action has only finitely many Smale classes, then it is Smale-bounded.
In particular, if the action of $G$ on $P$ is co-compact, then it is Smale-bounded.
\end{proposition}

\begin{proof} 
Suppose that $G$ has finitely many Smale classes $\Lambda_1, \dots, \Lambda_n$. 
By Lemma \ref{lem_only_prongs}, at least one of these classes is regular.  
By definition, each regular Smale class has minimal Smale class below it and maximal one above it for the Smale order.

Let $U$ be a trivially foliated open set in $P$.  By Axiom \ref{Axiom_dense}, there exists $x$ fixed by some $g\in G$ with $\cF^+(x)\cap U \neq \emptyset$. If $x$ is in a regular Smale class, then a minimal class below the class of $x$ will have a point $p$ with $\cF^+(p)\cap U\neq \emptyset$. 
Similarly, we can find a fixed point $y$ with $\cF^-(y)\cap U \neq \emptyset$ and if $y$ is in a regular class then we have some maximal $\Lambda$ with $\cF^-(\Lambda) \cap U \neq \emptyset$.  Thus, to prove the first statement of the proposition, we need only eliminate the case that the only possibilities for such $x$ (or $y$) are isolated prong singularities.  We treat the case of $x$, the other being symmetric. 

Suppose for a contradiction that every point $x$ in $\Fix_G$ such that $\cF^+(x)\cap U\neq \emptyset$ is an isolated prong singularity.
Let $x_1$ be such a point.  Then all the quadrants of $x_1$ are wandering lozenges (by Proposition \ref{prop:isolated_fixed_points}), so there exists a lozenge $L_1$ with corner $x_1$ such that $L_1 \cap U \neq \emptyset$. Now pick $x_2$ an isolated prong singularity such that $\cF^+(x_2)$ intersects $L_1\cap U$ and consider, similarly, a lozenge $L_2$ that intersects $L_1\cap U$. We can keep iterating this process and get a sequence $x_i$ of isolated prong singularities and lozenges $L_i$ which all intersect. Because there are only finitely many Smale classes, we can extract a subsequence such that all the $x_i$ are in the same $G$-orbit. In particular, (as in the proof of Lemma \ref{lem_only_prongs}) we can find $g\in G$ such that $x_k = gx_i$ for some $i$ and $k$ and such that $L_k = gL_i$ getting a (nonsingular) fixed point inside $L_i$ which is a contradiction.

Finally, we will show that if $G$ acts co-compactly, then the number of Smale classes must be finite:
Let $\Lambda_i$ be a Smale class. If $\Lambda_i$ is regular, then each $x \in \Lambda_i$ has a neighborhood $U_x$ that is product foliated, so any $y \in U_x$ satisfies $y \sim_G x$.    Suppose $G$ acts cocompactly on $P$. Since Smale classes are $G$-invariant, each is represented in a compact fundamental domain.  By the previous remark, there can be no accumulation points of distinct regular Smale classes, and this is also the case by definition for singular Smale classes, so there are only finitely many Smale classes.   
\end{proof} 

We note that there exist example of Anosov-like actions that are not cocompact but {\em are} Smale bounded, see Example \ref{ex:bounded_noncocompact}. 
It is also easy to build examples of actions that are not bounded, even in the flows setting, if one allows for non-compactness, see Example \ref{ex:non-bounded}.

\subsection{Density of non-corner fixed points}

Before we start with the proof, we will prove one more preliminary result that will help us simplify the argument. This is a refinement of the fact that being a corner point is a closed property among points fixed by nontrivial elements in $G$.\footnote{Notice that it may not be a closed property in general, see Proposition 3.5 of \cite{Fenley22}.}

\begin{lemma}\label{lem_corner_closed_property}
Let $x_n$ be a sequence of corners of lozenges $L_n$ that converge to some point $x$ fixed by a nontrivial element $g\in G$. Then $x$ is a corner of a lozenge $L$.
Moreover, up to taking a subsequence, the sides of $L_n$ abutting $x_n$ converge to the sides of $L$ abutting $x$.
\end{lemma}
What will be particularly useful to us is the contrapositive of the above result
\begin{corollary}\label{cor_non_corner_is_open}
Being a non-corner point is an open property among points fixed by nontrivial elements of $G$.
More precisely, assume $x$ is a point fixed by some nontrivial element $g\in G$. Let $Q$ be a quadrant of $x$. If $Q$ does not contain a lozenge with corner $x$, then there exists an open neighborhood $U$ of $x$ such that no lozenge contained in $U$ has a corner in $U\cap Q$.
\end{corollary}

\begin{proof}[Proof of Lemma \ref{lem_corner_closed_property}]
Up to taking a subsequence, we may assume that all $x_n$ are pairwise totally linked, totally linked with $x$, and nonsingular. We may further assume that all $x_n$ are in the same quadrant $Q$ of $x$ and that their leaves converge monotonically towards those of $x$ and hence all the $x_n$ lie in the product foliated rectangle defined by $\cF^{\pm}(x_1)$, $\cF^{\pm}(x)$ and their intersections.  
Calling $y_n$ the other corner of $L_n$, we can also assume that the $y_n$ are all in the same quadrant $Q'$ of $x$. (The quadrant $Q'$ could be the same as $Q$, or one of the two adjacent ones, or possibly, if $x$ is nonsingular, the opposite quadrant).  This implies that, for any $n$ we have $L_n \cap L_1 \neq \emptyset$. 

The possible configurations, depending on whether $y_n$ lies in $Q$, in an adjacent quadrant $Q'$, or in an opposite quadrant $Q'$ are shown in Figure \ref{fig:cases_configurations}. 

\begin{figure}
   \labellist 
  \small\hair 2pt
 \pinlabel $x_1$ at 52 53
  \pinlabel $x$ at 20 26
\pinlabel $y_1$ at 90 85
  \pinlabel $y_n$ at 71 75
   \pinlabel $x_1$ at 217 88
  \pinlabel $x$ at 181 62
\pinlabel $y_1$ at 230 25
  \pinlabel $y_n$ at 260 45
   \pinlabel $x_1$ at 430 88
  \pinlabel $x$ at 396 62
\pinlabel $y_n$ at 355 25
  \pinlabel $y_1$ at 388 45
    \endlabellist
     \centerline{ \mbox{
\includegraphics[width=14cm]{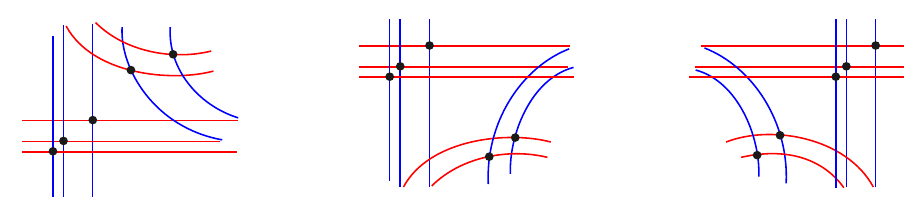} }}
\caption{Possible configurations of sequences of corners}
\label{fig:cases_configurations}
\end{figure}

In each case, we will have (up to switching labels of $\cF^-$ and $\cF^+$ and passing to a further subsequence) one of the following properties: 
\begin{itemize} 
\item $\cF^+(y_n)$ intersects two sides of $L_1$, or
\item $\cF^+(y_n)$ intersects one side of $L_1$ and $\cF^-(x)$
\end{itemize} 
for all $n$.  

Since there are no infinite product regions, $\cF^+(y_n)$ must converge to a single leaf, or to a union of nonseparated leaves.  Since $\cF^+(y_n)$ makes a perfect fit with $\cF^-(x_n)$, we have that one leaf in the limit makes a perfect fit with $\cF^-(x)$.  
Since $x$ was assumed to be fixed by $g$, and one of its leaves makes a perfect fit, Lemma \ref{lem:two_fix_points} implies that $x$ is the corner of a lozenge $L$. By construction, one side of $L$ through $x$ is in the limit of the side of $L_n$, one easily deduce that the other side of $L$ through $x$ is also in the limit of a side of $L_n$.
\end{proof}

We also make the following observation, which generalizes the well-know fact that hyperbolic attractors (resp.~repellers) of an Anosov flow are saturated by unstable (resp.~stable) leaves.
\begin{observation}\label{obs:maximal_are_F+_saturated}
If $\Lambda$ is a maximal Smale class in $\cR_G$, then $\cF^+(\Lambda)\cap \cR_G = \Lambda$. Similarly if  $\Lambda$ is a minimal Smale class in $\cR_G$, then $\cF^-(\Lambda) \cap \cR_G = \Lambda$.
\end{observation}
\begin{rem}
The above statement does not hold for the closure $\bar\Lambda$ in $\Fixbar_G$: It is possible to have an example where a prong $p$ contained in a maximal Smale class $\bar\Lambda$ but such that one ray of $\cF^+(p)$ is not in $\bar\Lambda$. See the example built in Proposition \ref{prop_chain-recurrent_non_transitive}.
\end{rem}

\begin{proof}
The two statements are equivalent up to switching $\cF^+$ and $\cF^-$, so we only prove the first one.

Assume that $\Lambda \subset \cR_G$ is a maximal Smale class.  Let $x\in \Lambda$ and let $z\in \cF^+(x)$. By density of fixed leaves (Axiom \ref{Axiom_dense}), together with the fact that $\mathrm{Sing}$ is discrete by Lemma \ref{lem:adjacent_corners_discrete}, there exists a sequence $y_n \in \cR_G$ such that the intersections $p_n = \cF^+(x) \cap \cF^-(y_n)$ converge to $z$. By maximality of $\Lambda$, we must have $y_n \in \Lambda$. In particular $y_n \sim_G x$ and thus Proposition \ref{prop_basic_product_structure} implies that $z$ is accumulated by points $z_n \in \Lambda$, which are totally linked with $z$.  Thus, by Corollary~\ref{cor:char_smale_class} we have $z\in \Lambda$.
\end{proof}

\begin{proposition}[Density of non-corner fixed points] \label{prop:simple_dense} 
Let $\Lambda$ be an extremal Smale class.  Then the set of non-corner fixed points is dense in $\Lambda$. 
\end{proposition}

\begin{rem}\label{rem_density_Anosov_case}
In the case of an Anosov flow, we can give a quick idea of the argument that we will make below: An extremal Smale class contains a boundary periodic orbit $\alpha$ (as in Corollary \ref{cor_boundary_leaves_and_boundary_points}) such that, for an appropriate choice of orientations, its wandering lozenges will be in the half-space ``above'', and another boundary orbit $\beta$ for which they will be below. By transitivity of the flow inside each Smale class, a dense set of periodic orbits will visit the two lower quadrants of $\alpha$ as well as the upper quadrants of $\beta$. Looking at the configuration in the orbit space, any such orbit cannot be a corner orbit.
\end{rem}

\begin{proof}[Proof of Proposition \ref{prop:simple_dense}]
For concreteness, assume that $\Lambda$ is maximal rather than minimal.  
If $\Lambda$ is the unique Smale class, then, by Theorem \ref{thm:transitive}, $G$ is a transitive Anosov-like action in the sense of \cite{BFM}, so the non-corner fixed points are dense by Lemma 2.30 of \cite{BFM}.

We now assume that there are at least two distinct Smale classes, and we will find a non-corner point of $\Lambda$ fixed by some nontrivial element of $G$.  Since the non-corners form an open subset of $\Fix_G$ (Lemma \ref{lem_corner_closed_property}) and $G$ acts topologically transitively on $\Lambda$ by Theorem \ref{thm:transitive}, this is enough to prove the proposition. 

By Theorem \ref{thm:Smale_chains_separate}, there exists a Smale chain $\cW$ with a corner $c\in \bar\Lambda$. By Observation \ref{obs:maximal_are_F+_saturated}, $\Lambda$ is $\cF^+$ saturated.  

Let $L_1$ be a lozenge of $\cW$ with corner $c$.  Since $\Lambda$ is $\cF^+$ saturated (Observation \ref{obs:maximal_are_F+_saturated}), there exists 
a face of  $\cF^+(c)$ that is accumulated by points of $\Lambda$; thus $c$ has a pair of regular TL subquadrants, that we denote by $R_1$ and $R_2$ which are adjacent (and share a ray of $\cF^-(c)$).  Up to switching $L_1$ to another wandering subquadrant if needed, we assume $L_1$ is adjacent to $R_1$. 
If either $R_1$ or $R_2$ contains non-corner points of $\Fix_G$, these lie in $\Lambda$ and we are done, so we assume such points are all corners.   

For concreteness, we orient $R_1 \cup R_2$ so that the face in $\cF^+(c)$ bounding $R_1 \cup R_2$ is horizontal, and $L_1$ lies in the quadrant above $R_1$.  
We first establish the following claim.  

\begin{claim} \label{claim:lower_quadrant}
There exists a dense set $D_1$ in $R_1 \cap \Fix_G$, such that the intersection of $D_1$ with any open set has nonempty interior, and such that, if $x\in D$ is the corner of a lozenge, then this lozenge is in a lower quadrant of $x$.
\end{claim}
\begin{proof} 
Suppose $x$ is a corner point in $R_1 \cap \Fix_G$. Lemma 2.29 in \cite{BFM}  says the lozenge cannot contain $c$, due to the perfect fit formed by sides of $L_1$.  This eliminates one of the two upper quadrants (the top right).  To finish the proof, we will show that any small neighborhood of $x$ contains an open set of points of $R_1 \cap \Fix_G$ that are either non-corner, or only have lozenges in their lower quadrant.
  To do this, consider some nonsingular corner point $y$ in $R_2$ totally linked with $x$.  By Corollary \ref{cor:TL_product}, for any trivially foliated neighborhood of $x$, there exists $z\in R_1 \cap \Fix_G$ in that neighborhood and $g \in G$ such that $g(z)$ lies arbitrarily close to $y$, and we can take $g$ to preserve the local orientation. By the first part $z$ cannot have a lozenge in its upper left quadrant. 
Now, either $c$ is singular, or it is nonsingular and its remaining quadrant is a wandering lozenge.  In either case, applying Lemma 2.29 of \cite{BFM} to $g(z)$ shows that it cannot be the corner of a lozenge $L$ in the quadrant containing $c$.  Equivalently, $z$ cannot have a lozenge in its other upper quadrant, proving the claim. 
\end{proof}

Let $l^-$ be a leaf of $\cF^-$ passing through $R_1$ and $L_1$. We will first show that there exists a finite line of wandering lozenges $L_1, L_2, \ldots L_k$ (possibly with $k=1$) intersecting $l^-$ such that the TL-subquadrant $Q$ intersecting $l^-$ above $L_k$ (i.e. on the opposite side of $L_k$ as $R_1$) is a {\em regular} TL subquadrant.    The argument for this claim as well as the next step of the proof is illustrated in Figure \ref{fig:R1R2}, which shows the two possible configurations for $c'$ depending on the parity of $k$. 

\begin{figure}
   \labellist 
  \small\hair 2pt
   \pinlabel $R_1$ at 80 30
    \pinlabel $c$ at 108 38
  \pinlabel $R_2$ at 130 30
 \pinlabel $c'$ at 40 125
 \pinlabel $L_1$ at 80 60
  \pinlabel $L_3$ at 80 120
    \pinlabel $L_2$ at 80 90
 \pinlabel $y_j$ at 55 144
 \pinlabel $R_1$ at 250 30
 \pinlabel $R_2$ at 310 30  
 \pinlabel $c$ at 288 38
 \pinlabel $L_1$ at 250 60
  \pinlabel $L_2$ at 250 90
 \pinlabel $L_3$ at 250 120
  \pinlabel $L_4$ at 250 144
 \pinlabel $c'$ at 290 155
 \pinlabel $y_j$ at 270 170
 \endlabellist
     \centerline{ \mbox{
\includegraphics[width=10cm]{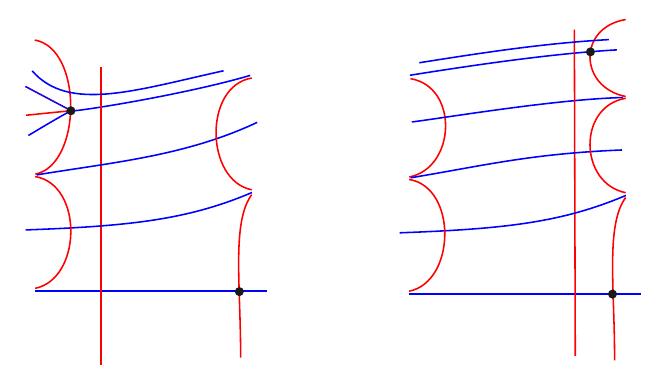} }}
\caption{Finite lines of wandering lozenges, and points $y_j$ that will be totally linked with any point sufficiently close to $c$ in $R_1$.}
\label{fig:R1R2}
\end{figure}
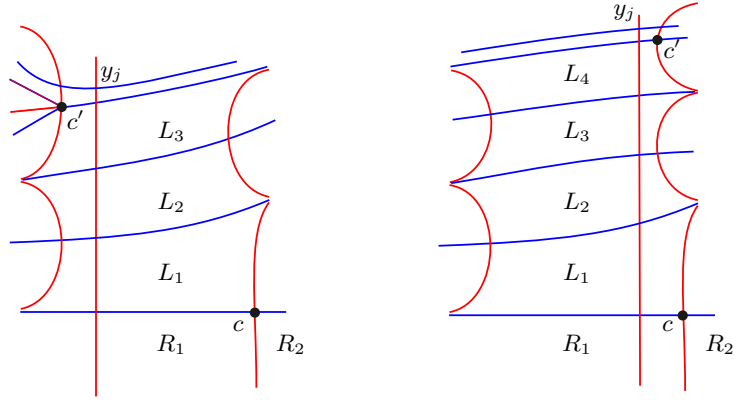

To see this, consider a maximal line $\cL$ of lozenges intersecting $l^-$.  If $\cL$ is an infinite line, then it is bi-infinite and forms a scalloped region. Thus, by Lemma \ref{lem_infinite_line_lozenges}, there exists $h \in G$ translating lozenges along the line, and so either $h(R_1)$ or $h^{-1}(R_1)$ lies above $L_1$ in the line and is regular.  We can then take $Q$ to be the first regular lozenge in $\cL$ above $L_1$.  

If instead $\cL$ is finite, consisting of $L_1, L_2, \ldots L_k$, (and possibly some lozenges below $L_1$), consider the TL subquadrant $Q$ above $L_k$.  By Corollary \ref{cor:wandering_then_lozenge}, if $Q$ were wandering, then it would be a lozenge, contradicting maximality of $\cL$.  This proves the claim.  

Let $c'$ be the corner of $Q$ shared with $L_k$.  We next argue that $c' \in \bar{\Lambda}$.  
To see this, consider a sequence of regular points $x_1, x_2, \ldots$ in $R_1$ approaching $c$, and a sequence $y_1, y_2, \ldots$ in $Q$ approaching $c'$.  Then for all $i, j$ we have that $\cF^-(x_i)$ and  $\cF^-(y_j)$ intersect each lozenge of the line, and so for sufficiently large $i, j$ we will have that $x_i$ and $y_j$ are totally linked, and thus in the same Smale class. 

Apply the argument of Claim \ref{claim:lower_quadrant} above with $c'$, $Q$ and $L_k$ playing the roles of $c$, $R_1$ and $L_1$ (and thus with the roles of ``upper" and ``lower" reversed). 
We conclude that there is a dense set $D_2$ in $Q \cap \Fix_G$, with nonempty interior in any open set, so that no point of $Q$ can be the corner of a lozenge that lies in one of its {\em lower} quadrants (since $Q$ lies above rather than below $L_k$).  

Finally, with this set-up, consider some non trivial element $f \in G$ with a fixed point in $R_1$ and preserving each ray through its fixed point.  Let $z$ be an interior point of $D_2 \subset Q$, which  is in $ \Fix_G \subset \Lambda$.  Then for some $n$ sufficiently large, we have $f^n(\cF^+(z)) \cap R_1 \neq \emptyset$.   Since $\Lambda$ is $\cF^+$ saturated, this means there is a point $z'$ of $\Fix_G \cap \Lambda \cap R_1$ with $f^{-n}(z') \in Q$, and thus by density of $D_1$, we can find such a point $z'$ that lies in $D_1$.  Since $f$ preserves local orientations, such a point $z'$ cannot have a lozenge in {\em any} of its quadrants, and thus is non-corner, which is what we needed to show. 
\end{proof}

\begin{corollary}\label{cor_simple_leaves_are_dense}
If $G$ is a Smale-bounded Anosov-like action, then the set of non-corner fixed leaves in $\cF^+$ and those in $\cF^-$ are dense.

\end{corollary} 

\begin{proof}
By Lemma \ref{lem_charac_Smale_bounded}, the $\cF^+$-saturation of the union of minimal Smale classes is dense in $P$, and so is the $\cF^-$-saturation of the union of maximal Smale classes. The result follows then trivially from Proposition \ref{prop:simple_dense}.
\end{proof}

 \section{Action at infinity determines the plane} \label{sec:boundary}
 
 As mentioned in the introduction, a plane $P$ with two transverse, possibly singular, foliations $\cF^{\pm}$ admits a natural compactification by a circle at infinity, denoted $\Pbound$, such that any foliation-preserving action of a group on the plane by homeomorphisms extends to an action on the circle.  
 
 A consequence of the main theorem of \cite{BFM} is that the {\em transitive} pseudo-Anosov flows on a compact 3-manifold $M$ are completely classified up to orbit equivalence by the induced actions of $\pi_1(M)$ on the circles at infinity of their orbit spaces (\cite[Theorem 1.5]{BFM}).   In this section we discuss the boundary action in the general (possibly nontransitive) case for Anosov-like actions and give a simpler, unified proof that the action at infinity determines the flow in all settings.  
 
 We will not recall all the details of the construction and properties of $\Pbound$ here, but refer the reader to Section 3 of \cite{BFM} and references therein. 
 In particular, we will use the following result.  Although proved for transitive Anosov-like actions in \cite{BFM}, the proof goes through with only the axioms we assume here.  
 
  \begin{proposition}[See \cite{BFM} Proposition 3.6] \label{prop:boundary_action_general}
Let $(P,\cF^+, \cF^-)$ be a nontrivial bifoliated plane with Anosov like action of $G$, and $g \neq 1 \in G$. 
\begin{enumerate} 
\item If $g$ fixes a non-corner point $x$, then the only points of $\Pbound$ fixed by $g$ are endpoints of $\cF^\pm(x)$. 

\item If $g$ fixes all corners in a maximal chain of lozenges $\cC$, then the set of fixed points of $g$ on $\Pbound$ is the closure of the set of endpoints of the sides of lozenges in $\cC$.  

\item If $g$ acts freely on $P$, then it either has at most two fixed points on $\Pbound$, or has exactly four fixed points and preserves a scalloped region.  
\end{enumerate} 
\end{proposition} 

As a consequence of this proposition, any element which fixes exactly four points on $\Pbound$ either fixes a unique nonsingular point in $P$, or fixes the four corners of a scalloped region.  These can be distinguished dynamically, as follows.  
\begin{proposition}\label{prop:four_fixed_points} 
Suppose $g \in G$ fixes exactly four points of $\Pbound$.  Either these points are alternating attractor/repellers and $g$ fixes a unique point of $P$, or there is exactly one attractor and one repeller, and $g$ acts freely on $P$ preserving a scalloped region. 
\end{proposition} 

The argument for the above proposition appears within the proof of Theorem 6.1 of \cite{BFM}, we outline it here for completeness and refer the reader to \cite{BFM} for details.  

\begin{proof} 
By Proposition \ref{prop:boundary_action_general}, $g$ can fix at most one interior point.  If it fixes one, then we are in the first setting.  
If $g$ instead acts freely, then Proposition \ref{prop:boundary_action_general} says that $g$ preserves a scalloped region $U$.  The fixed points in this setting correspond to the four unique accumulation points on $\Pbound$ of the sets of endpoints of boundary leaves of $U$.  Since $g$ acts freely, it must translate lozenges along each of the lines of lozenges that make up $U$.  This means that the two opposite families of boundary leaves are each translated, resulting in one attracting accumulation point, one (opposite) repelling accumulation point, and two accumulation points which are neither attractor nor repeller. 
\end{proof}

 Our next result is a simple consequence of the density of non-corner fixed leaves, and extends the result of \cite{Bonatti_boundary} that the induced action of $\pi_1(M)$ on the boundary of the orbit space of an Anosov flow is always minimal; showing this applies generally to Smale-bounded actions (and hence to every pseudo-Anosov flow).   
 
\begin{theorem}\label{thm_minimal_action_circle}
Suppose that $G$ is a Smale-bounded Anosov like action on $(P,\cF^{\pm})$. Then the action of $G$ on $\Pbound$ is minimal.
\end{theorem}
\begin{rem}
The same result with the same proof holds if we only assume that every Smale class admits an upper bound, or that every Smale class admits a lower bound, since we only need density of simple leaves for one of the foliations. The result however should presumably hold without any restrictions, but we do not attempt to prove that here. 
\end{rem}
 
\begin{proof}
By contradiction, assume that the orbit under $G$ of $\xi$ is not dense in $\Pbound$. Pick $I$ an open interval in the complement of $ \overline{G\xi}$. By density of non-corner fixed leaves $\cF^+$-leaves (Corollary \ref{cor_simple_leaves_are_dense}), we can find two distinct non corner points $x, y$ fixed by respectively $g,h$ and such that $\cF^+(x)$ and $\cF^+(y)$ have an ideal point in $I$. By definition of $I$, its two endpoints must be fixed by both $g$ and $h$. But, since $x,y$ are nonsingular, non-corners, the only points on $\Pbound$ fixed by $g$ and $h$ are the four endpoints of $\cF^{\pm}(x)$ and $\cF^{\pm}(y)$ respectively (Proposition \ref{prop:boundary_action_general}), and these endpoints must be distinct (otherwise, one of $\cF^{\pm}(x)$ would have to make a perfect fit with another leaf, contradicting the fact that $x$ is non-corner).
\end{proof}

 Another consequence of the density of simple leaves, will be that Anosov-like actions which induces conjugated actions on the circle at infinity are actually conjugated. In order to get this result, we need the following theorem obtained in \cite{prelaminations}.
   
 \begin{theorem}[Corollary E of \cite{prelaminations}]\label{thm_prelamination_to_foliation}
 Let $(P_i,\cF_i^{+},\cF_i^{-})$, $i=1,2$, be two topological planes equipped with a pair of transverse (possibly singular with simple prong singularities\footnote{We say that a singular foliation has {\em simple prong singularities} if all singularities are prong and any leaf contain at most one singularity. It is an immediate consequence of Axiom \ref{Axiom_prongs_are_fixed} that if a bifoliated plane admits an Anosov-like action, then the foliations have simple prongs singularities.}) foliations. For $i=1,2$, let $\ell_i^{\pm}$ be subsets of $\cF_i^{\pm}$ such that $\ell_i^+$, resp.~$\ell_i^-$, is dense in $P_i$. Call $\partial\ell_{i}^{\pm}$ the subsets of $\Pbound_i\times \Pbound_i$ consisting of the pairs of ideal points of leaves of $\ell_i^{\pm}$.
 
 Suppose that there exists a homeomorphism $h\colon \Pbound_1 \to \Pbound_2$ sending elements of $\partial\ell_{1}^{+}\cup \partial\ell_{1}^{-}$ to $\partial\ell_{2}^{+}\cup \partial\ell_{2}^{-}$. Then there exists a unique bifoliation-preserving homeomorphism $H\colon (P_1,\cF^{\pm}_1) \to (P_1,\cF^{\pm}_1)$ extending $h$ such that $H(\cF^+_1) = \cF^{\pm}_2$ and $H(\cF^-_1) = \cF^{\mp}_2$.
 \end{theorem}

 Given that theorem, it is now easy to prove Theorem \ref{thm:boundary_determines_action} from the introduction
  which we restate here for convenience.  

\boundarydeterminesaction*

\begin{rem}
Note that in the special case of transitive Anosov-like actions, the proof we obtain here is much easier than the one given in \cite{BFM}, since we do not use the main result of \cite{BFM}. Moreover, density of non-corner fixed leaves is much simpler to prove for transitive actions (compare \cite[Lemma 2.30]{BFM} with Proposition \ref{prop:simple_dense} here).
\end{rem}

\begin{proof}
Let $\rho_i$ be the two Smale-bounded Anosov-like actions on bifoliated planes on $(P_i, \cF_i^{\pm})$, and $h\colon \partial P_1 \to \partial P_2$ a homeomorphism conjugating the induced boundary actions.

Let $\ell_i^{\pm}$ be the subsets of $\cF_i^{\pm}$ made up of the non-corner fixed leaves. These leaves are dense in $P_i$ by Corollary \ref{cor_simple_leaves_are_dense}. Our goal is to show that the map $h$ sends $\partial\ell_{1}^{+}\cap \partial\ell_{1}^{-}$ to $\partial\ell_{2}^{+}\cap \partial\ell_{2}^{-}$.  Then we can use Theorem \ref{thm_prelamination_to_foliation} to obtain, up to switching $\cF_2^+$ and $\cF^-_2$, a bifoliation preserving homeomorphism $H$ between $(P_1, \cF^+_1,\cF^-_1)$ and $(P_2, \cF^+_2,\cF^-_2)$. The uniqueness of $H$, given by Theorem \ref{thm_prelamination_to_foliation}, implies that it conjugates the actions $\rho_i$ on $P_i$. 

Consider $g\in G$ such that $\rho_1(g)$ fixes a non-corner point $x\in P_1$ as well as all the rays of $\cF_1^{\pm}(x)$. Then, by Proposition \ref{prop:boundary_action_general}, $\rho_1(g)$ fixes exactly four points in $\Pbound_1$ (the endpoints of $\cF_1^{\pm}(x)$), two being attractors and two being repellers. Up to changing $g$ to $g^{-1}$, we can assume that $\partial\cF_1^{+}(x)\in \Pbound_1$ are the two attractors for $\rho_1(g)$. 

Since $h$ conjugates the actions on the boundary, we obtain that $\rho_2(g)$ fixes the four points $h(\partial\cF_1^{\pm}(x))$, which are alternating attractors/repellers and those are the only four fixed points of $\rho_2(g)$ on $\Pbound_2$. Proposition \ref{prop:boundary_action_general} now implies that either $\rho_2(g)$ fixes a unique non-corner point in $x$, or $\rho_2(g)$ acts freely on $P_2$ and fixes the four corners of a scalloped region. However, the latter case is excluded by Proposition \ref{prop:four_fixed_points}. 

Thus $h(\partial\cF_1^{\pm}(x))$ are the endpoints of the leaves of the unique non-corner fixed point fixed by $\rho_2(g)$.
Therefore, we deduce that $h(\partial\ell_1^{\pm}) \subset  \partial\ell_2^{\pm}$, and using $h^{-1}$ instead of $h$ in the argument above gives us $h(\partial\ell_1^{\pm}) =  \partial\ell_2^{\pm}$. This ends the proof.
\end{proof}

\section{Examples}\label{sec:examples}

In this section, we gather some examples of unusual behavior that can occur for nontransitive pseudo-Anosov flows and Anosov-like actions.  
\subsection{Smale chains, Smale-boundedness and lack of cocompactness}

Our first example shows that Smale-boundedness is strictly more general than having finitely many Smale classes.
\begin{lemma}
There exist Anosov-like actions that are Smale-bounded with infinitely many distinct Smale classes.
\end{lemma}

The proof of the above result is given in the following example.
\begin{example} \label{ex:bounded_noncocompact}
Let $\phi$ be a contact Anosov flow on a closed manifold $M$ (i.e., $\phi$ preserves a contact $1$-form.\footnote{The contactness assumption is not essential for this example, any skewed $\bR$-covered Anosov flow would do.}). We may furthermore take such a manifold $M$ that is not a rational homology sphere -- many examples exist; in fact one could start with any hyperbolic 3-manifold $M$ carrying a contact Anosov flow and use Agol's virtual fibering to pass to a fibered cover which will not be a rational homology sphere.  

Then, consider its lift $\hat\phi$ to the maximal Abelian cover $A$ of $M$. Since $\phi$ is contact, it is homologically full, and hence the lift $\hat\phi$ is also transitive on $A$, see \cite[Theorem 2.5]{GRH20}. Now pick infinitely many periodic orbits $\alpha_n$ of $\hat \phi$ in $A$, chosen so that they have pairwise distinct tubular neighborhoods of uniform size (i.e., we choose $\alpha_n$ so that they do not accumulate, for instance, one can choose them to be the distinct lifts of a single periodic orbit in $M$). Then, one can realize an attractive DA construction on each $\alpha_n$ (see, e.g., \cite{BBY} for the DA construction), and remove tubular neighborhoods that are transverse to the flow to obtain a flow $\hat\phi$ on a non-compact manifold with boundary $\hat A$. The flow $\hat\phi$ has one hyperbolic repeller, and $\hat A$ has infinitely many boundary tori, all of which are exiting for the flow. 
Then one can glue a hyperbolic attractor to each boundary tori of $\hat A$, as in the Franks-Williams example (see \cite{FW}, or \cite{BBY}). This gives us an Anosov flow on a non-compact manifold with infinitely many attractors and one repeller. One can then lift the flow to the universal cover, define the orbit space as usual, and obtain a Smale-bounded Anosov-like action on a bifoliated plane that has infinitely many distinct Smale classes.
\end{example}

\begin{lemma}
There exists Anosov-like actions that are not Smale bounded.
\end{lemma}
Such examples can be obtained by taking infinite cyclic covers of compact Anosov flows.
For instance:
\begin{example} \label{ex:non-bounded}
Consider the piece of the Bonatti-Langevin example before gluing (see \cite{BL}). It is a flow with one hyperbolic periodic orbit and two boundary tori, one entering one exiting.
Instead of gluing the entering to exiting tori as in the Bonatti-Langevin example, one can instead take a $\bZ$ worth $(N_i, \phi_i)$, $i\in \bZ$, of this model piece and glue the exiting torus of $(N_i, \phi_i)$ to the entering one of $(N_{i+1}, \phi_{i+1})$. (The gluing needs to be chosen so that the closed leaf of the stable lamination on the entering torus of $(N_{i+1}, \phi_{i+1})$ is glued to something not freely homotopic to the closed leaf of the unstable lamination on the exiting torus of $(N_{i}, \phi_{i})$.)

We then get an Anosov flow on a non-compact manifold (where the metric we chose for the Anosov property is the one coming from the metric on the model piece) with infinitely many isolated basic sets.

It is easy to verify that the associated action on the orbit space is Anosov-like (for Axiom \ref{Axiom_dense}, notice that the stable leaf of the periodic orbit in the piece $N_0$ get more and more dense in the piece $N_i$ as $|i|\to \infty$).
\end{example}

Note that one can easily build examples with infinitely many non-isolated basic sets by taking infinitely many copies of a hyperbolic plug (in the sense of \cite{BBY}) with one entering and one exiting torus, and gluing them in a row as in the example above.

\subsection{Toolkit: building pseudo-Anosov flows from fatgraphs}
For the next class of examples, we will use a technique developed in \cite{BarbFen_pA_toroidal} to produce examples of pseudo-Anosov flows with  \emph{periodic Seifert pieces} by gluing building blocks together using the combinatorial data of a fatgraph.  See also \cite{BF_totally_per,BFeM}.  

\begin{definition}[\cite{BarbFen_pA_toroidal}]
An \emph{admissible fatgraph} $(X,\Sigma)$ is a graph $X$ embedded in a surface $\Sigma$ that deformation retracts to $X$ that satisfies to the following properties:
\begin{enumerate} 
\item the valence of every vertex is even, and 
\item the set of boundary components of $\Sigma$ can be partitioned in two subsets (the outgoing and the incoming) so that for every edge $e$ of $X$, the two sides of $e$ lie in different subset of this partition.  
\item each loop in $X$ corresponding to a boundary component contains an even number of edges.
\end{enumerate}
\end{definition}

In \cite{BarbFen_pA_toroidal}, Barbot and Fenley describe a method to build a flow $\phi$ on a manifold $M$ which is a circle bundle over $\Sigma$ based on the data contained in an admissible fatgraph $(X,\Sigma)$. In this construction, each vertex corresponds to a hyperbolic orbit (possibly a $p$-prong, if the valence of the vertex in $X$ is strictly greater than $4$), and each edge corresponds to (a neighborhood of) an annulus bounded by the periodic orbits corresponding to the edges and transverse to the flow in its interior (from the incoming side to the outgoing side of the edge).

The dynamics of the flow on a Seifert manifold $M$ obtained from a fatgraph $X$ is particularly simple: The maximal invariant set in $M$ corresponds exactly to the periodic orbits associated to the vertices, other orbits are either in the stable or unstable leaves of the periodic orbits, or they enter through an incoming boundary, cross the Birkhoff annulus corresponding to an edge and exits through the corresponding outgoing boundary.
In particular, we point out that the stable, resp.~unstable, foliations of each periodic orbits traces a lamination with only closed leaves of the incoming, resp.~outgoing, boundary surfaces. These are called the \emph{stable and unstable boundary laminations}.

For our purpose of building examples, all we need is the following result which summarize the way one can use such constructions to build pseudo-Anosov flows on closed manifolds.
\begin{proposition}[Barbot--Fenley \cite{BarbFen_pA_toroidal}, see also \cite{BFeM}]\label{prop_gluing_fatgraphs}
Let $X_1, \dots, X_n$ be admissible fatgraphs. Let $(M_i,\phi_i)$, $i=1, \dots, n$, be the associated flows and call $S_{in}$ and $S_{out}$ the union of the boundary surfaces that are respectively entering and outgoing. Suppose that $h$ is a diffeomorphism from a subset $\hat S_{in}$ of $S_{in}$ to a subset $\hat S_{out}$ of $S_{out}$, such that the image by $h$ of the stable boundary lamination is not isotopic to the unstable boundary lamination. 

Then, we can find a (possibly empty) union $(A, \psi_A)$, $(E,\psi_E)$ of hyperbolic flows such that $\partial A$ (resp.~$\partial E$) is a union of incoming (resp.~outgoing) tori or Klein bottles and diffeomorphisms $f_A\colon \partial A \to S_{out}\smallsetminus\hat S_{out}$ and $f_E\colon \partial E \to S_{in}\smallsetminus\hat S_{in}$ such that the flow obtained by gluing $(M_i,\phi_i)$, $(A,\psi_A)$ and $(E,\psi_E)$ using the diffeomorphisms $f_A$, $f_E$, and (one isotopic to) $h$ is a pseudo-Anosov flow on a closed manifold.
\end{proposition}

\begin{proof}
Lemma 5.2 of \cite{BFeM} shows that (up to an isotopy of $h$) the flow $(M,\phi)$ obtained from $(M_i,\phi_i)$ by gluing the boundaries according to $h$ is hyperbolic on its invariant set aside from the finitely many prong orbits corresponding to the prong orbits of the $\phi_i$.

If $M$ is closed then $\phi$ is therefore a pseudo-Anosov flow. If $M$ is not closed, we can glue attractors $(A,\psi_A)$ to each outgoing boundaries of $M$ and repellers $(E,\psi_E)$ to each incoming boundaries, as in Proposition 1.1 of \cite{BBY}. (The fact that $M,\phi$ has possibly some singular orbits does not change the arguments in the proof of \cite[Proposition 1.1]{BBY}, which only rely on the usual cone criterion, which is still valid outside the singular orbits). The fact that we can find an attractor $(A,\psi_A)$ (resp.~a repeller) with the appropriate boundary lamination to apply \cite[Proposition 1.1]{BBY} follows from instance from \cite[Theorem 1.10]{BBY}.
\end{proof}

\subsection{Pseudo-Anosov flows, Smale classes and Smale order}
Using related constructions, we show that the presence of prong singularities can lead to points in the closure of distinct Smale classes, that there is no extension of the Smale order to a well defined order on the closure of Smale classes, and that one can have chain-recurrent pseudo-Anosov flows that are not transitive.  

\begin{proposition}\label{ex_prongs_in_distinct_classes}
There exists examples of pseudo-Anosov flows with distinct Smale classes $\Lambda,\Lambda'$ such that their closures intersect, and $\Lambda$, $\Lambda'$ are not comparable with respect to the Smale order. 
\end{proposition}
The proof is contained in the construction of the following example. Note that in this example, the closure of $\Lambda$ and $\Lambda'$ are in the same chain-recurrent component, but the chain-recurrent set is still equal to the non-wandering set. The next propositions will provide examples where the chain-recurrent set is strictly larger than the non-wandering set.

\begin{proof}
Consider the admissible fat graph $(X,\Sigma)$ where $\Sigma$ is the $2$-sphere with 8 disks removed, and $X$ is a graph with two vertices of valence 8 each and no loop-edges.  This is illustrated in Figure \ref{fig:prong_fatgraph}, after identifying the
outer circle of the figure to a single point. 

Call $T_1,\dots, T_8$ the associated transverse tori, numbered in circular order and such that $T_1$ is an outgoing torus for the associated flow.

Build a pseudo-Anosov flow $\phi$ by gluing $T_1$ to $T_2$, $T_5$ to $T_6$, and gluing hyperbolic attractors to $T_3$ and $T_7$, and hyperbolic repellers to $T_4$ and $T_8$.  This is possible by Proposition \ref{prop_gluing_fatgraphs}.  

Let $p_N, p_S$ denote the two prongs.  The edge of the fatgraph between $p_S$ and $p_N$ that lies between $T_1$ and $T_2$ corresponds to a Birkhoff annulus $Q_{12}$ for $\phi$ (i.e., an annulus transverse to the flow and bounded by the two periodic orbits $p_S$ and $p_N$).  Since $T_{1}$ is glued to $T_{2}$ there, the interior of $Q_{12}$ intersects recurrent orbits. In particular, the interior of $Q_{12}$ nontrivially intersects a Smale class $\Lambda_{12}$, and $p_S, p_N\in \overline{\Lambda_{12}}$. Similarly, the Birkhoff annulus $Q_{56}$ between $p_S$ and $p_N$ that lies between $T_5$ and $T_6$ intersects a Smale class $\Lambda_{56}$ which contains $p_S$ and $p_N$ in its closure.

By construction, no orbit passing through $Q_{12}$ can intersect $Q_{56}$.  Thus, the two Smale classes $\Lambda_{12}$ and $\Lambda_{56}$ must be distinct. One also easily sees that these Smale classes are not related by the Smale order
\end{proof}

\begin{figure}[h!]
\labellist 
 \small\hair 2pt
 \pinlabel $T_1$ at 100 80
 \pinlabel $T_2$ at 81 100
 \pinlabel $T_5$ at 32 51
 \pinlabel $T_6$ at 54 33
 \endlabellist
     \centerline{ \mbox{
\includegraphics[width=6cm]{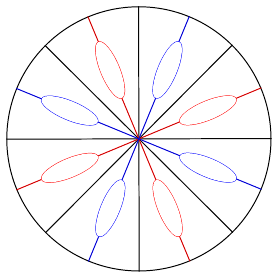} }} %
\caption{Identifying the outer (black) boundary to a single point gives a fatgraph (in black) on a 8-holed sphere.  The local stable and unstable foliations of the periodic orbits corresponding to each vertex in blue and red respectively. The blue circles corresponds to entering tori and those in red corresponds to exiting tori.}
\label{fig:prong_fatgraph}
\end{figure}
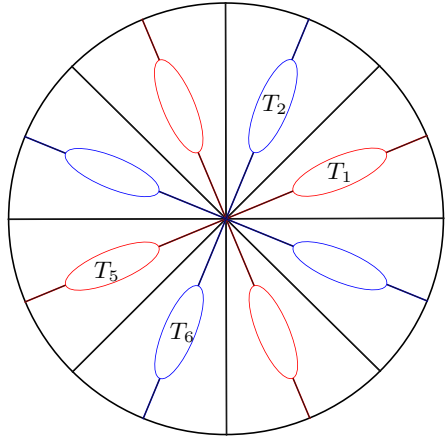

Building on this, we give an example showing that the Smale order \emph{cannot} be extended to the closure of Smale classes. In particular, 
\begin{proposition}\label{prop:weird_loops}
There exists examples of pseudo-Anosov flows with four distinct Smale classes $\Lambda_a, \Lambda_b, \Lambda_c, \Lambda_d$ such that $\Lambda_a <_G \Lambda_b$, $\Lambda_c <_G \Lambda_d$ and two prong singularities $p, q$ such that $p\in \bar\Lambda_a\cap \bar\Lambda_d$ and $b\in  \bar\Lambda_b\cap \bar\Lambda_c$.

Such pseudo-Anosov flows have chain-recurrent sets that are strictly larger than their non-wandering sets. Specifically, every point in $\cF^+(\Lambda_a)\cap\cF^-(\Lambda_b)$ is in the wandering set, but contains a segment of orbit that is part of a chain-recurrent trajectory, and $\Lambda_a, \Lambda_b,\Lambda_c,\Lambda_d$ are all in the same chain-recurrence class.
\end{proposition}

The proof consists of the following construction.  
\begin{proof}
Consider two copies $(X,\Sigma)$, $(X',\Sigma')$ of the same fat graph used in the example of Proposition \ref{ex_prongs_in_distinct_classes}. We use the same notations as in that example, and denote with a $'$ the corresponding objects in $(X',\Sigma')$.

As in Example \ref{ex_prongs_in_distinct_classes}, we build a flow by first gluing $T_1$ to $T_2$, $T_5$ to $T_6$ in the first piece, 
 and respectively $T_1'$ to $T_2'$, $T_5'$ to $T'_6$ in the second piece.  
 We glue hyperbolic attractors to $T_3'$ and to $T_7$, and repellers to $T_8$ and $T_4'$.  
Finally, attach the two copies together by gluing $T_3$ to $T'_8$ and $T'_7$ to $T_4$.

We now obtain a pseudo-Anosov flow with four singular orbits, corresponding to $p_S, p_N, p'_S$ and $p_N'$, We have four Smale classes $\Lambda_{12}$, $\Lambda_{56}$ (which both contain $p_S, p_N$ in their closure), and $\Lambda_{12}'$ and $\Lambda_{56}'$ (which both contain $p_S'$ and $p_N'$ in their closure). 
Since $T_3$ is glued to $T'_8$, there exists orbits that start on $Q_{12}$ and hit $Q'_{12}$: Indeed, there is an open set of orbits that start on $Q_{12}$, go out through $T_1$, so that they come in through $T_2$, and (thanks to the choice of gluing) cross $T_3$. Since $T_3$ is glued to $T_8'$, some of these orbits will later intersect $T'_1$, and thus $Q'_{12}$.
Hence, we deduce that $\Lambda_{12}$ is, say, less then $\Lambda_{12}'$ for the Smale order.

Similarly orbits pass from $Q'_{56}$ to $Q_{56}$, implying that $\Lambda'_{56}$ is less then $\Lambda_{56}$. Moreover, as in the previous example, we have $p_N,p_S \in \bar\Lambda_{12}\cap \bar\Lambda_{56}$ and $p_N',p'_S \in \bar\Lambda'_{12}\cap \bar\Lambda'_{56}$, proving the proposition.
\end{proof}

Finally, we also give an example which is chain-recurrent but nontransitive.
\begin{proposition}\label{prop_chain-recurrent_non_transitive}
There exists pseudo-Anosov flows with exactly two distinct Smale classes $\Lambda_1$, $\Lambda_2$, with $\bar\Lambda_1 \cap \bar \Lambda_2 \neq\emptyset$ and every orbit is chain-recurrent.
\end{proposition}

The proof consists of the following construction.  
\begin{proof}
Consider the fat graph on the $8$-punctured torus drawn in Figure \ref{fig:CBexample_fatgraph}, and build a pseudo-Anosov flow by gluing $T_1$ to $T_2$, $T_3$ to $T_4$, $T'_1$ to $T'_2$, and $T'_3$ to $T'_4$.
Then, one easily verifies that the band above the line containing $x$ and $y$ has recurrent orbits and contains a unique Smale class $\Lambda_{\mathrm{top}}$, and similarly the band below $x$ and $y$ contains a Smale class $\Lambda_{\mathrm{bot}}$. Orbits of the flow can go from the top band to the bottom one, but never leave the bottom one.
Thus, as announced, the flow obtained has exactly two Smale classes, $x,y,z,w$ are all in the closure of both $\Lambda_{\mathrm{top}}$ and $\Lambda_{\mathrm{bot}}$, and the flow is chain recurrent.
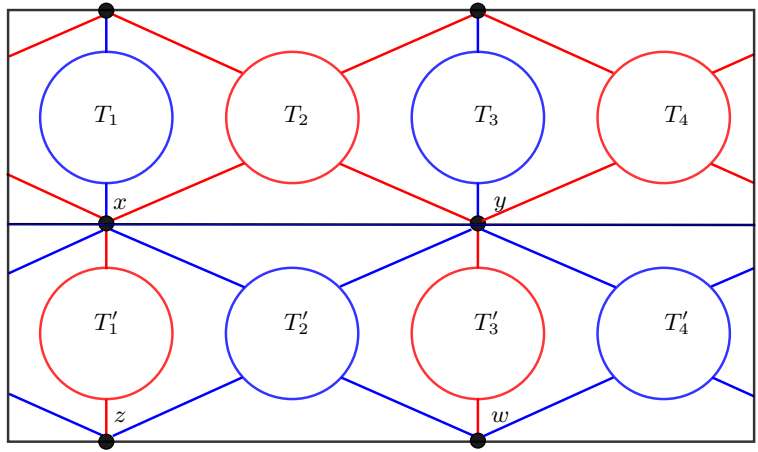
\begin{figure}[h!]
\labellist 
 \small\hair 2pt
 \pinlabel $T_1$ at 56 180
 \pinlabel $T_2$ at 156 180
  \pinlabel $T_3$ at 256 180
   \pinlabel $T_4$ at 356 180
    \pinlabel $T_1'$ at 56 70
 \pinlabel $T_2'$ at 156 70
 \pinlabel $T_3'$ at 256 70
  \pinlabel $T_4'$ at 356 70
 \pinlabel $x$ at 63 133
  \pinlabel $y$ at 263 133
   \pinlabel $z$ at 63 20
    \pinlabel $w$ at 263 20
 \endlabellist
     \centerline{ \mbox{
\includegraphics[width=10cm]{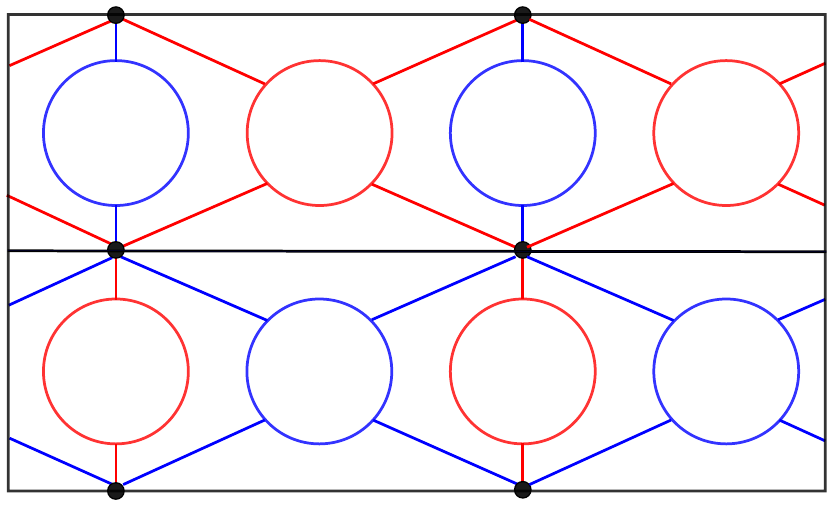} }}
\caption{Fatgraph with $4$ vertices on a $8$-punctured torus (not all the edges are drawn)}
\label{fig:CBexample_fatgraph}
\end{figure}
\end{proof}

Our last example involving Smale classes is the following, illustrating the importance of considering singular Smale classes.  
\begin{proposition}\label{ex_only_singular_Smale}
There exists pseudo-Anosov flows on non-compact $3$-manifolds whose non-wandering set consists entirely of isolated prong singularities.
\end{proposition}

\begin{proof}
The construction is similar to that of Example \ref{ex:non-bounded}.
Consider an admissible fat graph on a surface $S$ such that all vertices have valence $>4$. The associated flow hence has only prong singularities, and we build a pseudo-Anosov flow $\phi$ on a compact manifold $M$ by taking two copies of that associated flow and gluing them along their incoming/outgoing boundaries. 

Call $k$ the number of boundary components of $S$, so that $M$ has $k$ torus in its JSJ decomposition, that we denote by $T_1, \dots, T_k$. Let $G<\pi_1(M)$ be the kernel of the morphism $\pi_1(M)\to \bZ^k$ whose $i$-th coordinate is given by the intersection number of $g\in \pi_1(M)$ with the torus $T_i$.

Then $G$ corresponds to the fundamental group of an infinite (regular) cover $\widehat M$ of $M$, and any element $g\in G$ that represents a periodic orbit must by construction corresponds to an element of the fundamental group of one of the pieces of the JSJ decomposition of $\widehat M$, so must represent a prong singularity associated to one of the vertices of the original fat graph on $S$. Hence, the lift $\widehat\phi$ has only isolated prong singularities.

As in Example \ref{ex:non-bounded} one can easily verify that the action of $G$ on the orbit space is Anosov-like.
\end{proof}

\subsection{An action with a totally ideal quadrilateral}

We end this article by showing that Axiom \ref{Axiom_totallyideal} is necessary for Theorem \ref{thm:complement_of_fixbar} to hold; as well as its independence from the other axioms.

\begin{theorem} \label{thm:wandering_quadrilateral}
There exist actions of finitely generated groups on bifoliated planes that satisfy Axioms \ref{Axiom_A1}--\ref{Axiom_non-separated} but have wandering totally ideal quadrilaterals.  

In fact, if $M_1, M_2$ are compact 3-manifolds with non-$\R$ covered Anosov flows, and $G_i = \pi_1(M_i)$, then there exists a bifoliated plane $P$ and an action of $G_1 \ast G_2$ on $P$ satisfying \ref{Axiom_A1}--\ref{Axiom_non-separated} with wandering totally ideal quadrilaterals. 
\end{theorem}

The construction of $P$ involves blowing up copies of the orbit spaces of the flows and assembling them together so that the free product $G_1 \ast G_2$ acts on the resulting space with ping-pong dynamics. 
To make this precise, we introduce the following definition.  

\begin{definition} 
Let $(P, \cF^+, \cF^-)$ be a bifoliated plane and $S$ some collection of leaves of $\cF^+$.  We say that another plane $(\hat{P}, \hat{\cF}^+, \hat{\cF}^-)$ is a {\em blowup of $P$ along $S$ } if there exists a continuous map $h\colon \hat{P} \to P$ with the following properties 
\begin{itemize} 
\item $h$ is injective on the complement of $h^{-1}(\bigcup \{l \in S\})$
\item For each $l \in S$, $h^{-1}(l)$ is homeomorphic to $[0,1] \times \R$ via a homeomorphism sending leaves of $\hat{\cF}^+$ to the vertical foliation and leaves of $\hat{\cF}^-$ to the horizontal.
\item using the coordinates above, $h$ is constant on each $\hat{\cF}^-$ (horizontal) leaf $[0,1] \times \{x\}$ in each $h^{-1}(l)$.  
\end{itemize} 
\end{definition} 

We call a set $h^{-1}(l)$ the {\em blow-up of $l$}.  

\begin{lemma}[Equivariantly blowing up a leaf] \label{lem:blowup_exists}
Let $(P, \cF^+, \cF^-)$ be the orbit space of a pseudo-Anosov flow on a manifold $M$, with induced action of $G = \pi_1(M)$.  Let $l_0$ be a nonsingular leaf of $\cF^+$ that is not fixed by an element of $G$.  
Then there exists a blow-up $\hat{P}$ of $P$ along $S := G \cdot l_0$ and an action of $G$ on $\hat{P}$ by homeomorphisms that extends the action on $P$.  That is if $h\colon \hat{P} \to P$ is the projection from the definition of blow-up, we have 
$h \circ g = g \circ h$ for all $g \in G$.  
\end{lemma}

\begin{proof} 
First, blow up a nonsingular, non-periodic leaf $l_0$ of the stable foliation of an Anosov flow $\phi$ (see, e.g., Example 4.1 in \cite{Calegari_book}).  Lifting this to the universal cover $\wt{M}$, each lift $l$ of $l_0$ is replaced by a topological product $l \times [0,1]$.  The leaf $l$ is foliated by orbits of the lifted flow.  Extending this trivially over the product $l \times [0,1]$, one obtains a $1$-dimensional foliation on $\wt{M}$.  The 2-dimensional unstable foliation of $\phi$ can also be extended trivially over these product regions.   The fundamental group $G = \pi_1(M)$ acts naturally, permuting the blown-up leaves and preserving all of these foliations, and the action descends to an action by homeomorphisms on the quotient by leaves of the $1$-dimensional foliation.  This quotient space is a topological plane $\hat{P}$ with foliations $\hat{\cF}^\pm$ induced by the stable and unstable foliation, which (taking labels so that $\cF^+$ corresponds to stable) is a blow-up of the orbit space of $\phi$ along the leaf of the orbit space corresponding to $l$.  The fact that the action of $G$ on $\hat{P}$ is an extension of its action on the orbit space is immediate from the construction.  
\end{proof} 

We can modify this slightly so that the blown-up regions are compactified with leaves of $\cF^-$.  
\begin{construction}[Blow-up with boundary] \label{const:blow_up_boundary}
Let $\hat{P}$ be a blow-up along $G \cdot l_0$ as above.  The blow-up with boundary is a bifoliated space obtained as follows. For each $l \in G \cdot l_0$, we have that $h^{-1}(l) \cong \R \times [0,1]$. We would like to add a leaf of $\cF^-$ at $+\infty$ and at $-\infty$ in the $\R$-coordinate to each such region so the result is homeomorphic to a totally ideal quadrilateral.  
To do this, one can put a metric on $\hat{P}$ so that the preimages of leaves $l \in G l_0$ under $h$ (each homeomorphic to $[0,1] \times \R$) have the length of leaf $[0,1] \times \{x\}$ approaching $\infty$ as $x \to \pm \infty$, and then compactify this strip by adding two leaves homeomorphic to $\R \times \{+\infty\}$ and $\R \times \{-\infty\}$, one at either end. 
We call the result a {\em Blow-up of $P$ with boundary}.  This is no longer a topological plane, but is homeomorphic to a simply connected subset of the plane.  

Note also that the action of $G$ on $\hat{P}$ extends to the blow-up with boundary, permuting the ideal quadrilateral regions.  
\end{construction} 

Using this, we can construct examples of actions with (wandering) ideal quadrilaterals that satisfy Axioms \ref{Axiom_A1}, \ref{Axiom_prongs_are_fixed}, and \ref{Axiom_non-separated}.  We describe this construction as motivation and a warm-up for the construction in the proof of Theorem \ref{thm:wandering_quadrilateral}. 

Start with an Anosov-like action of a group $G$ on a bifoliated plane $(P, \cF^+, \cF^-)$ coming from the orbit space of a non $\bR$-covered Anosov flow, and fix a leaf $l_0$ with trivial stabilizer.  Let $\bar{P}$ be the blow-up with boundary along $G\cdot l_0$ as in  Construction \ref{const:blow_up_boundary}.  
Take another bifoliated plane $(P', \cF'^{+}, \cF'^{-})$, and cut $P'$ along a leaf $l^-$ of $\cF'^-$, obtaining two half-planes $H_+, H_-$.  For each $g \in G$, take copies $H_\pm(g)$ of $H_+$ and $H_-$
and glue them to the compactified ends of $h^{-1}(g l)$, equivariantly with respect to the action of $G$. 
The result is a bifoliated plane to which the action of $G$ extends and each compactified $h^{-1}(g l)$ is a totally ideal quadrilateral.  However, this action unfortunately will not satisfy all of the required axioms -- in particular, it will never satisfy Axiom \ref{Axiom_dense} since the set of $\cF^+$ leaves passing through one of these ideal quadrilaterals gives an open set of leaves which are fixed by no nontrivial element.  

Axiom \ref{Axiom_A1} does hold, however, thanks to the choice of $l$ with trivial stabilizer.  If $P'$ contains no singular points, axiom \ref{Axiom_prongs_are_fixed} will also be satisfied.  But \ref{Axiom_non-separated} is again problematic: if $l$ makes a perfect fit with another leaf $f^-$ in $P$, then $l^-$ and (the preimage under $h$ of) $f^-$ will be non-separated, but not fixed by any element, failing Axiom \ref{Axiom_non-separated}. (Note that, since there are no infinite product regions, as we started from a non-$\bR$-covered flow, a copy of $l^-$ can \emph{only} be non-separated with a leaf making a perfect fit with $l$.)
To avoid this, we use the following lemma, which follows from an argument of Fenley in \cite{Fenley16}.  

\begin{lemma} \label{lem:no_perfect_fit}
Let $(P,\cF^+,\cF^-)$ be the orbit space of a transitive pseudo-Anosov flow, equipped with the natural action of $G = \pi_1(M)$, and assume $P$ is nontrivial and non-skewed.  Then there exists leaves $l^\pm$ of $\cF^\pm$ with trivial stabilizer in $G$ and which do not make any perfect fits.  
\end{lemma}

\begin{proof} 
Let $o$ be a dense orbit in $M$, and let $p \in P$ be the image of a lift of $o$ to $\wt M$ projected to the orbit space. Since $o$ is dense, $G \cdot p$ is dense in $P$.  We claim that $\cF^\pm(p)$ cannot make any perfect fit.  

Suppose for contradiction that one of $\cF^\pm(p)$ makes a perfect fit with another leaf.  The proof of Proposition 5.5 in \cite{Fenley16} shows that $p$ is then accumulated by corners of lozenges.  We sketch this for completeness:  using the closing lemma one can find a periodic orbit $\delta$ which admits a lift to $\wt{M}$ that projects to a point $\wt \delta$ in the orbit space (notation here is borrowed from \cite{Fenley16}), close to $p$.  Recurrence (here given by density of $o$) in fact allows us to find such points arbitrarily close to  $p$.  This periodic orbit is associated to some $g \in G$ which fixes $\wt \delta$ and can be taken (again) to move $p$ an arbitrarily small amount.  Moreover, one can specify on which side of $\cF^s(p)$ the image $g \cF^s(p)$ lies. 
Considering the arrangement of leaves in the orbit space, as in the left side of Figure 9 of \cite{Fenley16}, (here $p_1$ is labeling the projection of $p$, and $g(p_2)$ its image under $g$) one concludes that $\wt \delta$ makes a perfect fit, and is the corner of a lozenge.   Thus, $p$ is approximated by corners of lozenges.  

Now by Lemma 2.30 of \cite{BFM}, the set of points in the orbit space of a transitive, non-$\bR$-covered Anosov flow that are corners of lozenges is nowhere dense.  This contradicts the density of $G \cdot p$ in $P$.  
\end{proof}

With this in hand, we can now prove the main result of this section.   The general idea is as follows: first, blow up an orbit $G_1 \cdot l_1$ in $P_1$ using the blow-up with boundary construction, and $G_2 \cdot l_2$ in $P_2$.  We then glue copies of these together along the blown-up quadrilaterals in a way so that the actions of $G_1$ and $G_2$ extend equivariantly, and play ``ping-pong" on the resulting space.  The proof below consists in making this idea precise.

\begin{proof}[Proof of Theorem \ref{thm:wandering_quadrilateral}] 
Let $M_1, M_2$ be 3-manifolds with transitive pseudo-Anosov flows, and let $G_i = \pi_1(M_i)$. 
Let $(P_i, \cF^+_i,\cF^-_i)$, for $i=1,2$ be the orbit spaces of these flows, equipped with the actions of $G_i$.  

Let $l_1 \in \cF^+(P_1)$ be a leaf with trivial stabilizer for the action of $G_1$, and 
$l_2 \in \cF^-(P_2)$ be a leaf with trivial stabilizer for the action of $G_2$.  In particular, these leaves are nonsingular and have no nonseparated leaves in their respective planes. Using Lemma \ref{lem:no_perfect_fit}, we can further assume that neither $l_1$ nor $l_2$ make a perfect fit with any other leaf.

\noindent \textbf{Step 1: taking many copies of blow-ups with boundary.}
Let $\bar{P}_1, \bar\cF^\pm_1$ and $\bar{P}_2, \bar\cF^\pm_2$ denote the blow-ups with boundary of $P_i$ along $G_i(l_i)$, respectively, as described in Construction \ref{const:blow_up_boundary}.  
Of particular importance to us is the fact that the $\bar{\cF}_1^-$ leaf space of $\bar{P}_1$ is identical to that of the original $\cF_1^-$ leaf space in $P_1$ itself.  Similarly, each $\bar{\cF}_2^+$ leaf of $\bar{P}_2$ is naturally identified with a leaf of $\cF_2^+$ in $P_2$.  

Consider the generating set $G_1 \cup G_2$ for the free product $G_1 \ast G_2$.  Reduced words of $G_1 \ast G_2$ with this generating set are written as strings of elements alternatively in $G_1$ and $G_2$.  
For each word $w$ whose initial element is in $G_2$, we take a copy $P_1(w)$ of $\bar{P}_1$.  For each word $v$ whose initial element is in $G_1$, we take a copy $P_2(v)$ of $\bar{P}_2$.  
We next assemble these copies together, along with two additional copies $P_1(e)$
and $P_2(e)$ for the identity element, to produce a bifoliated plane.  

\noindent \textbf{Step 2: gluing along ideal quadrilaterals.}
Let $Q_i \subset \bar{P}_i$, $i = 1,2$, be the ideal quadrilaterals corresponding to the compactified blow-ups of $l_i$.   Let $\Phi\colon Q_1 \to Q_2$ be a homeomorphism sending leaves of $\bar\cF^\pm_1$ to $\bar\cF^\pm_2$ (preserving $+$ and $-$).  
For an image $a Q_1$ of $Q_1$ under some element $a \in G_1$, let $\Phi_a$ be the gluing map $a Q_1 \to Q_2$ defined by $\Phi_a(x) = a \Phi a^{-1}(x)$.  Similarly for $b \in G_2$ let $\Phi_b \colon Q_1 \to b Q_2$ denote the homeomorphism whose inverse $\Phi_b^{-1} \colon bQ_2 \to Q_1$ is given by $\Phi_b^{-1}(x) = b \Phi^{-1} b^{-1}(x)$.  

Glue $P_1(e)$ to $P_2(e)$ using $\Phi$.  For each $a \in G_1$, glue $P_2(a)$ to $P_1(e)$ via $\Phi_a$.  And for each $b \in G_2$, glue $P_1(b)$ to $P_2(e)$ via $\Phi_b$.    

Generally, for each reduced word $w$ where $w = av$ with $a$ a nontrivial element of $G_1$ and $v$ starting with a nontrivial element of $G_2$, we glue $P_2(w)$ to $P_1(a^{-1} w)$ using $\Phi_a$.  Similarly, for each $v$ ending in some nontrivial $b \in G_2$, we glue $P_1(v)$ to $P_2(b^{-1}v)$ by $\Phi_b$.  In this way, every ideal quadrilateral is glued to exactly one other, so the plane extends on both sides of each.   Since $\bar\cF^+$ leaves are glued to $\bar\cF^+$-leaves (and $\bar\cF^-$ to $\bar\cF^-$), the result is a bifoliated space.  It is easy to see that this is simply connected and second countable, hence a plane.  Denote this bifoliated plane by $(P, \cF^+,\cF^-)$, making the obvious choice that the $\cF^+$ foliation is obtained from glued copies of $\bar\cF^+$ leaves.  

Furthermore, $P$ has a natural induced action of $G_1 \ast G_2$ which we now describe.
First, the embedded copies $P_1(e)$ and $P_2(e)$ in $P$ are stabilized by $G_1$ and $G_2$ respectively, and the restriction of the action of $G_i$ to $P_i(e)$ is the standard action.  
The equivariance of the gluing maps $\Phi_a$ and $\Phi_b$ extends this action naturally to the copies $P_1(b)$ and $P_2(a)$, for $a, b \in G_1$ and $G_2$ respectively, that are glued to $P_2$ and $P_1$; which are permuted by the actions of $G_1$ and $G_2$.  
Iteratively, a word 
$a_1 b_1 \ldots  a_k b_k$ takes $P_i(w)$ to $P_i( a_1 b_1 \ldots  a_k b_k w)$ (wherever defined), and 
$P_i(w)$ is stabilized by the conjugate of $G_i$ by $w$; with the action of $w G_i w^{-1}$ on $P_i(w)$ simply the conjugate action of $G_i$ on $P_i(e)$.  

\noindent \textbf{Step 3: the axioms are satisfied.}
It remains to check that the action satisfies Axioms \ref{Axiom_A1} -- \ref{Axiom_non-separated}.  
Let $l$ be a leaf of $\cF^\pm$ in $P$.  Then $l$ is contained in some $P_i(w)$.  
For concreteness, assume $w$ has initial letter $a \in G_1$ and so we are dealing with $P_2(w)$; the other case is analogous. 
Then the stabilizer of $P_2(a)$ is $w G_2 w^{-1}$.  If $l$ is fixed by $g \in G$, then $g \in w G_2 w^{-1}$ and so $w^{-1}gw$ fixes a leaf of $P_2(e)$.  Since the action of $G_2$ on $P_2$ is Anosov-like, and we blew-up a non-fixed leaf, the dynamics of $w^{-1}gw$ on $w^{-1}l$ (and thus of $g$ on $l$) are topologically expanding or contracting, with a unique fixed point. 

The argument for Axiom \ref{Axiom_prongs_are_fixed} (that prongs are fixed) is the same, using the fact that copies of the $\hat{P}_i$ in our glued construction are stabilized by conjugates of $G_i$ in the free product, and we did not introduce any new prongs in the gluing.  

To show Axiom \ref{Axiom_non-separated}, we will argue that any nonseparated leaves $l, l'$ in $P$ in fact lie in a single building-block plane $P_i(a)$ and are fixed by a nontrivial element of $G$.  

Consider first the case where $l$ is in $P_1(e)$ and is one of the ``boundary" leaves added to the blow-up of $P_1$.   Suppose that some sequence of leaves $l_n$ limits onto $l$ and another leaf $l'$.  (We will show that this leads to a contradiction).  Up to passing to a subsequence, we assume the $l_n$ limit from one side.  Thus, we may assume either all $l_n$ meet a single copy $P_2(a)$ glued to $P_1(e)$, or all lie in $P_1(e)$ and meet the ideal quadrilateral $a Q_1$.   Consider the latter case.  Here the projections of $l_n$ to $P_1$ each intersect $a l_1$.  Since $l_1$, and hence $a l_1$ makes no perfect fit (and $P_1$ contains no infinite product regions) it follows that the projections of the $l_n$ leave every compact set of $P_1$.  Thus, they limit exactly onto $l$ in $P_1$.  It follows that there are no leaves non-separated with $l$.   The case where $l_n$ limit from the other side is similar.  

By symmetry, and equivariance, the only case left to consider is where $l$ and $l'$ are both non-boundary leaves.  If they lie in different copies $P_i(a)$ and $P_i(b)$, then they are separated by a boundary of an ideal quadrilateral.  If they lie in the same plane $P_i(a)$, then they must be non-separated in $P_i$; since \ref{Axiom_non-separated} holds for the action of $G_i$ on the orbit space, they are both fixed by a conjugate of some $g \in G_i$.  

Finally, we show that the density axiom holds.  
Consider any open set $U$ in the leaf space of, say the $\cF^+$ foliation of $P$.  
Each leaf of $\cF^\pm$ meets at least one embedded $P_i(w)$, so $U$ contains an open set in the $\cF^+$ leaf space of some $P_i(w)$.  If this contains leaves that are not blown-up (i.e., the open set is not a subset of an ideal quadrilateral) then there is some element of $G$ fixing a leaf of $U$, since the stabilizer of $P_i(w)$ is a conjugate of $G_i$ with the induced blown-up action.   Otherwise, if the union of leaves in this open set is completely contained inside an ideal quadrilateral $Q$, then we must have that $i = 1$ since $\cF^+$ leaves that cross ideal quadrilaterals in the blow-up of $P_2$ always exit them.  Considering 
instead the copy of $P_2(v)$  glued to $P_i(w)$ along $Q$, we find as before an element of $G$ fixing a leaf. 

Thus, all of \ref{Axiom_A1} -- \ref{Axiom_non-separated} are satisfied by this action, as desired. 
\end{proof} 

 Note that we still do not know whether we can build examples with \emph{non-wandering} totally ideal quadrilateral. It would be interesting to be able to answer any of the following questions:
\begin{question}
Does there exists an action of a group $G$ satisfying Axioms \ref{Axiom_A1}--\ref{Axiom_non-separated} that admits a \emph{non-wandering} totally ideal quadrilateral?
If yes, can such an action be \emph{topologically transitive} (or equivalently with a unique Smale class, or satisfying also Axiom \ref{Axiom_fixed_points_dense})?
\end{question}

\bibliographystyle{amsalpha}
\bibliography{refs}
 
\end{document}